\documentstyle[12pt]{article}
\begin{document}

\input amssym.def

\input amssym

\newcommand{\qbinomial}{\left[\begin{array}{c}n+k\\n\end{array} \right]_{q}}
\newcommand{\eqref}[1]{(\ref{#1})}
\newcommand{\aP}{\tilde{P}}
\newcommand{\bz}{\mbox{$\Bbb Z$}}
\newcommand{\bq}{\mbox{$\Bbb Q$}}
\newcommand{\bc}{\mbox{$\Bbb C$}}
\newcommand{\bh}{\mbox{$\Bbb H$}}
\newcommand{\br}{\mbox{$\Bbb R$}}
\newcommand{\bp}{\mbox{$\Bbb P$}}
\newcommand{\bres}{\mbox{$\Bbb F$}}

\newcommand{\ra}{\mbox{$\rightarrow$}}
\newcommand{\lra}{\mbox{$\longrightarrow$}}
\newcommand{\proof}{\noindent {\it Proof: }}
\newcommand{\remark}{\noindent {\bf Remark: }}

\newcommand{\cA}{\mbox{${\cal A}$}}
\newcommand{\cB}{\mbox{${\cal B}$}}
\newcommand{\cC}{\mbox{${\cal C}$}}
\newcommand{\cD}{\mbox{${D}$}}  %%% removed cal to match dtil
\newcommand{\cE}{\mbox{${\cal E}$}}
\newcommand{\cF}{\mbox{${\cal F}$}}
\newcommand{\cG}{\mbox{${\cal G}$}}
\newcommand{\cH}{\mbox{${\cal H}$}}
\newcommand{\cI}{\mbox{${\cal I}$}}
\newcommand{\cJ}{\mbox{${\cal J}$}}
\newcommand{\cK}{\mbox{${\cal K}$}}
\newcommand{\cL}{\mbox{${\cal L}$}}
\newcommand{\cM}{\mbox{${\cal M}$}}
\newcommand{\cN}{\mbox{${\cal N}$}}
\newcommand{\cO}{\mbox{${\cal O}$}}
\newcommand{\cP}{\mbox{${\cal P}$}}
\newcommand{\cQ}{\mbox{${\cal Q}$}}
\newcommand{\cR}{\mbox{${\cal R}$}}
\newcommand{\cS}{\mbox{${\cal S}$}}
\newcommand{\cT}{\mbox{${\cal T}$}}
\newcommand{\cU}{\mbox{${\cal U}$}}
\newcommand{\cV}{\mbox{${\cal V}$}}
\newcommand{\cW}{\mbox{${\cal W}$}}
\newcommand{\cX}{\mbox{${\cal X}$}}
\newcommand{\cY}{\mbox{${\cal Y}$}}
\newcommand{\cZ}{\mbox{${\cal Z}$}}

\setlength{\unitlength}{1cm} \thicklines \newcommand{\ci}{\circle*{.13}}
\renewcommand{\mp}{\multiput}
\newcommand{\num}[1]{{\raisebox{-.5\unitlength}{\makebox(0,0)[b]{${#1}$}}}}
\newcommand{\mynum}[1]{{\raisebox{-.1\unitlength}{\makebox(-.4,0)[r]{$Y_{#1}$}}}}

\newcommand{\da}{\downarrow}
\newcommand{\el}{\ell}

\newcommand{\aut}{automorphism}
\newcommand{\clg}{compact Lie group}
\newcommand{\diffeo}{diffeomorphism}
\newcommand{\ems}{Eilenberg-Maclane}
\newcommand{\fg}{finitely-generated} 
\newcommand{\fd}{finite dimensional}

\newcommand{\homeo}{homeomorphism}
\newcommand{\ho}{homomorphism}
\newcommand{\homeq}{homotopy equivalence}
\newcommand{\hty}{homotopy type}
\newcommand{\iso}{isomorphism}
\newcommand{\les}{long exact sequence}
\newcommand{\lfp}{Lefschetz fixed-point theorem}
\newcommand{\mvs}{Mayer-Vietoris sequence}
\newcommand{\ses}{short exact sequence}
\newcommand{\nbhd}{neighborhood}
\newcommand{\pid}{principal ideal domain}
\newcommand{\sss}{spectral sequence}
\newcommand{\we}{weak equivalence}

\newcommand{\wrt}{with respect to}

\newcommand{\rn}{\mbox{$\br ^n$}}
\newcommand{\cn}{\mbox{$\bc ^n$}}

\newcommand{\gc}{\mbox{$G_{\bc}$}}

\newcommand{\phiplus}{\mbox{$\Phi ^+$}}
\newcommand{\phiminus}{\mbox{$\Phi ^-$}}
\newcommand{\phii}{\mbox{$\Phi _I$}}
\newcommand{\phij}{\mbox{$\Phi _J$}}
\newcommand{\phiiplus}{\mbox{$\Phi _I ^+$}}
\newcommand{\phiiminus}{\mbox{$\Phi _I ^-$}}
\newcommand{\phijplus}{\mbox{$\Phi _J ^+$}}
\newcommand{\phijminus}{\mbox{$\Phi _J ^-$}}
\newcommand{\phiirad}{\mbox{$\Phi _I ^{rad}$}}
\newcommand{\phijrad}{\mbox{$\Phi _J ^{rad}$}}
\newcommand{\phiipar}{\mbox{$\Phi _I ^{par}$}}
\newcommand{\phijpar}{\mbox{$\Phi _J ^{par}$}}
\newcommand{\ujrad}{\mbox{$U _J ^{rad}$}}
\newcommand{\uirad}{\mbox{$U _I ^{rad}$}}

\newcommand{\lieuirad}{u_I ^{rad}}
\newcommand{\lieujrad}{u_J ^{rad}}
\newcommand{\lieualpha}{u_{\alpha}}
\newcommand{\lieqw}{q_w}
\newcommand{\liepi}{p_I}
\newcommand{\lieuw}{u_w}
\newcommand{\lieuwp}{u_w ^\prime}
\newcommand{\liegc}{\frak{g}_{\bc}}

\newcommand{\gctil}{\mbox{$\tilde{G} _{\bc}$}}
\newcommand{\gctilp}{\mbox{$\tilde{G} _{\bc}/P$}}
\newcommand{\gctilad}{\mbox{$\tilde{G} _{\bc}^{ad}$}}
\newcommand{\gctiladp}{\mbox{$\tilde{G} _{\bc} ^{ad}/P^{ad}$}}
\newcommand{\btil}{\mbox{$\tilde{B}$}}
\newcommand{\dtil}{\mbox{$\tilde{D}$}}
\newcommand{\wtil}{\mbox{$\tilde{W}$}}
\newcommand{\wtils}{\mbox{$\tilde{W}^S$}}
\newcommand{\minreps}{\wtils}
\newcommand{\ntil}{\mbox{$\tilde{N}$}}
\newcommand{\stil}{\mbox{$\tilde{S}$}}
\newcommand{\liegctil}{\mbox{$\tilde{\frak{g}} _{\bc}$}}
\newcommand{\phitil}{\mbox{$\tilde{\Phi}$}}
\newcommand{\phitilplus}{\mbox{$\tilde{\Phi} ^+$}}
\newcommand{\phitilminus}{\mbox{$\tilde{\Phi} ^-$}}
\newcommand{\coroot}{\mbox{$Q^\vee$}}
\newcommand{\antidom}{\mbox{$Q^\vee _-$}}
\newcommand{\chamber}{\overline{\cC}}

\newcommand{\cpo}{closed parabolic orbit}
\newcommand{\clr}{coroot lattice representative}  
\newcommand{\svar}{Schubert variety}
\newcommand{\svars}{Schubert varieties}
\newcommand{\pd}{Poincar\' e duality}
\newcommand{\pds}{Poincar\' e duality space}
\newcommand{\ppoly}{Poincar\' e polynomial}
\newcommand{\pser}{Poincar\' e series}
\newcommand{\fnk}{\mbox{$X_{n,k}$}}
\newcommand{\fnkp}{\mbox{$X_{n,k} ^\flat$}}
\newcommand{\ppc}{positive pair condition}
\newcommand{\npc}{negative pair condition}
\newcommand{\osc}{opposite sign condition}
\newcommand{\und}{\underline{0}}
\newcommand{\nope}{$\lambda \notin \cQ ^\vee$}
\newcommand{\onemin}{1- \alpha _0 (\lambda)} 
\newcommand{\owt}{(otherwise $\lambda$ is overweight)} 
\newcommand{\adp}{admissible palindromic}
\newcommand{\fork}{(otherwise $\lambda$ covers a fork)}

\newcommand{\gsm}{graph-splitting move} 
\newcommand{\aslam}{\alpha _s
(\lambda)} 
\newcommand{\azlam}{\alpha _0 (\lambda)}
\newcommand{\atlam}{\alpha _t (\lambda)} 
\newcommand{\aulam}{\alpha _u
(\lambda)} 
\newcommand{\astlam}{\alpha _s (\lambda) + \alpha _t
(\lambda)} 
\newcommand{\alam}{\alpha (\lambda)} 
\newcommand{\blam}{\beta
(\lambda)} 
\newcommand{\aplam}{\alpha ^\prime (\lambda)}
\newcommand{\bpp}{$\beta$-positive pair}
\newcommand{\bnp}{$\beta$-negative pair}
\newcommand{\lam}{\lambda}
\newcommand{\circular}{spiral}
\newcommand{\loops}{\mbox{$\cL _G$}}
\newcommand{\apal}{admissible palindromic}
\newcommand{\spd}{satisfies Poincar\'e duality}
\newcommand{\xlam}{\mbox{$X_\lambda$}}

%\numberwithin{equation}{section}

\newtheorem{theorem}{Theorem}[section]

\newtheorem{proposition}[theorem]{Proposition}

\newtheorem{lemma}[theorem]{Lemma}

\newtheorem{conjecture}[theorem]{Conjecture}

\newtheorem{example}[theorem]{Example}

\newtheorem{corollary}[theorem]{Corollary}

\newtheorem{definition}[theorem]{Definition}

\title{Smooth and palindromic Schubert varieties\\
 in affine Grassmannians}

\author{Sara C. Billey\thanks{The first author was supported by the
Royalty Research Fund.} and Stephen A. Mitchell\thanks{The second
author was supported by the National Science Foundation}
}

\date{December 14, 2007}

\maketitle

\setcounter{tocdepth}{2}  %  could be 1 to get only sections
\begin{small}
\tableofcontents
\end{small}

%\noindent {\it Contents:} 

%\bigskip

%\S 1 Introduction

%\S 2 Notation

%\S 3 The coroot lattice

%\S 4 Parabolic orbits

%\S 5 The Palindromy Game I: Weak order and the coroot lattice

%\S 6 Poincar\'e duality and the affine Chevalley formula

%\S 7 Chains

%\S 8 Proof of the Smoothness Theorem~\ref{smooth} 

%\S 9 The Palindromy Game II: Bruhat order and the coroot lattice

%\S 10 Proof of the Palindromy Theorem~\ref{pal}

%\S 11 The spiral varieties in type $A$ 

%\S 12 Hasse diagrams and Dynkin diagrams

\section{Introduction}\label{s:intro}

Let $G$ be a simply-connected simple compact Lie group, with
complexification \gc . The affine Grassmannian $\cL_G$ is a projective
ind-variety, homotopy-equivalent to the loop space $\Omega G$ and
closely analogous to a maximal flag variety of \gc . It has a Schubert
cell decomposition 

$$\cL _G =\coprod _{\lambda \in \cQ ^\vee} e_\lambda,$$

\noindent where $\cQ ^\vee$ is the coroot lattice. The closure $X_\lambda$ of
$e_\lambda$ is a \fd\ projective variety that we call an {\it affine
  Schubert variety.} In this paper we completely determine the smooth
and palindromic affine Schubert varieties. 

In any ordinary flag variety there is one obvious class of smooth \svars
: the closed orbits of the standard parabolic subgroups. In fact each
such parabolic subgroup has a unique closed orbit $\cO _0$, namely the
orbit of the basepoint, and $\cO _0$ is smooth because it is
homogeneous. Indeed $\cO _0$ is itself a flag variety of the Levi factor
of the parabolic. A similar construction works in the affine setting,
provided that we consider only proper parabolic subgroups. Then every
such closed parabolic orbit is a smooth \svar\ in $\cL _G$.

\begin{theorem} \label{smooth} \marginpar{}

Let $X_\lambda$ be an affine \svar . Then the following are equivalent: 

\bigskip

a) $X_\lambda$ is smooth; 

b) $X_\lambda$ satisfies \pd\ integrally;

c) $X_\lambda$ is a \cpo . 

\end{theorem}

Of course (a) $\Rightarrow $(b) and (c) $\Rightarrow$ (a) are
immediate; the significant point is (b) $\Rightarrow$ (c). 

\begin{corollary}  \marginpar{}

There are only finitely many smooth \svars\ in a fixed $\cL _G$. 

\end{corollary}
 
In fact, it is easy to see that the non-trivial \cpo s are in bijective
correspondence with connected subgraphs of the affine Dynkin diagram
containing the special node $s_0$ (Proposition~\ref{cpograph}).

 A node of the Dynkin diagram is {\it minuscule} if there is an \aut\ of
the affine diagram carrying it to the special node $s_0$. A
\textit{minuscule flag variety} is a flag variety whose parabolic
isotropy group is the maximal parabolic obtained by deleting a minuscule
node. ({\it Warning:} Our ``minuscule'' flag varieties would be called
``co-minuscule'' in some sources.)  Similarly, if a variety $X$ is
isomorphic to $G'/P$ for some reductive algebraic group $G'$ and maximal
parabolic subgroup $P$, we will say $X$ is a \textit{maximal flag
variety}.

It turns out that every \cpo\ in $\cL_{G}$ is a minuscule flag variety of some
simple algebraic group, and that every minuscule flag variety occurs
as a \cpo\ in some affine Grassmannian (see
Proposition~\ref{mincpo}). Hence we obtain as a by-product:

\begin{corollary} \label{minflag} \marginpar{}

Let $X$ be a \svar\ in a minuscule flag variety. Then the following are
equivalent:

\bigskip  

a) $X$ is smooth;

b) $X$ satisfies \pd\ integrally;

c) $X$ is a \cpo .      

\end{corollary}

%%%todo:  need to cite Lakshmibai-Weyman here -- check ref exactly. 
This corollary generalizes the fact that in the type $A$ Grassmannian, 
$G_k \bc ^{n}$,  the smooth Schubert varieties are the ones that are
themselves Grassmannians \cite[Cor. 9.3.3]{billeylak}.  However, for general flag varieties $G/Q$,  
even maximal ones, it isn't true that every smooth \svar\ is a \cpo .
In type $C_n$, for example, the maximal flag variety obtained by
deleting the node $s_1$ of the Dynkin diagram is a $\bp ^{2n-1}$. All
of its Schubert varieties $\bp ^k$ are smooth, but for $n\leq k <
2n-1$ they are not \cpo s.

\bigskip 

Suppose that $\lambda$  is anti-dominant (i.e., for every positive root
$\alpha$, $\alpha (\lam ) \leq 0$) and non-trivial. Then $X_\lambda$
cannot be a \cpo , because it is not invariant under the action of
$s_0$. Hence Theorem~\ref{smooth} implies that $X_\lambda$ is always
singular. This statement is already known; a theorem of Evens-Mirkovic
(\cite{em}; see also \cite{mov}) shows that the smooth locus of
$X_\lambda$ is precisely the $\aP$-orbit of $\lambda$ if $\lambda $ is
antidominant. Hence $X_\lambda$ is smooth if and only if it is the
unique closed $\aP$-orbit, namely the basepoint. More generally, for any
$\lambda \in Q^\vee$ the stabilizer group of $X_\lambda$ is a
parabolic subgroup $P_{I_\lambda}$. This suggests:

\bigskip

\noindent {\bf Conjecture:} The smooth locus of $X_\lambda$ is
$P_{I_\lambda} \lambda$. 

\bigskip

A \svar\ $X_\lambda$ is {\it palindromic} if it has palindromic
\ppoly\ $|X_\lambda| (t) = 1 + a_1 t + \ldots + a_{d-1} t^{{d-1}}
+t^{d}$, where $d$ is the complex dimension of $X_ \lambda$. In other
words, $X_\lambda$ satisfies \pd\ additively: $a_k =a_{d-k}$ for all
$0<k<d$. Here $t$ is assigned real dimension 2. We say that
$X_\lambda$ is a {\it chain} if $|X_\lambda| (t)=1 + t + t^2 +\ldots +
t^d$. In type $A_n$, there are two infinite families of palindromic
\svars\ (one family if $n=1$) $X_{n,k}$, $X_{n,k} ^\prime$ of
dimension $kn$, introduced by the second author in \cite{mfilt}. We
call these {\it spiral varieties}, for reasons to be explained in
Section~\ref{s:spiral.varieties}. The two families are conjugate under
the automorphism of the affine Dynkin diagram fixing the special node
$s_0$.

%\begin{theorem} \label{pal} \marginpar{}                                           
%$X_{\lam}$ is palindromic if and only if one of the following conditions
%  holds:
%  \bigskip 

%a) $X_\lambda $ is a \cpo\ (equivalently, $X_\lambda$ is smooth) 

%b) $X_\lambda$ is a chain

%c) $\Phi$ has type $A$ and $\lam$ is spiral

%d)$\Phi$ has type $B_3$ and $\lam =(3,0,-1)$.     

%\end{theorem}       

\begin{theorem} \label{pal} \marginpar{}                                           
$X_{\lam}$ is palindromic if and only if one of the following conditions
  holds:
\begin{enumerate}
\item [a)] $X_\lambda $ is a \cpo\ (in particular, $X_\lambda$ is smooth). 
\item [b)]$X_\lambda$ is a chain.
\item [c)]$G$ has type $A_{n}$ and $X_{\lam}$ is spiral.
\item [d)]$G$ has type $B_3$ and $\lam =(3,0,-1)$.     
\end{enumerate}
\end{theorem}       

There is some overlap in conditions a)-d). For example, a chain is a \cpo\
if and only if it is a projective space, and these occur frequently. In
type $A_n$ the two spiral classes of minimal dimension $n$ are
projective spaces, but the others are neither smooth nor chains. The
peculiar exception in type $B_3$ is a singular 9-dimensional variety
with \ppoly\ $1+t+t^2 +2t^3 +2 t^4 +2 t^5 +2 t^6 +t^7 +t^8 +t^9$.

\begin{corollary}  \marginpar{}

If $G$ is not of type $A$, there are only finitely many palindromic
\svars\ in $\cL _G$.

\end{corollary}

It is easy to see that in all types there are only finitely many chains
(Corollary~\ref{finitechain}), so the corollary follows immediately from
the theorem.

By a special case of a theorem of Carrell and
Peterson (\cite{cp}; see also \cite{kumar}, XII, \S 2), an affine \svar\
is palindromic if and only if it is rationally smooth. This yields the
corollary: 
                               
\begin{corollary}  \marginpar{} Let $X_\lambda$ be an affine
  \svar . Then the following are equivalent:

a) $X_\lambda$ is palindromic;

b) $X_\lambda$ is rationally smooth; 

c) $X_\lambda$ satisfies rational \pd ; 

d) $X_\lambda$ satisfies one of the conditions (a)-(d) of
Theorem~\ref{pal}. 

\end{corollary}

Since we have enumerated all the palindromic \svars , the corollary can
be proved {\it ad hoc} by checking that the singular ones satisfy
rational \pd . The equivalence of rational \pd\ and rational smoothness
is well-known \cite{mcrory}. 

In the simply-laced case (this excludes affine type $A_1$, which should
not be regarded as simply-laced) every chain is a projective space
(Corollary~\ref{chainproj}), and so in particular is smooth. Hence:

\begin{corollary} \label{c:simply-laced} \marginpar{}

In types $D$ and $E$, an affine \svar\ is smooth if and only if it is
palindromic. In all other types there are singular palindromics. 

\end{corollary}

This contrasts with an unpublished theorem of Dale Peterson, which
asserts that for ordinary \svars\ the corollary holds in all
simply-laced types $ADE$.  Combining Peterson's result with
Corollary~\ref{c:simply-laced}, we get evidence for the following
conjecture. 

\begin{conjecture}\label{conj:simply-laced}
%% todo: check Kumar's wording for these Kacs-Moody groups.
In any flag manifold $\cG /\cQ$ of affine or classical type, smoothness is
equivalent to rational smoothness for all \textit{non-cyclic simply
laced} types.
\end{conjecture}

Theorem~\ref{pal} gives a second proof of Theorem~\ref{smooth}: Having
listed all the palindromic \svars , we need only run through the list
and show that only the \cpo s satisfy \pd\ integrally. 

In a second forthcoming article \cite{bm2}, the authors consider an
alternative approach to the proof of Theorem 1.4.  In this work, we
associate a natural family of bounded partitions to each element in the
coroot lattice in such a way that the relations in Young's lattice on
partitions imply relations in Bruhat order on coroot lattice elements.
These relations are sufficient to differentiate all palindromic affine
Schubert varieties from the non-palindromic ones.
 
The spiral varieties in type $A$ have a number of interesting
properties (the first three are proved in \cite{mfilt}). 

\bigskip

(1) $H_* X_{n,k}$ realizes the ``degree filtration'' on $H_* \Omega
    SU(n+1)$;

(2) $X_{n,k}$ is the variety of $k$-dimensional submodules in a free
    module of rank $n+1$ over the truncated polynomial ring $\bc
    [z]/z^k$; 

(3) $H^* X_{n,k}$ and $H^* G_k \bc ^{n+k}$ are isomorphic as graded
    abelian groups, but not as rings unless $k=1$; 

(4) (Cohen-Lupercio-Segal \cite{cls}) $X_{n,k}$ is homotopy-equivalent
    to $Hol _k (\bp ^1, G_{n+1} \bc ^\infty)$, the space of holomorphic
    maps of degree $k$. 

\bigskip

In item (3) the ring structures are quite different for $k>1$. The
Bruhat orders of $X_{n,k}$ and $G_k \bc ^{n+k}$ are also different, and
in fact the Bruhat order associated to $X_{n,k}$ is not self-dual in
general.

\bigskip

\noindent {\it Outline of the proofs:} The proof of
Theorem~\ref{smooth} begins by considering some elementary
obstructions to palindromy. We call this the ``palindromy game''. The
point is simply that $\cL _G$ has only one cell of each of the first
few dimensions, and hence a palindromic $X_\lambda$ can't have too
many cells near the top. This already narrows down the possibilities
considerably. In particular, in any affine Grassmannian there is a
unique 2-cell, and hence a palindromic $X_\lambda$ of complex
dimension $d$ can only have one ($2d-2$)-cell. If $X_\lambda$ satisfies
\pd , then multiplication by the generator of $H^2 X_\lambda$ induces
an \iso\ $H^{2d-2} X_\lambda \stackrel{\cong}{\lra} H^{2d}
X_\lambda$. It is known that the classical formula of Chevalley for
this cup product generalizes to the affine case (indeed to any
Kac-Moody flag variety), putting further severe restrictions on which
$\lambda$ can occur. Along the way we also classify the chains
together with their cup product structure.  A type-by-type analysis
then completes the proof.

The proof of Theorem~\ref{smooth} uses only the weak order on the coroot
lattice; in other words, it only uses descents of the form $\lam
\downarrow s\lam$ with $s$ one of the Coxeter generators of the affine Weyl
group. Theorem~\ref{pal}, on the other hand, requires a more elaborate
version of the palindromy game incorporating the full Bruhat order; in
other words, it requires descents of the form $\lam \downarrow r\lam$ in which
$r$ is an affine reflection associated to a non-simple root. 

\section{Notation}\label{s:notation}

We follow the notation from \cite{bourbaki} whenever possible.

$G$: simple, simply-connected \clg\ of rank $n$ 

$T$: maximal torus, with Lie algebra $\frak{t}$

$W$: Weyl group

$S = \{s_{1},s_{2},\ldots , s_{n} \}$: set of Coxeter generators for $W$ 

$\Phi, \Phi ^+$: root system, positive roots

$\alpha _s$, $s \in S$: simple positive roots.  If $s=s_{i}$, we also denote $\alpha_{s}$ by $ \alpha_{i}$.

$m_s (\alpha)$, $\alpha \in \Phi $: $\alpha =\sum _{s \in S} m_s(
\alpha)\alpha _s$  

$\alpha _0$: highest root; set $m_s =m_s (\alpha _0)$

$D$: Dynkin diagram with $S$ as set of nodes

$\cQ ^\vee, \cP ^\vee$: Coroot lattice, coweight lattice

\gc , $T_{\bc}$: complexification of $G,T$. 

$B$, $B^-$: Borel subgroup containing $T_{\bc}$, opposite Borel subgroup. 

\bigskip

\gctil : $\gc (\bc [z, z^{-1}])$, or regular maps $\bc ^\times \lra \gc$

$\aP$: $\gc (\bc [z]) \subset \gctil$, or regular maps $\bc \lra \gc$    

$\cL _G$: $\gctil / \aP$, the affine Grassmannian

\btil : $\{f \in \aP: f(0) \in B^-\}$

\wtil : affine Weyl group

\stil : $S \cup \{s_0\}$, the Coxeter generators for \wtil\

\dtil : affine Dynkin diagram with \stil\ as set of nodes

\phitil : affine root system $\bz \times \Phi$

\wtils : set of minimal length representatives for $\wtil /W$

$\el , \el ^S$: length function on \wtil , length function relative to
$S$

\bigskip

\noindent {\it Bruhat coverings.}
If $\sigma , \tau \in \wtil /W$, and $r$ is an affine reflection, we
write $\sigma \downarrow r \sigma$ if $\el ^S (r\sigma) =\el ^S (\sigma)
-1$. Thus $\sigma$ covers $\tau$ in the Bruhat order. Alternatively,
we write $\sigma  \uparrow r \sigma$. If $r \in \stil$ we call this a
{\it left descent}. The partial order generated by the left descents
is the {\it left weak order} or just \textit{weak order} for short.
Note the right weak order isn't useful on $\wtil /W$ since every
non-trivial element has only $s_{0}$ as a right descent.

\bigskip

\noindent {\it Minuscule nodes}: We call a node $s$ of \cD\ {\it
minuscule} if it satisfies the equivalent conditions: (i) There is an
\aut\ of \dtil\ carrying $s_0$ to $s$; (ii) $m_s =1$, where $m_s$ is
the coefficient of $\alpha _s$ in the highest root $\alpha _0$. We
call a flag variety of \gc\ {\it minuscule} if its parabolic isotropy
group is the maximal parabolic obtained by deleting a minuscule node.
Note, the minuscule fundamental coweights form a set of distinct
representatives for $\cP ^\vee /\cQ ^\vee$.  We caution the reader
that these coweights would be called ``co-minuscule'' in some sources
such as \cite{billeylak}.
However, our \textit{minuscule nodes} correspond with their
fundamental cominuscule coweights.

\bigskip

\noindent {\it Long and short nodes}: In the simply-laced case we
regard all roots as long.  A node $s$ of the Dynkin diagram is
regarded as long/short according as the corresponding simple root
$\alpha _s$ is long/short.

\bigskip

\noindent {\it Poincar\'e series}: If $A$ is a suitable graded
object---a ranked poset, a graded abelian group, etc.---we write $|A|
(t)$ for the \pser\ of $A$. The spaces considered in this paper
invariably have their homology groups concentrated in even dimensions,
and as a slight variant of this notation we write $|X| (t) =\sum _i
a_i t^i$, where $a_i =rank \, H_{2i} X$.  Similarly, the CW-complexes
$X$ we consider have only even-dimensional cells, and it will be
convenient to call the $2k$-skeleton of $X$ the {\it complex
k-skeleton}.  Here the \textit{k-skeleton} of $X$ is the union of all
cells of dimension up to and including $k$.

\section{The coroot lattice} 

In this section we set down some basic facts and notation concerning the
coroot lattice. 

\subsection{The coroot lattice and the affine Weyl group}

The affine Weyl group \wtil\ is the group of affine transformations of
$\frak{t}$ generated by all reflections across hyperplanes $\alpha =k$,
where $\alpha \in \Phi$, $k \in \bz$. It fits into a split extension 

$$ \cQ ^\vee \lra \wtil \lra W.$$

\noindent Hence there are bijections 

$$\cQ ^\vee \stackrel{\cong}{\lra} \wtil /W
\stackrel{\cong}{\lra} \minreps.
$$ 
mapping $\lambda \in \cQ $ to $w_{\lambda } \in \minreps$ if $\lambda$
and $w_{\lambda}$ are in same coset of $\wtil /W$. Thus we have two
distinct canonical sets of coset representatives for $\wtil /W$.  To
see how the two sets differ, note that $\el(\lambda )$ need not
equal $\el(w_{\lambda})$ as elements in $\wtil$.  Therefore, one can
define a second length function on $\cQ ^\vee$, namely 

$$\el ^S (\lambda) =\mathrm{min} \{\el (\lambda w) : w \in W \} =
\el(w_{\lambda }).$$

Both length functions can be computed using the formulas of \cite{im}:

\begin{equation}\label{e:length.fn}
\el (\lambda ) =\sum _{\alpha \in \Phi ^+} |\alpha (\lambda)|,
  \hspace{1in} \el ^S (\lambda) =\el (\lambda) -q(\lambda),
\end{equation}

\noindent where                                                                 
$$q (\lambda) =|\{ \alpha \in \Phi ^+ : \alpha (\lambda ) >0\}|.$$

The equivalent length generating functions $\displaystyle |\minreps|(t)=\sum_{w\in
\minreps} t^{l(w)}$ and $\displaystyle |\cQ ^\vee|(t)=\sum_{\lambda \in
\cQ ^\vee} t^{l^{S}(\lambda)}$ can be obtained from the following
beautiful formula due to Bott.

\begin{theorem}\label{t:Bott}\cite{Bott-56}
Let $e_{1},e_{2},\dots, e_{n}$ be the exponents of $W$.   Then 
\[
|\cL _G| (t)  = \prod_{i=1}^{n} \frac{1}{(1-t^{e_{i}})}.
\]
\end{theorem}

Now recall that $\lambda \in \cQ ^\vee$ is {\it dominant} (resp. {\it
anti-dominant}) if $\alpha (\lambda) \geq 0$ (resp. $\alpha (\lambda)
\leq 0$) for all $\alpha \in \Phi ^+$.  It follows that $\lambda =
w_{\lambda}\in \wtils$ if and only if $\lambda$ is anti-dominant.

The coroot lattice also inherits a Bruhat order and a left weak order from
$\wtil /W$. For $s \in S$ and $\lambda \in \cQ ^\vee$ we have 

\begin{equation}\label{e:simple.reflections}
\begin{array}{rcl}
 \lambda \downarrow s \lam & \iff & \aslam <0
\\
  \lambda \uparrow s \lam & \iff &  \aslam >0   
\\
  \lambda = s \lam  & \iff & \aslam =0   
\end{array}
\end{equation}
If $s=s_0$, the same three conditions hold with $\aslam$
replaced by $1-\azlam$. 

We view the set of double cosets $W \backslash \wtil /W$ asymmetrically,
regarding it as the orbit set of the left $W$ action on $\wtil /W$. Note
that $\lambda \in \cQ ^\vee$ is dominant (resp. anti-dominant) if and
only if it is the unique minimal (resp. maximal) element of its left
$W$-orbit in $\wtil /W$.

We will almost always denote elements $\lam \in \cQ ^ \vee$ using the
expansion of $\lam$ in terms of the fundamental coweights $\omega _s
^\vee$, $s \in S$, or equivalently as $\bz$-valued functions on \cD
. Again, we freely interchange the notation $\omega _s^\vee$ and
$\omega _i ^\vee$ if $s=s_{i}$.  In a given type $\Phi$ the elements
of $S$ are ordered as on page \pageref{f:dynkin} following
\cite{bourbaki}; note the slightly odd ordering there in type $E$, in
which the ``off-line'' node is labelled as $s_2$. If $\Phi$ has rank
$n$ and $\lambda = \sum a_{i} \omega _i ^\vee$ then we write $\lam
=(a_1,\ldots ,a_n)$, where $a_i =\alpha _i (\lam )$. We use the symbol
$\und$ to denote a sequence of zeros whose length is irrelevant or
determined by the context.  Occasionally, however, we use what we call
``standard notation'', meaning the customary explicit representation
of the root systems in some $\br ^n$ as in \cite{bourbaki}; we write
$e_i$ for the standard basis where Bourbaki writes $\epsilon _i$.

\bigskip

\noindent {\bf Remark:} We think of $\cQ ^\vee$ in several different
ways: (1) as a lattice in $\frak{t}$; (2) as a group of translations
acting on $\frak{t}$; (3) as a set of coset representatives for $\wtil
/W$; (4) as the subgroup $Hom \, (S^1 ,T) \subset \gctil$. This last
identification uses the fact that $G$ is simply-connected, so that the
coroot lattice and the integral lattice $Ker \, (exp: \frak{t} \lra T)
=Hom \, (S^1 ,T)$ coincide. It should be clear from the context which of
these interpretations is intended.

\subsection{Comparison with the coweight lattice} \label{sub:comparision}

The coroot lattice is a subgroup of finite index in the coweight lattice

$$\cP ^\vee =\{ v \in \frak{t} : \alpha (v) \in \bz \, \forall \alpha \in
\Phi \}.$$ 

\noindent The fundamental coweights $\omega _s ^\vee$, $s \in S$, are
defined by $\alpha _t (\omega _s ^\vee) =\delta _{st}$.  Below we
summarize criteria to determine if an element in $\cP ^{\vee}$ is
actually an element of $\cQ ^\vee$.

We write $\pi _1 \Phi $ for $\cP ^\vee /\cQ ^\vee$, since the latter
depends only on $\Phi$ and is the fundamental group of the adjoint form
of $G$. The minuscule fundamental coweights $\omega ^\vee _s$ form a
complete set of representatives for the non-trivial cosets in $\pi _1
\Phi$. It is possible, however, for non-minuscule fundamental coweights
to represent non-trivial elements of $\pi _1 \Phi$. For the convenience
of the reader, and because we need to know all such non-trivial
fundamental coweights and the relations between them, we will describe
$\pi _1 \Phi$ type by type. The reference is \cite{bourbaki}, where some
of the data is left implicit in the description of the fundamental
weights. Since the data in \cite{bourbaki} is in terms of weights rather
than coweights, some translation is necessary; in particular, the
weights in type $B$ are the coweights in type $C$ and vice-versa. For
simplicity we identify the fundamental coweights with the nodes of \cD ,
and write $s \sim t$ if $s=t$ in $\pi _1 \Phi$. If $\delta \in \cP
^\vee$, we write $\delta =(a_1,\ldots ,a_n)$ if $\delta =\sum a_i \omega _i
^\vee$.

\bigskip                                                    

$A_n$: We have $\pi _1 \Phi =\bz /(n+1)$, with the elements of
$\{\omega_{i}^{\vee} : 1\leq i \leq n \}$ representing the distinct
non-trivial classes.  Equivalently, the nodes of $D$ represent the
non-trivial classes.  Hence, $(a_1,\ldots ,a_n) \in \cQ ^\vee$
$\Leftrightarrow$ $\sum ia_i \equiv 0\, \mathrm{mod} \ (n+1).$
                                                                
\bigskip                                                                        

$B_n$: Here $\pi _1 \Phi =\bz /2$; the non-trivial nodes are the odd
nodes and these are all identified.  Hence, $(a_1,\ldots ,a_n) \in \cQ
^\vee$ $\Leftrightarrow$ $\sum a_{odd}  \equiv 0 \   \mathrm{mod}  \ 2$.

\bigskip $C_n$: Here $\pi _1 \Phi =\bz /2$; $s_n$ is the only
non-trivial node.  Hence, $(a_1,\ldots ,a_n) \in \cQ ^\vee$
$\Leftrightarrow$ $a_n \equiv 0 \ \mathrm{mod} \ 2$.

\bigskip                  

$D_n$: Here we have                                                                                                                                       
\[\pi _1 \Phi \cong    \left\{ \begin{array}{ll}
\bz /2 \times \bz /2 & \mbox{if $n$ even}\\ 
\bz /4 & \mbox{if $n$ odd}          
\end{array}
\right. \]

\noindent The non-trivial classes are represented by $s_1, s_{n-1},
s_n$. If $k<n-1$, then $s_k$ is non-trivial if and only if $k$ is odd,
in which case $s_k \sim s_1$. Now let $\Sigma _{odd}$ denote the sum
of the $a_i$'s with $i$ odd, $i <n-1$.  If $n$ is odd, $(a_1,\ldots
,a_n) \in \cQ ^\vee$ $\Leftrightarrow$ $a_{n-1} -a_n +2 \Sigma _{odd}
 \equiv 0 \   \mathrm{mod}  \  4$. If $n$ is even, $(a_1,\ldots ,a_n) \in \cQ ^\vee$
$\Leftrightarrow$ $\Sigma _{odd} +a_{n-1}  \equiv 0 \equiv \Sigma _{odd} + a_n \ 
 \mathrm{mod}  \  2$.

\bigskip

$E_6$: $\pi _1 \Phi =\bz /3$, with $s_1$ and $s_6$ representing the
non-trivial nodes. The other non-trivial nodes are $s_3 \sim s_6$ and
$s_5 \sim s_1$.  Hence, $(a_1,\ldots ,a_6) \in \cQ ^\vee$ $\Leftrightarrow$
$a_1 -a_3 +a_5 -a_6  \equiv 0 \   \mathrm{mod}  \  3$.

\bigskip $E_7$: $\pi _1 \Phi =\bz /2$, with $s_2, s_5, s_7$ the
non-trivial nodes.  Hence, $(a_1,\ldots ,a_7) \in \cQ ^\vee$
$\Leftrightarrow$ $a_2 + a_5 +a_7 \equiv 0 \ \mathrm{mod} \ 2$.

\bigskip                        

$E_8, F_4, G_2$: $\pi _1 \Phi$ is trivial, so $\cQ^{\vee} = \cP^{\vee}$.

\section{Parabolic orbits}\label{s:cpos}

\subsection{Closed parabolic orbits}

Let $I \subset \stil$ be a proper subset that contains $s_0$. The
corresponding parabolic subgroup $P_I \subset \gctil$ is generated by
\btil\ and the simple reflections in $I$ lifted to \gctil. Note that
$P_{I}$ has a unique closed orbit in $\cL _G$, namely

$$Y_I =P_I \aP/ \aP =P_I /P_{I-\{s_0\}}.$$

Note that $Y_I$ depends only on the component of $s_0$ in the subgraph
of \dtil\ defined by $I$. Hence:

\begin{proposition} \label{cpograph} \marginpar{}

The non-trivial \cpo s are in bijective correspondence with connected
subgraphs of \dtil\ containing $s_0$.

\end{proposition}

From now on we will assume that $I$ is connected, and let $p_G$ denote
the number of subgraphs as in the proposition. (We usually ignore the
trivial orbit, which corresponds to $P_{\emptyset} =\btil$). 

Let $L_I$ denote the Levi factor of $P_I$. Since any proper
sub-Coxeter system of $(\wtil, \stil)$ is finite, $L_I$ is a \fd\
algebraic group whose commutator subgroup $G_{I,\bc}$ is the simple
algebraic group associated to the Dynkin diagram defined by $I$. It is
clear that $Y_I =G_{I, \bc} /Q$, where $Q$ is the maximal parabolic
subgroup of $G_{I,\bc}$ associated to $I-\{s_0\}$. Hence $Y_I$ is
isomorphic to an ordinary maximal flag variety, and in particular is
irreducible and smooth. Since $Y_I$ is also \btil -invariant, it is
therefore a smooth Schubert variety in $\cL _G$. Hence $Y_I =X_\sigma$,
where $\sigma$ is the maximal element of $(W_I)^{I-\{s_0\}} \subset
\wtils$.

Note that using Proposition~\ref{cpograph}, one can easily compute
$p_G$ and list all the \cpo s explicitly. For example, in type $A_n$
there are ${n+1 \choose 2}$ non-trivial \cpo s, all of which are
Grassmannians. In type $E_8$ there are ten non-trivial \cpo s, nine of
which are projective spaces. The exception is the maximal \cpo\ $Y_I$
obtained by deleting the node $s_1$; from the affine Dynkin diagram we
see that $Y_I$ has type $D_8 /D_7$, a nonsingular quadric hypersurface
of dimension 14.

The dimension of any ordinary flag variety $\gc /Q_J$ can be computed as
follows: Let $\Phi _J$ denote the root subsystem corresponding to
$J$. Then

$$dim \, (\gc /Q_J) =|\Phi ^+ | -|\Phi _J ^+|.$$ 

\noindent In particular, we have 

$$dim\, Y_I =|\Theta _I ^+| -|\Theta _{I-\{s_0\}} ^+|.$$

\noindent where $\Theta _I$ is the root system associated to the
subgraph $I$. We identify $\Theta _I$ with the root subsystem of
$\Phi$ having $\{\alpha _0\} \cup \{-\alpha _s : \, s \in I-\{s_0\}\}$
as a base.  

Now let $\cN (I)$ denote the set of neighbor nodes of $I$; note that
$\cN (I)$ uniquely determines $I$ given that $s_{0} \in I$ and $I$ is
connected. Let

$$A_I =\{ \alpha \in \Phi ^+ : \, m_s (\alpha) =m_s (\alpha _0) \, \forall s
\in \cN (I)\}.$$

\begin{lemma} \label{cpo1} \marginpar{}

$A_I =\Theta _I ^+ -\Theta _{I-\{s_0\}} ^+.$ Hence $dim \, Y_I =|A_I|$.

\end{lemma}

\proof It is clear that $\Theta _I ^+ -\Theta _{I-\{s_0\}} ^+ \subset
A_I.$ For the reverse inclusion, note that any positive root $\alpha \in
\Phi ^+$ can be obtained from $\alpha _0$ by successively subtracting
simple roots $\alpha _s$ for various $s \in S$. If $\alpha \in A_I$,
then no such $\alpha _s$ can have $s \in \cN (I)$. Now write \stil\ as a
disjoint union $\stil =I \coprod \cN (I) \coprod K$. Then if $\beta \in
\Theta _I$ and $\gamma \in \Phi _K$, $\beta + \gamma$ is not a
root. Hence no such $\alpha _{s}$ can have $s \in K$, and it follows
that $\alpha \in \Theta _I ^+ -\Theta _{I-\{s_0\}} ^+$.

\bigskip

Now let $\lambda _I \in \cQ^\vee$ denote the coroot lattice
representative for the top cell in $Y_I$. Then 

$$\el ^S (\lambda _I ) =dim \, Y_I =|A_I|.$$

\noindent Furthermore, setting $S^+ (\lambda) =\{s \in S:
\alpha _s (\lam ) >0\}$, we have:

\begin{lemma} \label{cpo2} \marginpar{}

$S^+ (\lam _I) =\cN (I)$. 

\end{lemma}

\proof If $s \in I$ then $s \lambda _I \leq \lambda _I$ in $\wtil/W$, so
$\alpha _s (\lam _I) \leq 0$. If $s \in K$ (where $K$ is as in the proof
of Lemma~\ref{cpo1}), then $s$ commutes with the elements of $I$; hence
$\alpha _s (\lam _I )=0$. Finally, if $s \in \cN (I)$ then $\lam _I
\uparrow s \lam _I$; hence $\alpha _s (\lam _I)>0$. 

\bigskip

There is a simple way to recognize a dominant \cpo :

\begin{proposition} \label{domcpo} \marginpar{}                                 
Suppose $\lambda$ is non-trivial and dominant. Then $X_\lambda$ is a
\cpo\ if and only if $\alpha _0 (\lambda) =2$.

\end{proposition}                                                                                                                                            
\proof Suppose $\alpha _0 (\lambda )=2$. Equivalently, $\lambda $ has
the form (i) $\omega _i ^\vee$ with $m_i =2$; or (ii) $2 \omega _i
^\vee$ with $m_i =1$ (i.e. $i$ minuscule); or (iii) $\omega _i ^\vee +
\omega _j ^\vee$ with $i,j$ minuscule. It follows that there is a
unique connected subset $I \subset \stil$ containing $s_{0}$ such that $S^+ (\lam )=\cN (I)$, where in case (iii) this
uses the fact that every minuscule node of $S$ is a leaf node. Moreover,
\xlam\ is $P_I$-invariant, and inspection of cases (i)-(iii) shows that
$\el ^S (\lam )=|A_I|$. Hence $dim \, \xlam =dim \, Y_I$ by
Lemma~\ref{cpo1}, forcing $\xlam =Y_I$.

Conversely, suppose $\lambda _I$ is dominant. Then $\lambda _I
\downarrow s_0 \lambda _I$, and hence $\alpha _0 (\lambda _I) \geq
2$. Moreover, 

$$|A_I| =\el ^S (\lambda _I) = \sum _{\alpha (\lambda _I)>0} (\alpha
(\lambda _I) -1) \geq \sum _{\alpha \in A_I} (\alpha (\lambda _I)-1)
= (\alpha _0 (\lambda _I) -1)|A_I|.$$

\noindent Hence $\alpha _0 (\lambda _I) =2$.

\bigskip

We conclude this section with a proof of Corollary~\ref{minflag}.

\begin{proposition} \label{mincpo} \marginpar{}

Every \cpo\ is a minuscule flag variety of some simple algebraic
group. Moreover every minuscule flag variety occurs as a \cpo\ in some
affine Grassmannian. 

\end{proposition}

\proof The first assertion is immediate, since $\alpha _0$ can occur at
most once in a positive root of the system $\Theta$. Conversely,
suppose $s$ is a minuscule node of \cD , and let $Z=\gc /Q_{S-s} $
denote the corresponding flag variety. Then there is an \aut\ $\phi$ of
\dtil\ taking $s$ to $s_0$. Hence $Z \cong Y_I$, where $I = \stil
-\{s\}$.

\bigskip

Corollary~\ref{minflag} is now clear from Theorem~\ref{smooth}. 

\subsection{General parabolic orbits}

In this section we fix a proper subset $I \subset \stil$ and consider
arbitrary $P_I$ orbits in $\cL _G$. The propositions here are
well-known (see \cite{mpar} for a detailed exposition), so we omit the
proofs.  

Call an element $\lam \in \cQ ^\vee$ {\it I-minimal} if it is the
minimal element of its left $\wtil _I$-orbit. Then $\lambda$  is $I$-minimal
if and only if $\aslam \geq 0$ for every $s \in I$ (or $1-\azlam \geq 0$
when $s=s_0$). Every $P_I$ -orbit contains a unique $I$-minimal $\lambda$,
and from now on we assume $\lambda$  is $I$-minimal unless otherwise
specified.  Let $\cO _\lambda =P_I \lambda \subset \cL _G$, and let
$M_\lambda =L_I \lambda \subset \cO _\lambda$ denote the corresponding
Levi orbit.

\begin{proposition} $\cO _\lambda$ is isomorphic as an algebraic
variety to the total space of a vector bundle $\xi _\lambda$ over
$M_\lambda$, with fiber dimension $\el ^S (\lam )$. (Here $\lambda$ is
the $I$-minimal representative of the orbit.)

\end{proposition}

The bundle $\xi _\lambda$ can be described explicitly in terms of a
certain representation of $L_I$ arising from the adjoint representation
of $L_I$ on the the Lie algebra of the unipotent radical of
$P_I$. We will not need this description here.

The closure relations on the $P_I$ orbits are given by the Bruhat order
on the set of $I$-minimal $\lam$. We then have a filtration of $\cL _G$
for which the quotients are the Thom spaces $T(\xi _\lambda)$. Now let
$p_\lambda (t)$ denote the generating function for the cells that lie in
$\cO _\lambda$. Note that $p_\lambda (t)$ is counting cells, not
homology groups (although $1+ p_\lambda (t) =|T(\xi _\lambda)| (t)$). In
fact 

$$p_\lambda (t) =t^{\el ^S (\lam)} |M_\lambda | (t).$$

Now let $K_\lambda \subset I$ denote those $s$ such that $\aslam =0$
(or $1-\azlam =0$ if $s=s_0$; in the language of the next section,
these are the ``zero nodes'' of $\lambda$). Then $M_\lambda$ is a flag
variety of type $\Phi _I /\Phi _{K_\lambda}$.

\begin{proposition} \label{levipoly} 
We have
$$|\cL _G| (t) =\sum
_{\lambda } t^{\el ^S (\lam)} |M_\lambda | (t),$$

\noindent where the sum is over all $I$-minimal $\lambda$.  Moreover,

$$|M_\lambda| (t) =\frac{\prod _{i=1} ^{|I|} (1-t^{e_i})}
{(1-t)^{|I|-|K_\lambda|} \prod _{j=1} ^{|K_\lambda|} (1-t^{f_j})},$$

\noindent where $e_1,...,e_{|I|}$ (resp. $f_1,...,f_{|K_\lambda|}$) are
the degrees of $\wtil _I$ (resp. $\wtil _{K_\lambda}$).
%%% todo: are these numbers really degrees and not exponents?
\end{proposition}

We apply this proposition to the exceptional Schubert variety in type
$B_3$. 

\begin{corollary} \label{funnypal} 
Let $G$ have type $B_3$. Then $X_{(3,0,-1)}$ is a singular palindromic
of dimension 9, with \ppoly\ $1112222111$. It satisfies \pd\ over \bq\
but not over \bz , and has singular locus $X_{(2,0,0)}$.
\end{corollary}

\proof Note that $\el ^S (3,0,-1)=9$. It is not difficult to compute the
Bruhat order on \wtils\ through dimension 9 (by hand or by computer);
then one can read off the \ppoly . Another approach is as follows: Let
$I =\{s_0, s_2, s_3\}$ and note that $X_{(3,0,-1)}$ is $P_I$
invariant. We will now use the proposition to analyze the $P_I$-orbit
decomposition through dimension 9.

Let $\lambda _0 =(0,0,0)$, $\lambda _1 =(-1,0,1)$, $\lambda _3
=(-2,1,0)$. Note that these elements are $I$-minimal of $S$-length 0, 3,
and 5, respectively. The Levi orbits $M_{\lambda _i}$ have types
$B_3/B_2$, $B_3 /A_2$, $B_3 /A_1$ respectively. Since $B_n$ has
exponents $2,4, \ldots 2n$, while $A_n$ has exponents $2,3,\ldots n+1$,
we conclude from Proposition~\ref{levipoly} that

\begin{enumerate} 

\item $|M_{\lambda _0}| (t) =111111$ (dimension 5); note that
  $M_{\lambda _0}= \cO _{\lambda _0}$ is the \cpo\ $X_{(2,0,0)}$.

\item $|M_{\lambda _1}|(t)=1112111$ (dimension 6)

\item $|M_{\lambda _2}|(t)=123444321$ (dimension 8). 

\end{enumerate} 

The generating function for the cells of $\cO_{\lambda _0} \cup 
\cO_{\lambda _1} \cup \cO_{\lambda _2}$ is then $|M_{\lambda _0}| (t) +
t^3 |M_{\lambda _1}|(t) +t^5 |M_{\lambda _2}|(t).$ Comparing with 
$|\cL _G (t)|=\frac{1}{(1-t)(1-t^3)(1-t^5)}$, we conclude that the
complex 8-skeleton of $\cL _G$ is contained in $\cO_{\lambda _0} \cup 
\cO_{\lambda _1} \cup \cO_{\lambda _2}$. Since the top cell of $\cO
_{\lambda _3}$ is in dimension $5+8=13>9$, it follows that $X_{(3,0,-1)}
=\cO _{\lambda _0} \cup \cO _{\lambda _1}$. Hence $|X_{(3,0,-1)}|$ is as
claimed, and in particular $X_{(3,0,-1)}$ is palindromic. 

By Theorem~\ref{smooth} (or by direct application of the Chevalley
formula; see \S 6), $X_{(3,0,-1)}$ does not satisfy \pd\ integrally, and in
particular is singular. Since the open orbit $\cO _{\lambda _1}$ is
smooth, the singular locus can only be $\cO _{\lambda _0}
=X_{(2,0,0)}$. Finally, $X_{(3,0,-1)}$satisfies \pd\ rationally by the
Carrell-Peterson theorem. 

\section{The palindromy game I: weak order and the coroot lattice}

In \S 5.1 we give an informal overview of the palindromy game, in its
simpler form using only the weak order. More details are given in \S
5.2. 

\subsection{The node-firing game} 

In this section we describe a variation on Mozes \textit{numbers game}
\cite{mozes} which we call the \textit{node-firing game}.  Our
description of the game follows \cite{b-b,KErik93}.  The purpose of
this game is to make the bijection between $\minreps$ and $\cQ ^\vee$
explicit and to highlight the left-weak order.   

We identify the coweight lattice $\cP ^\vee$ with the group of
$\bz$-valued functions on the Dynkin diagram \cD , where $s \mapsto
\alpha _s (\lambda)$. We extend this labeling to the affine diagram
\dtil\ by putting the value $1-\azlam $ on the node $s_0$. This latter
value is of course determined by the others, but it is important to
include it as part of the picture. The coroot lattice $\cQ^\vee$ can
then be identified with a subgroup of finite index in the group of all
labelled diagrams using the criteria in
Section~\ref{sub:comparision}. For example, suppose $\Phi$ has type
$D_n$ with $n\geq 6$, and $\lambda =-\omega _2 ^\vee + 2\omega _3
^\vee +\omega _{n-2} ^\vee$. Then the labelled diagram of $\lambda$ is

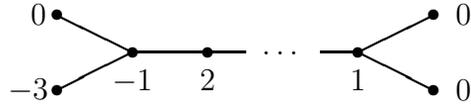
\begin{center}
{\begin{picture}(5,.8)
\mp(1,0)(1,0){2}{\ci}
\put(1,0){\num{-1}}\put(2,0){\num{2}}
\put(3,0){\makebox(0,0){\ldots}}
\put(4,0){\ci}
\put(4,0){\num{1}}
\put(1,0){\line(1,0){1.5}}
\put(4,0){\line(-1,0){.5}}
\put(4,0){\line(2,1){1}}
\put(4,0){\line(2,-1){1}}
\put(5,.5){\ci}\put(5,-.5){\ci}
\put(5,.5){\makebox(0,0)[l]{\hspace{.3\unitlength}$0$}}
\put(5,-.5){\makebox(0,0)[l]{\hspace{.3\unitlength}$0$}}
\put(0,.5){\ci}\put(0,-.5){\ci}
\put(0,.5){\makebox(0,0)[l]{\hspace{-.35\unitlength}$0$}}
\put(0,-.5){\makebox(0,0)[l]{\hspace{-.65\unitlength}$-3$}}
\put(1,0){\line(-2,1){1}}
\put(1,0){\line(-2,-1){1}}
\end{picture}}
\end{center}

\bigskip

\bigskip

\noindent Here $s_0$ is the lower left node, and the nodes not shown
are all labeled zero.  Note, this labeled diagram is not in $\cQ^\vee$
if $n$ is odd.  In fact, $\lam =\omega _{n-2} ^\vee \, \mathrm{mod} \,
\cQ^\vee$, and hence $\lam \in \cQ ^\vee$ if and only if $n$ is even.

Now $\lambda$  has a left descent $\lam \da s\lam $ precisely when $s$ is a
negative node; i.e., $\alpha _s (\lam) <0$, or $1-\azlam <0$ when
$s=s_0$. We refer to this descent as {\it firing} the node $s$. The
effect of such a firing on the labelled diagram is as follows:

\bigskip
                                                                                
(1) The value at $s$ is replaced by its negative;
                                                                                
(2) $k \alpha _s (\lambda)$ (or $1-\alpha _0 (\lambda)$, when
    $s=s_0$) is added to each adjacent node $t$, where 

\[k=    \left\{ \begin{array}{ll}
  1  & \mbox{if $s$ is long or $s,t$ are joined by a single edge}\\
  2  & \mbox{if $s,t$ are joined by a double edge with $s$ at the short
  end, or $\Phi$ has type $A_1$}\\
3 &\mbox{if $s,t$ are joined by a triple edge with $s$ at the short end}
\end{array}
\right. \]

 \bigskip 

If $s$ is at the short end of a multiple bond, we call the firing {\it
back-firing}. As an example in type $F_4$, let $\lambda$  be given by the
following diagram:

\bigskip  
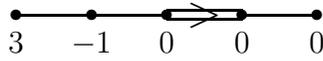
\begin{center}
{\begin{picture}(5,0)                                                           
\mp(0,0)(1,0){5}{\ci}                                                           
\put(0,0){\num{3}}                                                              
\put(1,0){\num{-1}}\put(2,0){\num{0}}\put(3,0){\num{0}}\put(4,0){\num{0}}       
\put(0,0){\line(1,0){2}}                                                        
\put(3,0){\line(1,0){1}}                                                        
\put(2,.06){\line(1,0){1}}                                                      
\put(2,-.06){\line(1,0){1}}                                                     
\put(2.5,0){\makebox(0,0){\Large$>$}}                                           
\end{picture}}                                                                                                                                                
\end{center}                                                                      
\bigskip
Firing twice yields                    

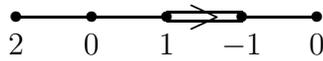
\bigskip                                                                        

\begin{center}
{\begin{picture}(5,0)                                                           
\mp(0,0)(1,0){5}{\ci}                                                           
\put(0,0){\num{2}}                                                              
\put(1,0){\num{0}}\put(2,0){\num{1}}\put(3,0){\num{-1}}\put(4,0){\num{0}}       
\put(0,0){\line(1,0){2}}                                                        
\put(3,0){\line(1,0){1}}                                                        
\put(2,.06){\line(1,0){1}}                                                      
\put(2,-.06){\line(1,0){1}}                                                     
\put(2.5,0){\makebox(0,0){\Large$>$}}                                           
\end{picture}}                                                                  
\end{center}

\bigskip                                                                        
\bigskip                                                                        
If we were firing along a type $A$ subgraph, the configuration of
adjacent 1,-1 surrounded by zeros would simply continue moving steadily
to the right. Here, however, firing $s_3$ back-fires against the arrow
to produce
                                                                                
\bigskip                                                                        

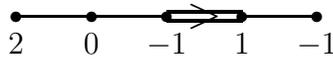
\begin{center}
{\begin{picture}(5,0)                                                           
\mp(0,0)(1,0){5}{\ci}                                                           
\put(0,0){\num{2}}                                                              
\put(1,0){\num{0}}\put(2,0){\num{-1}}\put(3,0){\num{1}}\put(4,0){\num{-1}}      
\put(0,0){\line(1,0){2}}                                                        
\put(3,0){\line(1,0){1}}                                                        
\put(2,.06){\line(1,0){1}}                                                      
\put(2,-.06){\line(1,0){1}}                                                     
\put(2.5,0){\makebox(0,0){\Large$>$}}                                           
\end{picture}}                                                                  
\end{center}

\bigskip                                                                        
\bigskip                                                                        

\noindent which has two negative nodes and hence covers a pair of
elements in the weak order. Thus we have reached a fork in the Hasse
diagram of the order ideal of $\lambda$.

We remark that the node-firing game yields a simple algorithm for
computing the bijections $\cQ ^\vee \cong \wtils$ defined by the diagram

$$
\cQ^\vee \stackrel{\cong}{\lra} \wtil /W \stackrel{\cong}{\lra}
\wtils.
$$

\noindent Suppose first that we are given $\lam \in \cQ ^\vee$. If
$\lam = 0$, then every node has label 0 except for the node
corresponding to $s_{0}$ which is labeled 1.  If $\lam \neq 0$, then
$\lam$ has at least one negative node $t_1 \in \stil$. Then $\lam
\downarrow t_1 \lam$. Repeating the process yields

$$\lam \downarrow t_1 \lam \downarrow t_2 t_1 \lam \downarrow \ldots
\downarrow t_m t_{m-1} \ldots t_1 \lam =0, $$

\noindent where $m=\el ^S (\lam )$. Taking $\sigma =t_1 t_2 \cdots t_m$,
we have (i) $\sigma \in \wtils$ (in particular, $t_m=s_0$), (ii) the
product is reduced, and (iii) $\lam W =\sigma W$. Hence $\lam \mapsto
\sigma$.

In the reverse direction, suppose we are given $\sigma \in
\wtils$. The corresponding $\lam \in \cQ ^\vee$ is obtained by letting
$\sigma$ act on $0 \in \cQ ^\vee$, and is computed explicitly as
follows: Choose a reduced decomposition $\sigma =t_1 t_2 \cdots t_m$,
where necessarily $t_m =s_0$. Then fire up starting from $0$: 

$$ 0 \uparrow t_m \cdot 0 \uparrow t_{m-1} t_m \cdot 0 \uparrow \ldots
\uparrow t_1 t_2 \ldots t_m \cdot 0 =\lam .$$ 
Many examples of these node-firings can be found below.

\subsection{Elementary obstructions to palindromy}

In order to show that a given $\lambda$ is {\it not} palindromic, we
show that it has too many cells near the top dimension. Often these
violations of palindromy can be detected by merely firing negative
nodes. In general, however, we must consider more general Bruhat
descents $\lam \da r\lam$ defined by non-simple reflections $r \in
\wtil$. Whenever possible, we arrange things so that the necessary
information can be read off directly from the labelled diagram.
Indeed the reader may find it helpful---or at least amusing---to think
of this process as a game, called the \textit{palindromy game}, in
which the object of one player is to find a palindromic $\lambda$
satisfying given initial conditions, while the object of the other is
to prevent it by finding an excess of Bruhat descents $\lam \da
r\lam$. In fact we will often refer to
such descents as ``moves''.

We emphasize that the ``palindromy game'' is much easier to play than to
write down. In many cases, the reader may prefer to draw the pictures
and work out the moves for herself, rather than wade through the
verbiage required to explain them in print. To get started, here is an
informal discussion of the most basic principles of the game. Suppose
that $\lambda$ is nonzero and palindromic. Then there are the following
{\it palindromy rules}.  These rules give necessary conditions for
$X_{\lambda }$ to be palindromic.

\bigskip

{\bf Rule 1.} $\lambda$  has exactly one negative node $s \in \dtil$ (the case
$s=s_0$ corresponds to $\lam $ dominant).

\bigskip

{\bf Rule 2.} Except in type $A$, there cannot be two zero nodes adjacent to
$s$.

\bigskip

{\bf Rule 3.} If $s \neq s_0$, then $\azlam \leq 1$.   

\bigskip

{\bf Rule 4.} More generally, $\lambda$ cannot ``fork too soon''. 

\bigskip

Note that Rule 3 is a special case of rule 1.  If an arbitrary
$\lambda$ violates Rule 3 or if $s=s_{0}$ and $\alpha_{0}(\lambda)>2$
then we say $\lambda$ is \textit{overweight}.

See below for the precise definition of ``forks too soon''. Informally,
this just means that the Hasse diagram of $\lambda$ (coming down from
the top) reaches a fork sooner than the Hasse diagram of \wtils\ (coming
up from the bottom), thereby violating palindromy. Often one can see
this instantly from the diagram.

\bigskip

\noindent {\it Example:} In the $F_4$ example above, $\lambda=(-1,0,0,0)$ 
clearly forks too soon because $s_1$ has a ``head-start'' on $s_0$.
                                    
\bigskip                            

\noindent {\it Example:} Surprisingly, $E_8$ is in many ways the
simplest type. One reason for this is that it is the unique simply-laced
type with no minuscule nodes; another reason will be given below. Firing
up from $0\in \cQ ^\vee$ in $E_{8}$ it takes six steps to reach the
fork. Hence if a given $\lambda$ is to have any chance at palindromy,
the fork at $s_4$ must be suitably protected. For example, suppose
$s=s_1$ is the unique negative node of $\lambda$. Then either $s$ must
be blocked away from $s_4$ by an intermediate positive node (picture the
negative value moving to the right under repeated firings), or at least
one of the exit nodes $s_2, s_5$ of the fork must be positive (and in
fact must be at least as large as $|\alpha _1 (\lam )|$, but we ignore
this refinement for the moment).

Now suppose in addition that $\alpha _1 (\lam )=-1$. Then by
exploiting Rule 1 and Rule 3 and the coefficients $m_s$ in the
expansion $\alpha _0 =\sum _{s \in S} m_s (\alpha) _s$ (see \S 12), we
see at once that $\lambda$ can only have the form
 
\bigskip

$$
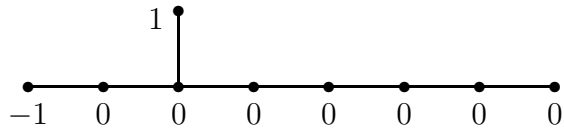
{\begin{picture}(5,1.2)
\mp(0,0)(1,0){8}{\ci}\put(2,1){\ci}
\put(0,0){\num{-1}}\put(1,0){\num{0}}\put(2,0){\num{0}}\put(3,0){\num{0}}
\put(4,0){\num{0}}\put(5,0){\num{0}}\put(6,0){\num{0}} \put(7,0){\num{0}}
\put(2,.9){\makebox(0,0)[r]{$1$\hspace{.2\unitlength}}}
\put(2,0){\line(0,1){1}}
\put(0,0){\line(1,0){7}}
\end{picture}}
$$

\bigskip

\noindent Here $X_\lambda$ turns out to be the \cpo\ $Y_I$ with $I=\stil
-\{s_2\}$, which has has type $\bp ^8$.  This is easily seen by firing
the negative node all the way down to the bottom. Note that the 1 serves
to protect the fork, and is killed by the -1 as it passes by.  This
example illustrates why all three of the $E$ types are actually easier
than the classical types: There are few repetitions among the
coefficients $m_s$.

%\subsection{Elementary obstructions to palindromy}

Consider the length generating function in Theorem~\ref{t:Bott}
expanded out 

$$|\wtils | (t) = \sum_{\sigma \in \wtils} t^{\el(\sigma )}= 1+t + \ldots + t^{k-1} +a_k t^k + \ldots$$

\noindent where $k=k_G$ is minimal such that $a_k >1$. If no such $k$
exists, we set $k_G =\infty$ (this happens only in type $A_1$). 

Pictorially, the first fork (going up) in the Hasse diagram for \wtils\
occurs at height $k_G -1$ (see the diagrams at the end of the paper for
examples). To compute $k_G$ we start at $s_0$ on the affine Dynkin
diagram, and follow the only possible path until a node of degree 3 or
higher is reached. The number of nodes in such a path is $k_G -1$. Here
we allow doubling back along a multiple edge; for example, in type $F_4$
we reach a fork at $s_3$, where we have the option of continuing to
$s_4$ or following the unused edge back to $s_2$:

$$
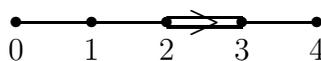
{\begin{picture}(5,0)
\mp(0,0)(1,0){5}{\ci}
\put(0,0){\num{0}}
\put(1,0){\num{1}}\put(2,0){\num{2}}\put(3,0){\num{3}}\put(4,0){\num{4}}
\put(0,0){\line(1,0){2}}
\put(3,0){\line(1,0){1}}
\put(2,.06){\line(1,0){1}}
\put(2,-.06){\line(1,0){1}}
\put(2.5,0){\makebox(0,0){\Large$>$}}
\end{picture}}
$$

\bigskip

The coefficient $a_{k_G}$ is just the number of options at the fork. Thus
$k_G$ and $a_{k_G}$ are easily determined by inspecting the affine
Dynkin diagrams.

\begin{equation}\label{hassefork}
k_G=    \left\{ \begin{array}{ll}                                               
 2   & \mbox{type $A_n$, $n>1$}\\                                               
 3   & \mbox{type $B,C,D$}\\                                                    
 4   & \mbox{type $E_6$}\\                                                      
 5   & \mbox{type $E_7, F_4, G_2$}\\                                            
 7   & \mbox{type $E_8$}\\                                                      
 \infty & \mbox{type $A_1$}                                                     
\end{array}                                                                     
\right.
\end{equation}

\begin{equation}\label{e:options}
a_{k_G}=    \left\{ \begin{array}{ll}
 3   & \mbox{in type $D_4$}\\
 2   & \mbox{otherwise}
\end{array}
\right. 
\end{equation}

\noindent {\it Remark:} This result can also be proved by computing the
rational cohomology of the loop group: It is well known that
                                                                                
$$H^* (BG; \bq) \stackrel{\cong}{\lra} (H^* (BT; \bq)) ^W,$$                    
\noindent and that the ring of invariants $(H^* (BT: \bq)) ^W$ is a
polynomial algebra on generators of complex dimension $d_1 \leq d_2
\leq  \ldots \leq d_n$. The
degrees $d_i$ can be computed explicitly in each Lie type; see
\cite{hcox}, p. 59. Since $\Omega G$ is the double-loop space of BG, we
have $k_G=d_2 -2$, yielding the table above. From this point of view,
the exceptional value $a_{k_G}=3$ in type $D_4$ can be traced to the
fact that $H^8 BSpin(8)$ has rank 3, with generators the Pontrjagin
classes $p_1 ^2, p_2$ plus the Euler class.  

\bigskip

Now for all $X_{\lambda}$, $|X_\lambda|(t) \leq |\wtils| (t)$
coefficient by coefficient.  
In addition, if $X_\lambda$ is
palindromic of dimension $d$, then
$$
D|X_\lambda| (t) \leq |\wtils| (t)
$$ 
where if $f(t)= a_{0}+a_{1}t+\dots +a_{n}t^{n}$ is a polynomial of
degree $n$, the {\it dual polynomial} is $Df(t) =t^n
f(t^{-1})=a_{n}+a_{n-1}t+\dots +a_{0}t^{n}$.  In a range of dimensions
(depending on $\lambda$), the inequality will actually be an
equality. In any case, \eqref{hassefork} forces restrictions on
$|X_\lambda | (t)$ near the top dimension $d$. Consider for example
the configurations in a Hasse diagram

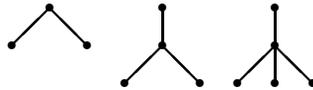
\begin{center}
\setlength{\unitlength}{.1cm}
\begin{picture}(40,10)(0,0)
%\Thicklines
\put(0,5){\line(1,1){5}}
\put(0,5){\circle*{1}}
\put(5,10){\line(1,-1){5}}
\put(5,10){\circle*{1}}
\put(10,5){\circle*{1}}
\put(15,0){\line(1,1){5}}
\put(15,0){\circle*{1}}
\put(20,5){\line(0,1){5}}
\put(20,5){\circle*{1}}
\put(20,10){\circle*{1}}
\put(20,5){\line(1,-1){5}}
\put(25,0){\circle*{1}}
\put(30,0){\line(1,1){5}}
\put(30,0){\circle*{1}}
\put(35,0){\line(0,1){10}}
\put(35,0){\circle*{1}}
\put(35,5){\circle*{1}}
\put(35,10){\circle*{1}}
\put(35,5){\line(1,-1){5}}
\put(40,0){\circle*{1}}
\end{picture}
\end{center}

\bigskip

\noindent which we call a \textit{pair}, a \textit{fork}, and a
\textit{trident} respectively, with $\lambda$ sitting at the top. If
$X_\lambda$ is palindromic then $\lambda$ cannot cover a pair or a
trident, and if $\Phi$ is not of type $A$ then it cannot cover a
fork. In type $B$ we will encounter an (upside down) \textit{scepter}

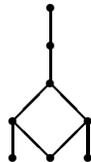
\begin{center}
\setlength{\unitlength}{.1cm}
\begin{picture}(0,20)(0,-10)
%\Thicklines
\put(0,0){\circle*{1}}
\put(0,0){\line(-1,-1){5}}
\put(-5,-5){\circle*{1}}
\put(0,0){\line(1,-1){5}}
\put(5,-5){\circle*{1}}
\put(5,-5){\line(0,-1){5}}    %% right bottom
\put(-5,-5){\line(1,-1){5}}  %%% comment out this line to get your broom
\put(5,-5){\line(-1,-1){5}}
\put(-5,-5){\line(0,-1){5}}  %% left bottom
\put(0,-10){\circle*{1}}
\put(5,-10){\circle*{1}}      %% right bot point
\put(-5,-10){\circle*{1}}     %%left bot point
\put(0,0){\line(0,1){10}}
\put(0,5){\circle*{1}}
\put(0,10){\circle*{1}}
\end{picture}
\end{center}

\noindent A palindromic $\lambda$ cannot cover a scepter.  

In general, we say that $\lambda$ {\it forks too soon} if the Hasse
diagram of its order ideal (coming down from the top) reaches a fork
sooner than the Hasse diagram of \wtils\ (coming up from the
bottom). More precisely: Say $|X_\lambda| (t) =1 + a_{1} t + \ldots +
a_k t^k + t^{k+1}+\ldots + t^{m} $, where $m = \el^{S}(\lambda )$ and
$k$ is maximal such that $a_k >1$. Then $\lambda$ forks too soon if
$0\leq m - k <k_G$. Hence, this proves Rule 4 in the palindromy game.

\bigskip

\noindent {\it Example:} We show that if $G$ has type $E_8$ and
$\lambda$ is anti-dominant and non-trivial, then $X_\lambda$ is not
palindromic. Assume there is a unique negative node $s$ (otherwise
$\lambda$ covers a pair), and all other nodes are zero. Assume further
that $s$ is a leaf node of \cD\ (otherwise $\lambda$ covers a
fork). Finally, if $s$ is one of the three leaf nodes $s_1, s_2, s_8$,
then by repeated firing we reach the fork in the Dynkin diagram in 3, 2
or 5 steps respectively. But it takes 6 steps to reach the fork from
$s_0$. Thus $\lambda$ forks too soon and hence is not palindromic.

\section{Poincar\'e duality and the affine Chevalley formula}

We state an affine version of the Chevalley formula,
Proposition~\ref{chev}, and record its implications for Poincar\'e
duality. Proposition~\ref{chev} is a special case of a vastly more
general cup product formula in equivariant cohomology, valid for
arbitrary Kac-Moody flag varieties (\cite{kumar}, Corollary 11.3.17 and
Remark 11.3.18). 

Let $[X_\lambda] \in H_{2d} \cL _G$ denote the homology class carried by
$X_\lambda$, where $d=\el ^S (\lambda)$. These classes form the {\it
Schubert basis} of $H_* \cL _G$. Of course we can equally well regard
$[X_\lambda ]$ as a homology class in any \svar\ containing
$X_\lambda$. Let $y_\lambda \in H^{2d} \cL _G$ be Kronecker dual to
$[X_\lambda]$ with respect to the Schubert basis.  We will use the
abbreviation $y$ for the special class $y=y_{\alpha _0 ^\vee} =y_{s_0}$,
the generator of $H^2 \cL _G$.

We have seen that if $X_\lambda$ is palindromic of dimension $d$, then
it has just one $2d-2$ cell, and hence $H^{2d-2} X_\lambda \cong \bz
$. If in addition $X_\lambda$ satisfies \pd , then cup product with $y$
defines an \iso\ $H^{2d-2} X_\lambda \stackrel{\cong}{\lra} H^{2d}
X_\lambda$. Hence the following {\it affine Chevalley formula} puts
severe restrictions on the possible such $\lambda$. 

\begin{proposition} \label{chev} \marginpar{}

If $\lambda \uparrow s\lambda$ for $s \in \stil$, then 

\[<yy_\lambda , [X_{s\lambda}]>=    \left\{ \begin{array}{ll}
c \alpha _s (\lambda)    & \mbox{if $s \neq s_0$}\\
1-\alpha _0 (\lambda)    & \mbox{if $s=s_0$}
\end{array}
\right. \]

\noindent where
                                                                                
\[c=    \left\{ \begin{array}{ll}                                               
1     & \mbox{if $\alpha_{s}$ is long}\\                                            
2   & \mbox{if $\alpha_{s}$ is short in type B,C,F}\\                               
3 & \mbox{if $\alpha_{s}$ is short in type $G_2$}                                   
\end{array}                                                                     
\right. \]                                                                                   
\end{proposition}
Note that the assumption $\lambda \uparrow s\lambda$ is equivalent to
the positivity of $\aslam$ or $1-\azlam$ in the node-firing game.
It will be convenient to reformulate the Chevalley formula in terms of
the cap product. 

\begin{proposition} \label{chevcap} \marginpar{}                                                                                                       
If $\lambda \downarrow s\lambda$ for $s \in \stil$, then 

\[<y_{s\lambda}, y \cap [X_{\lambda}]>=    \left\{ \begin{array}{ll}
-c\alpha _s (\lambda)    & \mbox{if $s \neq s_0$}\\
\alpha _0 (\lambda) -1    & \mbox{if $s=s_0$}
\end{array}
\right. \]

\noindent where $c=1,2,3$ is the constant defined in
Proposition~\ref{chev}. 
 
\end{proposition}    

\proof This follows by simply reversing the roles of $\lambda$ and
$s\lambda$ in Proposition~\ref{chev}. 

\bigskip

We then have at once: 

\begin{proposition} \label{pdtest} \marginpar{}

If $X_\lambda$ satisfies \pd\ and $\lambda \downarrow s\lambda$ for $s \in \stil$, then $c=1$ and $\alpha _s (\lambda)=-1$
(or $1-\alpha _0 (\lambda )=-1$). In particular $\alpha _s$ must be
long.  

\end{proposition}

\noindent {\it Remark:} This proposition already suffices to show that
there are only finitely many \svars\ satisfying \pd\ in a fixed $\cL
_G$, since it bounds the values $\alpha _s (\lambda)$ for $s \in S$.

\section{Chains}\label{s:chains}

In this section we study the chains in $\cL _G$. In particular, we will
show that the theorems of the introduction hold for chains. We begin
with some simple observations.

\subsection{General observations}

\begin{proposition} \label{p:cofinal} \marginpar{}

Every infinite subset of \wtils\ is cofinal in the Bruhat order. 

\end{proposition}

\proof If $I \subset \stil$ is a proper subset, then $\wtil _I$ is a
finite Coxeter group. Let $k_I$ denote the maximal length of an element
of $\wtil _I$, and let $k=max \, k_I$, where $I$ ranges over all such
proper subsets. Then if $w \in \wtil$ and $\el (w) >k$, every reduced
expression for $w$ contains every $s \in \stil$ at least once. 

Now let $V \subset \wtils$ be an infinite subset, and let $\sigma \in
\wtils$. Since $\wtils$ has only finitely many elements of any fixed length,
we can choose $v \in V$ with $\el (v) \geq (k+1) \el (\sigma)$. It then
follows from the preceding paragraph that $\sigma \leq v$, proving that
$V$ is cofinal. 

\begin{corollary} \label{finitechain} \marginpar{}

If $W$ is not of type $A_1$, there are only finitely many chains  $X_{\lambda}$ in $\cL$.

\end{corollary}

\proof If there are infinitely many chains, then it follows from 
Proposition~\ref{p:cofinal}  that $\wtils$ itself is a chain. But this is the case only
in type $A_1$. 

\bigskip

Following Stembridge's terminology \cite{stem}, call an element of a
Coxeter group {\it rigid} if it has a unique reduced expression. Note
that if $w$ is rigid, then so is any element obtained by taking a
factor of the reduced expression for $\sigma$.  The next result
follows by an easy induction on length.

\begin{proposition}  \marginpar{}

Every chain is rigid. 

\end{proposition}

We will classify the chains by first classifying all the rigid
elements. In fact the rigid elements are easily determined from the
affine Dynkin diagram. Suppose $s, t \in \stil$ with $s \neq t$, and
let $m_{st}$ denote the order of $st$. Thus $m_{st}=\{2,3,4,6
\}$. Then a rigid element cannot contain subwords of the form $st,
sts, stst, ststst$ respectively in these four cases.  Therefore, rigid
elements are fully commutative \cite{stem}.  We interpret these
restrictions on the Dynkin diagram as follows:

Let $\sigma =t_k t_{k-1}\cdots t_1$ be the reduced expression for a
rigid element, where $\sigma \in \wtils$ and hence $t_1 =s_0$. Then
every pair of adjacent nodes in this expression must be adjacent in
\dtil . Hence $\sigma$ determines and is determined by a path in
\dtil\ starting at $s_0$.  Furthermore the path in question can
reverse direction only along a multiple edge, and if the multiple edge
is a double edge then it can reverse direction only once. If it is a
triple edge then the path can reverse direction at most three
times. Let us call such a path an {\it admissible path}. Then every
rigid element is associated to an admissible path in this way and vice versa. This
allows us to read off the rigid elements directly from \dtil .

\noindent 

{\it Examples:} 1. Type $A$. There are two infinite families of rigid
elements, obtained in the evident way by starting at $s_0$ and
following an admissible path of arbitrary length clockwise or
counterclockwise around the diagram. These are the only rigid
elements. Note that the two families are conjugate under the
involution of \dtil\ fixing $s_0$.

2. Type $C$. Consider an admissible path (starting at $s_0$). 

$$
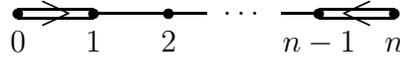
{\begin{picture}(5,0)
\mp(0,0)(1,0){3}{\ci}
\put(1,0){\num{1}}\put(2,0){\num{2}}
\put(3,0){\makebox(0,0){\ldots}}
\mp(5,0)(-1,0){2}{\ci}
\put(4,0){\num{n-1}}\put(5,0){\num{n}}
\put(1,0){\line(1,0){1.5}}
\put(4,0){\line(-1,0){.5}}
\put(4,.06){\line(1,0){1}}
\put(4,-.06){\line(1,0){1}}
\put(4.5,0){\makebox(0,0){\Large$<$}}
\put(.5,0){\makebox(0,0){\Large$>$}}
\put(0,.06){\line(1,0){1}}
\put(0,-.06){\line(1,0){1}}
\put(0,0){\num{0}}
\end{picture}}
$$

\bigskip

\noindent At $s_1$ there are two options: We can reverse direction to obtain a
maximal rigid element $s_0 s_1 s_0$, or we can continue to the
right.  In the latter case we obtain an infinite family of rigid
elements by running back and forth along \dtil\ in the evident
way. These are the only rigid elements.

3. If $W$ is not of type $A$ or $C$, then there are only finitely many
rigid elements. This is also clear, because then $\dtil$ has no
cycles and at most one multiple edge; hence every maximal admissible path
eventually terminates at a leaf node.  

\bigskip

Let $\sigma$ be a chain of length $m$ and let $y_0=1, y_1,\ldots ,y_m$
denote the Schubert basis for $H^* X_\sigma$. Define integers $a_k$ by
$y_1 y_{k-1} =a_k y_k$ for $1 \leq k \leq m$. Note that these integers
are positive by the Chevalley formula, and $a_1 =1$. In particular,
$H^* (X_\sigma ; \bq) $ is a truncated polynomial algebra $\bq
[y_1]/y_1 ^{m+1}$, and hence $X_\sigma$ satisfies rational \pd . Call
$(a_1,\ldots ,a_m)$ the {\it cup sequence} of $\sigma$.
                   
\begin{proposition}  \marginpar{}

Let $\sigma$ be a chain of length $m$. Then $X_\sigma$ satisfies
Poincar\'e duality over $\bz$ if and only if the cup sequence of
$\sigma$ is palindromic, in the sense that $a_k =a_{m-k+1}$ for all
$k$.

\end{proposition}

\begin{proof}
Define $c_k$ by $y_1 ^k =c_k y_k$, and note that $c_k \neq
0$. Then \pd\ holds if and only if $c_k c_{m-k} =c_m$ for all $k$. But
$c_k =a_1 \ldots a_k$, and the result follows by induction on $k$.
\end{proof}
                                                           
\subsection{Simply-laced types} 

We show that the chains in the simply laced types all lead to smooth
affine Schubert varieties.

\begin{proposition}  \marginpar{}
Let $\sigma =t_k t_{k-1}\cdots t_1$ be a chain in \wtils\ whose
associated admissible path has no multiple edges. Then $X_\sigma$ is a
\cpo\ isomorphic to $\bp ^k$, and hence $X_\sigma$ is smooth.
\end{proposition}

\begin{proof}
Suppose that $W$ is not of type $A$. Then the $t_i$'s are
distinct, since \dtil\ has no cycles and the path cannot reverse
direction. Hence the path of $\sigma$ is just a type $A_k$ subgraph $I$
of \dtil , and  $X_\sigma =Y_I \cong \bp ^k$. 

If $W$ has type $A_n$, then $n>1$ and $\sigma$ belongs to one of the
two infinite families of rigid elements described above. The two
families are conjugate, so we may as well suppose the path of $\sigma$
runs counterclockwise, so $\sigma =s_{k-1}\cdots s_1 s_0$ for some
$k$, the subscripts being interpreted mod $n+1$. Then $k-1 <n$: For if
$k-1=n$ then $\sigma \downarrow s_{k-1} s_{k-3} \cdots s_1 s_0$, so
that $\sigma$ covers a pair; hence for $k-1 \geq n$, $\sigma$ is not a
chain. Therefore $k-1<n$, in which case the argument used above shows
that $X_\sigma$ is a \cpo\ isomorphic to $\bp ^k$.
\end{proof}

\begin{corollary} \label{chainproj} \marginpar{}
In the simply-laced case $ADE$ (excluding $A_1$), every chain is a \cpo\
isomorphic to $\bp ^k$. 
\end{corollary}

To describe the chains explicitly, it suffices to list the maximal
chains. In type $A_n$ for $n>1$ there are two maximal chains namely
$s_{2}s_{3}\cdots s_{n}s_{0}$ and $s_{n-1}\cdots s_{1} s_{0}$
with coroot lattice representatives $(1,-1,\und)$ and $(\und,-1,1)$ respectively.  The
corresponding affine Schubert varieties are both isomorphic to $\bp
^n$. In types $DE$ every rigid element is a chain. In type $D_n$ there
are three maximal chains, two $\bp ^n $'s and one $\bp ^3$. Finally
there are two $\bp ^5$'s in $E_6$, a $\bp ^7$ and a $\bp ^5$ in $E_7$,
and a $\bp ^8$ and a $\bp ^7$ in $E_8$. We leave it to the interested
reader to write down the minimal length and coroot lattice
representatives for these chains.

\subsection{Non-simply-laced types}

In this subsection, we classify the chains in each of the
non-simply-laced types and identify which ones index smooth affine
Schubert varieties.

$A_1$: Every element is a chain; indeed $\wtils$ itself forms an
infinite chain with cup sequence $a_k =k$. To see this, note that the
coroot lattice representative for the element of length $k$ is
$\lambda = (2j)$ if $k=2j-1$ and $\lambda =(-2j)$ if $k=2j$. In the
first case we have $\alpha(\lambda) =2j$ and in the second $1-\alpha
(\lambda)=2j+1$, where in this case $\alpha = \alpha_{1} =\alpha _0$
is the unique positive root. Hence our claim follows from the
Chevalley formula. We also conclude that the only non-trivial smooth
Schubert variety in type $A_1$ is the unique \cpo , $X_{s_{0}} \approx \bp ^1$.

This discussion also recovers the well-known fact in type $A_{1}$ that
$H^* \cL _G$ is a divided power algebra, a fact normally deduced from
the equivalence $\cL _G \cong \Omega SU(2)$ and the Serre \sss .

\bigskip

$B_n$: There are two maximal chains for $n\geq 3$: (1) $\sigma
=s_1s_2s_0$ with coroot lattice representative $(-1,0,1,\und
)$; and (2) $\tau =s_0 s_2 s_3\cdots s_ns_{n-1}\cdots s_3 s_2 s_0$,
with \clr\ $(2,\und )$.

Note that both $X_{\sigma }$ and $X_\tau$ are \cpo s $Y_I$, where
$I=\{s_{0},s_{1},s_{2} \}$ and $I=\stil -\{s_1\}$ respectively.
$X_{\sigma } \approx \bp^3$ is of the simply laced type covered above.
$X_{\tau}$ is a flag variety of type $B_n/B_{n-1}$; that is, a
nonsingular quadric hypersurface of dimension $2n-1$. A standard
calculation shows that the cup sequence is $(1,1,\ldots ,1,2,1,\ldots
,1,1)$; here this follows at once from the Chevalley formula. Thus the
complex $k$-skeleton of $X_{\tau}$ satisfies Poincar\'e duality if and only if $k
\leq n-1$, in which case it is just $\bp ^k$.  We observe that firing
down from $\tau$ yields the chains below $\tau$, namely starting with 

\begin{center}
$\tau =$ \hspace{.3in}
\begin{picture}(5,0)
\mp(1,0)(1,0){2}{\ci}
\put(1,0){\num{0}}\put(2,0){\num{0}}
\put(3,0){\makebox(0,0){\ldots}}
\mp(5,0)(-1,0){2}{\ci}
\put(4,0){\num{0}}\put(5,0){\num{0}}
\put(1,0){\line(1,0){1.5}}
\put(4,0){\line(-1,0){.5}}
\put(4,.06){\line(1,0){1}}
\put(4,-.06){\line(1,0){1}}
\put(4.5,0){\makebox(0,0){\Large$>$}}
\put(0,.5){\ci}\put(0,-.5){\ci}
\put(0,.5){\makebox(0,0)[l]{\hspace{-.35\unitlength}$2$}}
\put(0,-.5){\makebox(0,0)[l]{\hspace{-.7\unitlength}$-1$}}
\put(1,0){\line(-2,1){1}}
\put(1,0){\line(-2,-1){1}}
\end{picture}
\end{center}

\noindent we get 
\[
(2, \und) \da (2,-1,\und) \da (1,1,-1, \und ) \da (1,0,1,-1, \und
)\da \ldots.
\]

There is one additional rigid element $\zeta =s_1 s_2 \cdots s_n \cdots s_2
s_0$. Note that after omitting the $s_2$ on the left we still have an
element of \wtils ; this shows that $\zeta$ covers a pair and so is not
a chain. 

\bigskip

\noindent {\it Remark:} In type $D_n$, the \cpo\ $Y=Y_{\tilde{S} -\{s_1\}}$
has type $D_n/D_{n-1}$, a nonsingular quadric hypersurface of dimension
$2n-2$. Its \clr\ is $(2,\und )$ as in type $B_n$. In this case,
however, the middle homology has rank two, by a standard calculation or
inspection of the Bruhat poset. Hence the Schubert subvarieties of $Y$
of dimension $k$, $n-1 <k<2n-2$, are not even palindromic. For future
reference, we note that firing down from the top yields
 
%%todo: find where used and reference it appropriately
\begin{equation}\label{e:typeD.segments}
(2, \und ) \da (2,-1, \und) \da (1,1,-1, \und) \da \cdots ,
\end{equation}

\noindent exactly as in type $B_n$ except that just above the middle
dimension we reach the element $(1, \und , 1, -1,-1)$, which covers a pair.

\bigskip

$C_n$: Recall that there is an infinite family of rigid elements,
obtained by running back and forth along \dtil . The maximal chain in
this family is $\sigma =s_1 s_2 \cdots s_n s_{n-1} \cdots s_1 s_0$ (note that
$s_0 \sigma$ covers a pair). There is one other maximal chain: $\tau
=s_0 s_1 s_0$.

Note that $X_\tau$ is a \cpo\ and is not $\bp ^3$ but rather
the symplectic Grassmannian $Sp(2)/U(2)$. By the Chevalley formula its
cup sequence is $(1,2,1)$ and hence its complex 2-skeleton is not
smooth. The \clr\  for $\tau $ is $(0,1, \und) = \omega_{2}^{\vee}$. 

The \clr\ of $\sigma$ is $\lambda =(-1,\und) = -\omega_{1}^{\vee}
=-\alpha _0 ^\vee$, hence $\lambda$ is anti-dominant and $\lambda
=\sigma$ as elements of $\wtil$.  The cup sequence is $(1,2,2,\ldots
,2)$. Indeed, the complete list of \clr s below $\sigma$ is obtained
as follows, starting from the top:

$$(-1, \und) \da (1,-1,\und) \da \ldots \da (\und , 1, -1, 0) 
\stackrel{s_{n-1}}{\da} (\und , 1, -2)  \stackrel{s_n}{\da}
(\und, -1, 2) \stackrel{s_{n-1}}{\da} (\und , -1,1,0) \da \ldots \da (1,
\und ).
$$
All but one of the factors of 2 in the cup sequence occurs because of
a short node; the application of $s_n$ also yields a factor of 2
because $\alpha _n (\und ,-1,2)=2$. It follows that none of the
complex skeleta are smooth, except for the $\bp ^1$ at the bottom.

An alternative way to identify the cup sequence of $\sigma$ is to note
that $X_\sigma$ is the closure of the lowest non-trivial $P$-orbit,
and it can be shown that it is therefore the Thom space of the line
bundle over $\bp ^{2n-1}$ associated to the highest root
\cite{littig}. This line bundle is just $(\gamma ^*)^2$, where
$\gamma ^* \downarrow \bp ^{2n-1}$ is the hyperplane section
bundle. Hence if $u$ is the Thom class, by a general formula we have
$u^2 =c_1 (\gamma ^{*2}) u=2yu$ where $y =c_1 (\gamma ^*) $, yielding
the cup sequence above.

\bigskip

$F_4$: There are two maximal chains. The first is $s_0 s_1 s_2 s_3 s_2
s_1 s_0$, a \cpo\ of type $B_4/B_3$ with \clr\ $(0,0,0,1)$. The \clr s
of its skeleta are given by 

$$(0,0,0,1) \da (-1,0,0,1) \da (1,-1,0,1) \da (0,1,-1,1) \da
(0,-1,1,0) \da (-1,1,0,0) \da (1,0,0,0).$$

\noindent From our analysis of the type $B$ case, we know that the cup
sequence is $(1,1,1,2,1,1,1)$ and that only the complex $k$ skeleta
with $k \leq 3$ or $k=7$ are \pds s.

The second maximal chain is $s_4s_3s_2s_1s_0$, which has \clr\
$(0,1,0,-1)$ and cup sequence $(1,1,1,2,2)$. Hence it is not a \pds .

Note that every rigid element is a chain. 

\bigskip

$G_2$: There are two maximal chains. The first is $s_2 s_1 s_2 s_1 s_2
s_0$, which has \clr\ the anti-dominant class $(0,-1)=-\omega _2
^\vee$. Its cup sequence is $(1,1,3,2,3,1)$ and hence it is not a \pds
. To see this, we write down the \clr s of its skeleta: 

$$(0,-1) \da (-1,1) \da (1,-2) \da (-1,2) \da (1,-1) \da (0,1),$$

\noindent where we recall that firing the short node $\alpha_{1}$ in
type $G_2$ adds three times the short value to its neighbor. From the
Chevalley formula we see that the short node firings produce cup
product coefficients of 3, while the long node firing $(1,-2)
\downarrow (-1,2)$ yields a coefficient of 2 because of the -2 in the
second position. We conclude that only the complex 1 and 2 skeleta are
smooth; these are the two \cpo s $\bp ^1$, $\bp ^2$. It can be shown
that $X_{-\omega _2 ^\vee }$ is the Thom space of a line bundle over
the maximal flag variety of $G_2$ omitting the long node
\cite{littig}.

The second maximal chain is $s_0 s_2 s_1 s_2 s_0$, with \clr\
$(1,0)$. Its cup sequence is $(1,1,3,2,2)$, and hence it is not a \pds.

There is one more rigid element $\zeta =s_0 s_2 s_1 s_2 s_1 s_2
s_0$. Note that omitting the $s_2$ on the left yields another element
of \wtils ; hence $\zeta$ covers a pair and is not a chain.

\subsection{Conclusions}

The results of this section imply that Theorem~\ref{smooth} is true for
all chains.  More precisely:

%%todo: let's upgrade this to proposition.  
\begin{proposition} \label{chain} \marginpar{}

Let $X_{\lambda}$ be a chain. Then $X_\lambda$ is smooth if and only if it
is a \cpo , in which case $X_\lambda$ is either a projective space, a
quadric hypersurface of type $B_k/B_{k-1}$ for some $k$, or a symplectic
Grassmannian of type $C_2/A_1$. 

\end{proposition}

\section{Proof of the Smoothness Theorem~\ref{smooth}}

It is well known that a \cpo\ is smooth and every smooth affine
Schubert variety satisfies Poincar\'e duality over $\bz$.  Therefore,
to prove Theorem~\ref{smooth} it is only necessary to show that if
$X_{\lambda}$ satisfies Poincar\'e duality integrally then it is a
closed parabolic orbit.  Call $\lambda$ {\it admissible} if
\begin{enumerate}
\item  $\lam$ has a unique negative node
$s \in \stil$; 
\item  $s$ is a long node; 
\item $\aslam =-1$ (or $1-\alpha _0 (\lam )=-1$, if $s=s_0$).
\end{enumerate}
Note that if $X_\lambda$ satisfies \pd , then $\lambda$ is an \apal\
by Proposition~\ref{pdtest}.  In Section~\ref{s:chains} we have
classified all smooth chains.  Hence to finish the main proof it
suffices to prove the following key lemma:

\begin{lemma} \label{smooth1} \marginpar{}

Suppose $\lambda$  is an \adp . Then either (i) $X_\lambda$ is a \cpo ; or
(ii) $X_\lambda$ is a singular chain. 

\end{lemma}

We first dispose of the dominant and anti-dominant cases. 

\begin{lemma} \label{smooth2} \marginpar{}

Suppose $\lambda$  is admissible. Then 

\begin{enumerate}
\item [(a)] if $\lambda$  is dominant, then \xlam\ is a \cpo ;

\item [(b)] if $\lambda$  is anti-dominant and palindromic, then $G$ has type $G_2$
and $\lam =-\omega _2 ^\vee$. Hence \xlam\ is a singular chain. 
\end{enumerate}
\end{lemma}

\proof{ (a) If $\lambda$ is dominant and admissible, then $\azlam =2$,
and the assertion follows immediately from Proposition~\ref{domcpo}.
(b) Suppose $\lambda$ is antidominant, admissible, and
palindromic. Then $\lambda =-\omega _s ^\vee$, where $s$ is a long
node and $\omega _s ^\vee \in Q^\vee$. In particular, $s$ is not
minuscule. This rules out type $A$, since then all nodes are
minuscule. Since in all other types $\lambda$ cannot cover a fork, we
conclude that $s$ is a leaf node.  This eliminates types BCD at once,
since every leaf node is either short or minuscule. In type $E$ every
leaf node is either minuscule or forks too soon. In type $F_4$ the
long leaf node $s_1$ forks too soon.  This leaves type $G_2$ with
$\lambda =-\omega _2 ^\vee$. This element is a chain, and is
singular.}

\bigskip

Recall from Section~\ref{s:cpos} that $p_G$ is the number of
non-trivial closed parabolic orbits, and that we have shown $p_G$ is
just the number of connected subdiagrams of the affine diagram
containing $s_0$.  Recall also that $\lambda $ is \textit{overweight}
if it does not satisfy
\[\alpha _0 (\lambda ) \leq    \left\{ \begin{array}{ll}
 2   & \mbox{for all $\lambda$}\\
 1   & \mbox{if $\lambda$ is not dominant}
\end{array}
\right. \] An admissible palindromic cannot be overweight.   

\begin{lemma} \label{lem:admissible.pal.simply.laced} \marginpar{}

If $G$ is simply-laced, there are at most $p_G$ \adp s.

\end{lemma}

\proof{
Assume $\lambda$ is admissible and palindromic.  Let $s \in \dtil$
denote the unique negative node of $\lambda$. By
Lemma~\ref{smooth2}(b), we may assume that $\lambda$ has at least one
positive node. Note that ``positive node'' always refers to a node of
\cD , while ``negative node'' refers to a node of \dtil .  The proof
now proceeds type by type.

$A_n$: $p_G ={n+1 \choose 2}$. We may assume $n>1$, since the case $n=1$
was already settled in our study of chains. Note that $\lambda$ can have
at most two positive nodes \owt . Moreover if $s_i, s_j$ are the
positive nodes, with $i \leq j$, then $\alpha _i (\lambda) =1 =\alpha _j
(\lambda)$, where in the case $i=j$ this is to be interpreted as $\alpha
_i (\lambda) =2$. Now if $s_k$ is the negative node, then $i+j-k=0 \,
mod \, (n+1)$ (otherwise \nope ). Since there is a unique such $k$, and
$k \neq i,j$, this shows that there are at most ${n+1 \choose 2}$ \adp
s. 

\bigskip

$D_n$: $p_G =2n$. There are at most three positive nodes \owt . If there
are exactly three, then they are all minuscule and take the value 1 on
$\lambda$  \owt. Furthermore $s=s_{n-2} $ \fork . Hence $\lambda =\omega
^\vee _1 + \omega ^\vee _{n-1} + \omega ^\vee _n -\omega ^\vee _{n-2}$.

If there are two positive nodes $t,u$ then at least one of them, say
$t$, is minuscule \owt . If $u$ is also minuscule, then using the
characterization of the coroot lattice elements, one can check that
$\lambda$ has one of the following forms:

\[\lambda =    \left\{ \begin{array}{ll}
 \omega _1 ^\vee \pm (\omega ^\vee _{n-1} -\omega ^\vee _{n}) & \\
\omega ^\vee _{n-1} + \omega ^\vee _n -\omega ^\vee _1  & \mbox{($n$ even)}\\
\omega ^\vee _{n-1} + \omega ^\vee _n & \mbox{($n$ odd)}
\end{array}
\right. 
\] 

If $u$ is not minuscule, then $s \neq s_0$, $s$ is not minuscule, and
$\atlam =1=\aulam$ \owt. Furthermore, $t=s_1$ (otherwise $\lambda
\notin \cQ ^\vee$) and $s,u$ are adjacent (otherwise $\lambda$ covers
a fork). Hence $\lam =\omega _1 ^\vee \pm (w_k ^\vee -w_{k-1}^\vee)$,
where $2<k<n-1$. But $\omega _1 ^\vee -w_k ^\vee +w_{k-1}^\vee$ is a
non-palindromic skeleton of the quadric in \eqref{e:typeD.segments},
so we must have $\lam =\omega _1 ^\vee + w_k ^\vee -w_{k-1}^\vee.$

Hence there are at most $n-1$ \adp s with two positive nodes.

Suppose $\lambda$ has exactly one positive node $i$. Then we claim

\[\lambda =    \left\{ \begin{array}{ll}
 \omega ^\vee _i   & \mbox{if $i$ even, $i\neq n-1,n$}\\
 \omega ^\vee _i -\omega ^\vee _1   & \mbox{if $i$ odd, $i \neq 1,
 n-1,n$}\\
2\omega _1 ^\vee & \mbox{if $i=1$}\\
2\omega _i ^\vee & \mbox{if $n$ even, $i=n-1,n$}\\
2\omega _i ^\vee -\omega _1 ^\vee & \mbox{if $n$ odd, $i=n-1,n$}
\end{array}
\right. 
\] 
To prove one case of the claim, suppose $i=1$. Then the
negative node $s$ can only be $s_0$ or $s_2$ (otherwise either
$\lambda$ covers a fork or \nope ). This forces $\alpha _1 (\lambda)
=2$ (otherwise either \nope\ or $\lambda$ is overweight). But if
$s=s_2$ then $\lambda =(2,-1, \und )$, a non-palindromic skeleton of
the quadric. Hence $s=s_0$ and $\lambda =2 \omega _1 ^\vee$.

To prove another case, suppose $n$ odd and $i=n$. Then the negative
node $s$ must be $s_0, s_1$ or $s_{n-1}$ (otherwise $\lambda$ covers a
fork). If $s=s_{0}$ then $\alpha _n (\lambda) =0 \, mod \, 4$
(otherwise \nope ), and hence $\lambda$ is overweight. If $s=s_{n-1}$
then $\lambda =(\und , -1, a)$ for some $a>0$. Then $a+1 =0 \, mod \,
4$ (otherwise \nope ), and hence $\lambda$ is overweight.  Hence $j=1$b
and $\alpha _n (\lambda)$ is even (otherwise \nope ), and then $\alpha
_n (\lambda )=2$ \owt . Hence $\lambda =2 \omega _n ^\vee -\omega _1
^\vee$.

The remaining cases are left to the reader.  Thus there are at most
$1+n-1+n=2n$ \adp s, as desired.

\bigskip

$E_n$: $p_G =10$. For each $n=6,7,8$, we will show that there are at
most $n$ \adp s with one positive node, at most $10-n$ with two positive
nodes, and none with more than two positive nodes. 

\bigskip

$E_6$: There are at most three positive nodes (otherwise
$\lambda$ is overweight). If there are exactly three, then the negative node
must be $s_4$ and only one of the adjacent nodes is occupied (otherwise
$\lambda$ is overweight). But then $\lambda$ covers a fork, a
contradiction. So there are at most two positive nodes. 

Suppose there are exactly two. Then there are four possibilities:
$(1,0,0,0,0,1)$, $(0,0,1,-1,1,0)$, $(0,0,1,0,-1,1)$ and $ (1,0, -1, 0,
0,1)$. For example, suppose that $s_4$ is the negative node. Then the
positive nodes must be adjacent to it (otherwise $\lambda$ covers a
fork), and hence must each take the value 1 (otherwise $\lambda$ is
overweight). But this forces $s_3, s_5$ as the positive nodes (otherwise
\nope ). Hence $\lambda =(0,0,1,-1,1,0,0)$.

If there is just one positive node $t$, then for each choice of $t$
there is only one possibility for $\lambda$. For example, suppose the
positive node is $s_6$. Then the negative node must be $s_1$: For if
$s_0$ is negative, then $\alpha _6 (\lambda) =0 \, mod \, 3$ (otherwise
\nope ); but then $\lambda$ is overweight. The other four choices of
negative node all fork too soon. It follows that $\alpha _6 (\lambda)$
is even (otherwise \nope ), and hence $\alpha _6 (\lambda) =2$
(otherwise $\lambda$ is overweight). Hence $\lambda =(-1,0,0,0,0,2)$.

\bigskip

$E_7$: There are at most three positive nodes (otherwise $\lambda$ is
overweight). If there are exactly three, then the negative node must be
$s_4$ and the adjacent nodes $s_3,s_5$ must be zero (otherwise $\lambda$
is overweight). But then $\lambda$ covers a fork, a contradiction. So
there are at most two positive nodes. 

If there are exactly two, we find that there are three
possibilities: $(0,1,0,-1,1,0,0)$, $(0,1,0,0,-1,0)$,$(0,1,0,0,0,-1,1)$.

\bigskip

If there is just one positive node $t$, then for each choice of $t$
there is only one possibility for $\lambda$. For example,
suppose $s_4$ is positive. Then $\alpha _4 (\lambda) =1$ and the
negative node must be $s_3$ or $s_5$ (otherwise $\lambda$ is
overweight). But $s_5$ can't occur, since then \nope . Hence $\lambda
=\omega ^\vee _4 -\omega ^\vee _3$.

\bigskip

$E_8$: There are at most three positive nodes (otherwise $\lambda$ is
overweight). If there are exactly three, then $\lambda$ is still
overweight unless $\lam = \omega _1 ^\vee +\omega _2 ^\vee +
\omega _8 ^\vee -\omega _4 ^\vee .$ But then $\lambda$  covers a fork, a
contradiction. So there are at most two positive nodes.

If there are exactly two, then $\lambda$ must be either $(1,1,-1,\und)$
or $(0,1,1,-1,\und)$. If there is only one positive node $t$, then for
each such $t$ there is at most one possibility for $\lambda$. In all cases
one finds that the alternatives fork too soon or are overweight; details
are left to the reader. 
}

\bigskip

We can now prove Lemma~\ref{smooth1}.  In the simply laced case, we
know that every \cpo\ satisfies Poincar\'e duality and hence is
admissible and palindromic.
Lemma~\ref{lem:admissible.pal.simply.laced} implies the converse holds
also so Lemma~\ref{smooth1} holds.  

We now turn to the non-simply laced types. 

\bigskip 

$B_n$: $p_G=2n-2$. We will show that (1) there are at most $2n-2$ \apal s
that \spd , and (2) all other \apal s are singular skeletons of the
quadric $X_{(2, \und )}$ (and in particular, are chains). 

There are at most two positive nodes \owt . If there are exactly two,
then one of them is $s_1$ \owt . If the other is $s_2$ then $\lam
=(1,1,-1)$ \fork . This element is a singular skeleton of the
quadric. So suppose that the two positive nodes are $s_1, s_j$, where
$j>2$. Then we claim $\lambda =\omega ^\vee _1 +\omega ^\vee _j
-\omega ^\vee _{j-1}$.  To prove the claim, let $s=s_i$.  If $i \neq
2$ then $j=i \pm 1$ (otherwise $\lambda$ covers a fork). Then there
are two palindromic solutions: $\lam =(1, \und , -1,1,\und )$ and
$\lam =(1, \und , 1, -1, \und)$. In the second case $\lambda$ is a
singular skeleton of the quadric, hence a singular chain. If $i=2$
then $j=3$ and $\lam =(1,-1, 1, \und )$ (otherwise $\lambda$ covers a
fork, since $s_0$ is necessarily a zero node). Hence there are at most
$n-2$ \apal s that satisfy \pd .

Suppose there is one positive node $s_j$. Then we claim that

\[\lambda =    \left\{ \begin{array}{ll}
 \omega _j ^\vee   & \mbox{if $j$ even}\\
 \omega _j ^\vee -\omega _1 ^\vee   & \mbox{if $j$ odd, $j>1$}\\
2\omega ^\vee _1 & \mbox{if $j=1$}
\end{array}
\right. \]
If $j$ is even then $i$ is also even (otherwise \nope ) . Since $i\neq n$, this
forces $i=0$ \fork . Hence $\lam =\omega ^\vee _j$. 

If $j$ is odd and $j>1$, then $\alpha _j (\lam )=1$ \owt . Hence $i$ is
odd \owt . Since $i\neq n$, this forces $i=1$ \fork . 

Now suppose $j=1$. If $\alpha _1 (\lambda)$ is odd, then $i$ is
odd (otherwise \nope ), in which case $\lambda$ covers a fork. So $i$ is
even. Hence $\alpha _j (\lambda)$ is even (otherwise \nope ), forcing
$\alpha _j (\lambda) =2$ and $i=0,1$. If $i=1$ then $\lambda =(2,-1,\und
)$, a singular skeleton of the quadric. Hence $i=0$ and $\lambda
=2\omega ^\vee _1$. This completes the proof of our claim.

\bigskip

$C_n$: $p_G=n$. We must have $s=s_i$ for $i=0,n$. But if $i=n$ then
\nope . Hence $i=0$ and $\lambda$ is dominant. This forces $\lambda
=\omega ^\vee _i$ if $i<n$, and $\lambda =2\omega ^\vee _n$ if $i=n$, as
shown earlier.

\bigskip

$F_4$: $p_G =4$. If $\lambda$  has more than one positive node, then $\lam
=(1,-1,0,1)$, a singular chain. 

Now suppose there is one positive node $s_j$. For each of $j=2,3$, it is
easy to check that there is only one corresponding \adp , namely
$(-1,1,0,0)$ and $(0,-1,1,0)$. For $j=1,4$ there are two \adp s:
$(1,0,0,0)$ and $(1,-1,0,0)$ for $j=1$, and $(0,0,0,1)$ and $(-1,0,0,1)$
for $j=4$. The dominant classes are singular chains.

\bigskip

$G_2$: $p_G =2$. Suppose $\lambda$  is an \adp\ (and is not
anti-dominant). Then $\lambda$  has exactly one positive node $s_j$ \owt ,  . 
$j=1,2$. If $j=2$ there are three \adp s: $(-1,1)$,
$(-1,2)$ and $(0,1)$. The first two are singular chains. If $j=1$ there
are two: $(1,0)$ and $(1,-1)$. The first is a
singular chain. 

\bigskip

This completes the proof of Lemma~\ref{smooth1} and
Theorem~\ref{smooth}.

\bigskip

\noindent 
{\it Remark:} Let $I$ be a connected subgraph of \dtil\ containing
$s_0$. Then $S^+(\lam) =\cN (I)$ by Lemma~\ref{cpo2}. From this point
of view, the proof of Theorem~\ref{smooth} amounts to showing that (i)
If \xlam\ satisfies \pd , then $S^+( \lam) =\cN (I)$ for some (unique)
$I$, and (ii) for each $I$ there is a unique $\lambda$ such that $S^+(
\lam) =\cN (I)$ and \xlam\ satisfies \pd .

\section{The palindromy game II: Bruhat order and the coroot lattice}

In the characterization of the smooth \svars , it turned out (somewhat
surprisingly) that we only needed the weak order. For the palindromy
theorem, however, we will need more general Bruhat descents of the form
$\lam \da r\lam$, where $r$ is an affine reflection associated to a
non-simple root. As it happens, we will only need two kinds of such
reflections: The linear reflection $s_\beta \in W$ associated to a
positive root $\beta$, and the affine reflection $r_\beta =r_{1,\beta}$
associated to the affine root $(1,\beta)$, where again $\beta$ is a
positive root. In the spirit of the palindromy game, we will often refer
to such descents as ``moves''. Whenever possible we describe these moves
$\lam \da s_\beta \lam$, $\lam \da r_\beta \lam$ in terms of the Dynkin
diagram \dtil .

Note that if $r_1$ and $r_2$ are any two distinct reflections (linear or
affine), then by elementary geometry we have that $r_1 \lam =r_2 \lam$
$\Rightarrow$ $r_1 \lam =\lam =r_2 \lam$. Hence if $\lam \da r_1 \lam$
and $\lam \da r_2 \lam$, it follows that $\lambda$  covers a pair.

\subsection{$\beta$-positive and $\beta$-negative pairs}
                                    
%Throughout this section symbols $\alpha, \beta$ denote positive
%roots. In the simply-laced case, all roots are regarded as long.

%The ordinary reflection associated to $\beta$ is denoted $s_\beta$. Also
%associated to $\beta$ is the affine reflection $r_\beta$ across the line
%$\beta =1$.

Let $\beta$ be a positive root and let $\lambda \in Q^\vee$.  We want
to determine when the reflections $s_{\beta }$ and $r_{\beta}$ lower
the $S$-length of $\lambda$ by 1. Note that $r_\beta \lambda =s_\beta
\lambda +\beta ^\vee.$ For all $\alpha \in \Phi^{+}$ we have the
formulas
                                                                                
$$\alpha (s_\beta \lambda)=\alpha (\lambda )-\alpha (\beta ^\vee) \beta
(\lambda)$$
                                                                                
\noindent and
                                                                                
$$\alpha (r_\beta \lambda)=\alpha (\lambda )+\alpha (\beta ^\vee)(1-
\beta (\lambda)).$$
                                                                                
In order to evaluate the change in $S$-length after reflection, it is
convenient to partition the positive roots into the following sets:
                                                                                
\bigskip                                                                        
                                                                                
1. $\beta$ itself;
                                                                                
2. The $\beta$-null roots; i.e., $\{\alpha : \alpha (\beta ^\vee)
   =0\}$. Thus $s_\beta \alpha =\alpha$.
                                                                                
3. $\beta$-positive pairs $\alpha , \alpha ^\prime$: These are
 characterized by $s_\beta \alpha =\alpha ^\prime$.
                                                                                
4. $\beta$-negative pairs $\alpha , \alpha ^\prime$: These are
   characterized by $s_\beta \alpha =-\alpha ^\prime$.
                                                                                
\bigskip                                                                        
                                                                                
The apparent symmetry of $\alpha$ and $\alpha ^\prime$ is
misleading. In the $\beta$-positive case, we will always take $\alpha$
to have $\alpha (\beta ^\vee) <0$. Then $\alpha ^\prime (\beta
^\vee)>0$, and $\alpha + k\beta =\alpha ^\prime$, where $k=-\alpha
(\beta ^\vee) =1,2$, or $3$.
                                                                                
In the $\beta$-negative case $\alpha ^\prime (\beta ^\vee)= \alpha
(\beta ^\vee) $, and $\alpha +\alpha ^\prime =k \beta$. When $k=1$ there
is no way to distinguish $\alpha, \alpha ^\prime$, but for $k=2,3$ we
can and will always choose $\alpha$ so that $\beta -\alpha$ is a
positive root. Then $\beta -\alpha ^\prime$ is a negative root.
                                                                                
\subsection{Linear reflections}
                                               
Let $\lambda$ be an element of the coroot lattice, and fix a positive
root $\beta$. Note that $s_\beta \lambda =\lambda $ $\Leftrightarrow$
$\beta (\lambda) =0$.  We say that the {\it opposite sign condition}
is satisfied on $(\beta ,\lambda )$ if for every $\beta$-negative pair
$\alpha, \alpha ^\prime$ the values $ \alpha (\lambda), \alpha ^\prime
(\lambda)$ have opposite sign. In particular, both are nonzero.

\begin{proposition} \label{p:linearmove}
Let $\beta \in \Phi^{+}$ and $\lambda \in Q^\vee$.
\begin{enumerate}
\item [(a)] $\el ^S (s_\beta \lambda) < \el ^S (\lam) $
$\Leftrightarrow$ $\beta (\lambda ) <0$.
\item [(b)] $\lam \da s_\beta \lam$ $\Leftrightarrow$ $\beta (\lam )<0$
  and the \osc\ is satisfied.
\end{enumerate}
\end{proposition}                                                                                                             
\begin{proof}
Note that application of a linear reflection $s_\beta$ to $\lambda$
can only affect the $q$ term in the $S$-length formula
\eqref{e:length.fn}.  Furthermore the $\beta$-null positive roots and
the $\beta$-positive pairs contribute zero to $\Delta q$. Since we
always have $\beta (s_\beta \lambda)=-\beta (\lam )$, the change in
length will be determined by what happens on the $\beta$-negative
pairs $\alpha, \alpha ^\prime$.

Now suppose $\beta (\lambda) <0$ and $\alpha, \alpha^{\prime}$ are a
$\beta$-negative pair. Since $\alpha + \alpha ^\prime$ is a positive
multiple of $\beta$, the values $\alpha (\lam ), \alpha ^\prime (\lam
)$ either have opposite sign, are both negative, or a negative and a
zero. The pairs with opposite sign contribute zero to $\Delta q$,
while in the other two cases we have respectively $\Delta q =2$,
$\Delta q =1$. This proves $\Leftarrow$ in (a), and also (b).

If $\beta (\lam ) >0$, then $\el ^S (s_\beta \lam )> \el ^S (\lam
)$ (substitute $s_\beta \lam$ for $\lambda$  and apply the previous
case). This yields $\Rightarrow$ in (a), completing the proof of the
theorem. 
\end{proof}
\bigskip

The following application will be particularly useful. Given a fixed
$\lambda$  and nodes $s,t \in \cD$ with opposite sign, let $I$ denote the
unique minimal path between them, regarded as a subgraph of \cD . If all
interior nodes of $I$ vanish on $\lambda$, we say that $s$ and $t$ are {\it
  linked} by $I$. 

\begin{lemma} (Linear $ABC$-moves) \label{amove} \marginpar{}
Suppose that nodes $s,t$ have opposite sign, with $s$ the negative
node, and they are linked by a subgraph $I$ of type $A$, $B$, or $C$.
Then if any one of the following conditions holds, $\lambda$ covers a
pair.
\begin{enumerate}
\item [(a)] $I$ has type $A,B$ or $C$, and $\alpha _s (\lam ) + \alpha
_t (\lam ) <0$.  Furthermore, if $I$ has type $B$ or $C$, the positive
node $t$ is required to be the minuscule node of $I$.

\item [(b)]$I$ has type $BC$, the positive node $t$ is the minuscule node of
$I$, $\astlam >0$, and $2 \aslam + \atlam <0$.
\item [(c)]$I$ has type $BC$, the negative node $s$ is the minuscule node of
$I$, $\astlam <0$ and $\aslam +2 \atlam \neq 0$.  
\end{enumerate}
\end{lemma}

\proof We will find a non-simple $\beta \in \Phi _{I} ^+$ such that (i)
$\lam \da s_\beta \lam $. It then follows that $\lam$ covers a pair.
Since the $\beta$-negative pairs all lie in $\Phi _{I} ^+$, we may as
well assume $I =\cD$; i.e., that $\Phi$ itself has type $A_n, B_n, C_n$
and $\{s,t\}=\{s_1,s_n\}$. We will prove part (a) in detail and sketch
the rest.

\vspace{.2in}
\noindent 
\textbf{Case} (a): Let $\beta =\alpha _1 + \alpha _2 + \ldots +\alpha
_n$. Then by assumption $\blam <0$, and we claim that the \osc\ is
satisfied. Let $\alpha , \alpha ^\prime$ be a \bnp . If $\beta$ is
long, or $\beta $ is short and the pair is also short, then $\alpha +
\alpha ^\prime =\beta$ and it is clear that one of the two contains
$\alpha _1$ and the other contains $\alpha _n$, proving our claim. In
particular, this settles type $A$.

In type $B_n$ we have $\beta=e_{1}$, which is a short root, and there
are long \bnp s $e_1 -e_i, e_1 + e_i$. In that case we have $\alpha +
\alpha ^\prime =2\beta$. Thus each root of the pair contains $\alpha
_1$. On the other hand $\alpha ^\prime =\beta + (\beta -\alpha)$, where
$\beta -\alpha$ is a positive root. It follows that $\alpha ^\prime$
contains $\alpha _n$ twice, and since we are assuming $\alpha _n
(\lambda ) <0$, the \osc\ is satisfied as required.

In type $C_n$ we have $\beta =e_1 + e_n$, which is again short. There is
one long \bnp\ $2e_1, 2e_n$, and again we have $\alpha +\alpha ^\prime =
2\beta$. Here $\alpha ^\prime =2e_1 =\alpha _0$, and so contains $\alpha
_1$ twice. The \osc\ follows as before. 

\vspace{.2in}
\noindent 
\textbf{Case} (b):  
Take as $\beta$ the smallest root containing $\aslam$ twice. In type
$B_n$, $\beta =\alpha _1 +\ldots \alpha _{n-1} + 2 \alpha _n =e_1
+e_n$. In type $C_n$, $\beta =\alpha _0$. In each case it is easy to
check the \osc .

\vspace{.2in}
\noindent 
\textbf{Case} (c):   
If $\aslam + 2\atlam >0$, use the $\beta$ of part (a). If $\aslam +
2\atlam <0$, use the $\beta$ of part (b).

\subsection{Affine reflections}

In this section we give necessary and sufficient conditions for $\lam
\da r_\beta \lam$ and $\lam \uparrow r_\beta \lam $. Note that $r_\beta
\lam =\lam $ $\Leftrightarrow$ $1 -\blam =0$.

\subsubsection{The positive and negative pair conditions}

Suppose $\lambda \in Q^\vee$.  We say that $\beta \in \Phi ^+$
satisfies the {\it positive pair condition} if either:

\begin{enumerate}
\item [(i)] $1-\blam <0$ and whenever $\alpha, \alpha ^\prime$ is a
$\beta$-positive pair, either $\alpha (\lambda) >0$ or $\alpha ^\prime
(\lambda ) \leq 0$; or
\item [(i)*] $1-\blam >0$ and whenever $\alpha, \alpha ^\prime$ is a
$\beta$-positive pair, either $\alpha (\lambda) <0$ or $\alpha ^\prime
(\lambda ) >1$.
\end{enumerate}
We say that $\beta \in \Phi ^+$ satisfies the {\it negative pair
condition} if either:
\begin{enumerate}
\item [(ii)] $1-\blam <0$ and either $\beta$ is long, or $\beta$ is short and
whenever $\alpha, \alpha ^\prime$ is a long $\beta$-negative pair,
either $\alam \leq 0$ or $\alpha ^\prime (\lambda) \leq 0$; or
\item [(ii)*] $1-\blam >0$ and either $\beta$ is long, or $\beta$ is short and
whenever $\alpha, \alpha ^\prime$ is a long $\beta$-negative pair,
either $\alam \geq 2$ or $\alpha ^\prime (\lambda) \geq 2$.
\end{enumerate}
If we wish to refer only to a specific $\beta$-positive or
$\beta$-negative pair we say that $\alpha$, $\alpha ^\prime$ satisfies
the \ppc\ or \npc . In fact we will be concerned almost exclusively
with the case $1-\blam <0$; the case $1-\blam >0$ is included for
completeness.

\begin{proposition} \label{affinemove} \marginpar{} 
Let   $\beta \in \Phi ^+$ and $\lambda \in Q^{\vee}$.  

\begin{enumerate}
\item [a)] $\el ^S (r_\beta (\lam )) < \el ^S (\lam )$ $ \Leftrightarrow $ $
1-\blam <0$.
\item [b)] 
Suppose $1-\blam <0$. Then 
$\lam \da r_\beta \lam $ $\Leftrightarrow
$ the \ppc\ and the \npc\ are satisfied. 
\item [c)] The analogous statements hold for $1-\blam >0$. 
\end{enumerate}
\end{proposition}

If there exists a $\beta$ such that $\lam \da r_\beta \lam $, we say
$\lambda $ \textit{has an affine move}.  In particular, if $\beta \neq
\alpha_{0}$, then $\lambda$ cannot be palindromic.

\proof For $m \in \bz$ let 

\[f(m)=    \left\{ \begin{array}{ll}
  -m  & \mbox{if $m\leq 0$}\\
  m-1  & \mbox{if $m>0$}
\end{array}
\right. \]

Then 

$$\el ^S (\lam )=\sum _{\alpha \in \Phi ^+} f(\alam ).$$ 

We will analyze the effect of $r_\beta$ on $\el ^S$ by considering the
four types of roots listed above separately. For any subset $\Gamma
\subset \Phi ^+$, we write $\Delta _\Gamma$ for
the contribution of $\Gamma$ to the change in $\el ^S$; more precisely: 

$$\sum _{\alpha \in \Gamma} f(\alpha (r_\beta \lambda ))
=\sum _{\alpha \in \Gamma} f(\alpha (\lambda )) + \Delta _\Gamma .$$
First observe that

$$\beta (r_\beta \lambda)=2-\beta (\lambda),$$

\noindent while 

\[f(2-m)=    \left\{ \begin{array}{ll}
 f(m)-1   & \mbox{if $1-m <0$}\\
  f(m)  & \mbox{if $1-m=0$}\\
f(m) +1 & \mbox{if $1-m>0$}
\end{array}
\right. 
\]
Hence 

\[ \Delta _{\{\beta\}}=   \left\{ \begin{array}{ll}
 -1   & \mbox{if $1-\blam <0$}\\
  1  & \mbox{if $1-\blam >0$}
\end{array}
\right. 
\] 
Second, if $\Gamma$ is the set of positive roots $\alpha$ such that
$s_\beta (\alpha )=\alpha$, then $\Delta _\Gamma=0$.  Therefore, the
proposition then follows from the lemma below:

\begin{lemma}  \marginpar{}
If $\alpha, \alpha ^\prime$ is a $\beta$-positive ($\beta$-negative)
pair, then

\[
\begin{array}{lcl}
1-\blam <0 &	 \Rightarrow  &	\Delta _{\{\alpha, \alpha ^\prime\}} \leq 0
\\
1-\blam >0 &	 \Rightarrow  &	\Delta _{\{\alpha, \alpha ^\prime\}} \geq 0.
\end{array}
\]
In each case $\Delta _{\{\alpha, \alpha ^\prime\}} =0$
$\Leftrightarrow$ the positive (negative) pair condition holds for
$\alpha, \alpha ^\prime$.
\end{lemma}

\proof   We have

$$\alpha (r_\beta \lambda)=\alpha ^\prime (\lambda) + \alpha (\beta ^\vee)$$

$$\alpha ^\prime (r_\beta \lambda)=\alpha (\lambda) -\alpha (\beta ^\vee).$$
Set $a=\alpha (\lambda)$, $a^\prime =\alpha ^\prime (\lambda)$, and $k=
-\alpha (\beta ^\vee) =1,2,3$. Thus we need only compute the effect of
the transformation $(a,a^\prime) \mapsto (a-k, a^\prime +k)$ on $f(a) +
f(a^\prime)$. Taking $k=1$ for simplicity, we find that

\begin{center}
\begin{tabular}{l}
$\Delta _{\{\alpha , \alpha ^\prime \}}=0$ if $a>0$ and $a^\prime >1$,
or $a<0$ and $a^\prime \leq 0$, or $a=0$ and $a^\prime =1$;
\\
$\Delta _{\{\alpha , \alpha ^\prime \}}<0$ if $a=0$ and $a^\prime >1$,
or $a<0$ and $a^\prime \geq 1$;
\\
$\Delta _{\{\alpha , \alpha ^\prime \}}>0$ if $a>0$ and $a^\prime
\leq1$, or $a=0$ and $a^\prime < 0$.
\end{tabular}
\end{center}

If $\alpha ,\alpha'$ are a $\beta$-positive pair, the lemma now
follows easily on inspection, making use of the fact that $\beta
+\alpha =\alpha ^\prime$. Note, for example, that if $1-\blam <0$ and
$a>0$ then automatically $a^\prime >1$ and hence $\Delta _{\{\alpha ,
\alpha ^\prime \}}=0$. If $k>1$ a similar argument applies, making use
of the fact that $k\beta +\alpha =\alpha ^\prime$.

If $\alpha, \alpha ^\prime$ is a $\beta$-negative pair, then $\alpha
(\beta ^\vee)=\alpha ^\prime (\beta ^\vee)$ and
$$\alpha (r_\beta \lambda) =\alpha (\beta ^\vee) -\alpha ^\prime
(\lambda )$$
$$\alpha ^\prime (r_\beta \lambda) =\alpha (\beta ^\vee) -\alpha 
(\lambda ).$$
Let $k=\alpha (\beta ^\vee) =\alpha ^\prime (\beta ^\vee) =1,2,3$.  
Then we need only compute the effect of the
transformation $(a, a^\prime) \mapsto (k-a, k-a^\prime)$ on $f(a) +
f(a^\prime)$. Since $f(1-a) =f(a)$ for all $a$, when
$k=1$ we find $\Delta _{\{\alpha , \alpha ^\prime\}}=0$. If $k>1$ and
$1-\blam <0$ we find

\[\Delta _{\{\alpha, \alpha ^\prime\}}   \left\{ \begin{array}{ll}
  =0  & \mbox{if $\alam \leq 0$ or $\aplam \leq 0$}\\
  <0  & \mbox{otherwise}
\end{array}
\right. 
\]
Note that if, say, $\alam \leq 0$, then $\aplam \geq k\blam \geq 2k$.
The case $k>1$ and $1-\blam >0$ is similar.  Since $k>1$
$\Leftrightarrow$ $\beta$ is short and $\alpha, \alpha ^\prime$ are
long, this completes the proof of the lemma.

\subsubsection{Graph-splitting and affine ABC-moves}

Two types of affine moves will be particularly useful. We continue to
fix $\lam \in \cQ ^\vee$.

\begin{lemma} (Graph-splitting moves) \label{gsm} \marginpar{}
Let $I$ be a proper connected subgraph of \cD , and let $\alpha _I$
denote the highest root of $\Phi _I$. If $\alpha _I (\lam )\geq 2$ and
$\alpha _I$ satisfies the \ppc , then $\lam \da \alpha _I (\lam )$ and
$\lambda$  covers a pair. 
\end{lemma}

\proof Note that $\alpha _I$ is long in the sub-root system $\Phi _I$,
whether it is long in $\Phi$ or not. Since the \npc\ involves only
roots in $\Phi _I$, it will automatically be satisfied in this
case. If the \ppc\ also holds, then $\lam \da r_{\alpha_{I}} \lam$ by
Proposition~\ref{affinemove}. Since there is also a left descent $\lam
\da s\lam$ for some $s \in \stil$, it follows that $\lambda$ covers a
pair. (Note that $\alpha _I \neq \alpha _0$, since $I$ is a proper
subgraph.)

\bigskip 

We call descents $\lam \da r_{\alpha_{I}} \lam$ as above {\it
graph-splitting moves}, since we are splitting the Dynkin graph \cD\
into $I$ and the components of its complement.

\begin{lemma}(Affine $A$-moves) \label{affineamove} \marginpar{}
Suppose that $s_0$ is linked to a nonzero node $t$ by a proper type
$A$ subgraph $I$, and that $1-\azlam$, $\atlam$ have opposite sign. If
$1-\azlam +\atlam <0$, then $\lambda$ covers a pair.
\end{lemma}

\proof Let $\beta =\alpha _0 -\sum _{s \in I: s \neq
s_0} \alpha _s $. Then $\beta (\lam ) >1 $ and $\beta$ is a long root. Hence
the \npc\ holds. Moreover, the \ppc\ is also automatically satisfied:
For suppose $\alpha , \alpha ^\prime$ is a \bpp , so that $\beta + \alpha
=\alpha ^\prime$. Then $\alam =\atlam$ and $\aplam =\azlam$. Hence
either $\alam >0$ or $\aplam \leq 0$. Then $\lam \da r_\beta \lam$ and
$\lambda$  covers a pair. 

\bigskip

We call the descent $\lam \da r_\beta \lam$ constructed above an {\it
  affine A-move}.  There are similar but less productive moves in the
  B,C cases; the following lemma will suffice for our purposes:

\begin{lemma} \label{affinebcmove} \marginpar{}
(Affine $BC$-moves) Suppose $s_0$ is a negative node of $\lambda$,
linked to a positive node $t$ by a subgraph $I$ of type $BC$. If
$1-\azlam +\atlam <0$ and $1-\azlam +2 \atlam \neq 0$, then $\lambda$
covers a pair.
\end{lemma}

\proof We may assume $\lambda$  is dominant. There are three cases: 

\bigskip

(1) $\Phi$ has type $B_n$ and $t=s_n$;

(2) $\Phi$ has type $C_n$ and $t=s_i$, with $i<n$;

(3) $\Phi$ has type $F_4$ and $t=s_3$.

\bigskip

Suppose $1-\azlam +2 \atlam >0$ (this rules out Case 3). Then $\lambda$  has
no other positive nodes. Let $\beta =\alpha _0 -\sum _{s \in I: s \neq
s_0} \alpha _s $, as in the previous lemma. Then $\blam >1$, and the
\ppc\ is satisfied because $\beta$ is the highest root containing
$\alpha _t$ once. We also have (i) $\beta$ is a short root; and (ii) if
$\alpha , \alpha ^\prime$ is a long \bnp , then $\alpha$ does not
contain $\alpha _t$ (compare the proof of Lemma~\ref{amove}), and
hence $\alam =0$. Hence the \npc\ is satisfied, $\lam \da r_\beta
\lam$, and $\lambda$  covers a pair. 

Now suppose $1-\azlam +2 \atlam <0$. Let $J=\Phi - \{s_n\}$ in Case 1,
$J= \{s_{i+1}, \ldots s_n \}$ in Case 2. Then $\alpha _J (\lam )=\azlam
-2\atlam \geq 2$ (note that $m_s (\alpha _J) =m_s (\alpha _0)$ for all
$s \in J$). Since the unique neighbor node of $J$ in $S$ is positive,
$\lambda$  covers a pair by Lemma~\ref{gsm}. 

In Case 3 let $\beta$ be the highest root containing $\alpha _3$
twice. Then $\blam \geq 2$ and the \ppc\ is satisfied. Moreover $\beta
=1222=e_1 -e_4$ is a long root,\footnote{Here we are following
Bourbaki notation so
$1222=\alpha_{1}+2\alpha_{2}+2\alpha_{3}+2\alpha_{4}$.}  so the \npc\
is satisfied. Hence $\lam \da r_\beta \lam$ and $\lambda$ covers a
pair.

\section{Proof of the Palindromy Theorem~\ref{pal}}

The proof considers the anti-dominant, dominant, and ``mixed'' (i.e.,
neither dominant nor anti-dominant) cases separately, proceeding by a
process of elimination based on the Palindromy Game II. The reader
should keep at hand the list of chains (\S 7).  The spiral varieties
were shown to be palindromic in \cite{mfilt}; see \S 12. The
exceptional case $\lambda =(3,0,-1)$ in type $B_{3}$ was shown to be
palindromic in Corollary~\ref{funnypal}, therefore, it remains to
prove the ``only if'' part of the theorem.  We will make frequent use
of the \gsm s and linear/affine ABC-moves introduced in \S 9.

\subsection{Anti-dominant case}

\begin{theorem}  \marginpar{}
If $\lambda$ is nonzero and anti-dominant, then $X_\lambda$ is
palindromic in precisely the following cases:

\begin{enumerate}
\item [(i)] $\Phi$ has type $A_n$ and $\lambda$ is a spiral class of the form
$-k(n+1) \omega ^\vee _i$ for $i=1,n$ and $k\geq 1$;
\item [(ii)] $\Phi$ has type $C_n$ or $G_2$ and $\lambda =-\alpha _0 ^\vee$. In
these cases $\lambda$  is a chain. 
\end{enumerate}
\end{theorem}

\proof Suppose $\lambda$ is palindromic. There is a unique $s \in S$
such that $\alpha _s (\lambda) <0$, since otherwise $\lambda$ covers a
pair. Thus $\lambda =-m \omega ^\vee _s$ for some $s \in S$ and
$m>0$. If $s$ is not a leaf node, firing it shows that $\lambda$
covers a fork. This contradicts palindromy except in type $A$. In type
$A$ we have $\lambda =(\und , -m , \und)$ with $-m$ in the $i$-th
position, where $1<i<n$ and $m \geq 2$ (otherwise \nope ). Then
$\lambda \da \mu =(\und , -m , m ,-m , \und )$. Then there is a
graph-splitting move $\mu \da r_{\alpha _i} \mu $. Hence $\lambda$
covers a trident, a contradiction.  Thus $s$ is a leaf node.

In type $A$ we then have $s=s_1, s_n$ and $m$ divisible by $n+1$, so
that $\lambda$ is spiral as claimed. For the remainder of the proof we
assume $G$ is not of type $A$.

 We claim $m=1$; in other words, $\lambda =-\omega _s ^\vee$ for some
leaf node $s$. In particular, $s$ is not minuscule.  To prove the
claim, suppose $m>1$, and let $a \in S$ denote the unique node
adjacent to $s$. If $s$ is a long node, then firing it yields $\mu
=s\lam =m\omega ^\vee _s -m \omega ^\vee _a$. Then there is a \gsm\
$\mu \da r_{\alpha _s} \mu$ and hence $\lambda$ covers a fork. To see
this we need only check the \ppc . But if $\alpha _s + \alpha =\alpha
^\prime$, then since $s$ is a leaf node we have $m_s (\alpha ^\prime)
<m_a (\alpha ^\prime)$, and hence $\alpha ^\prime (\mu ) \leq 0$. A
similar argument works if $s$ is short (note this only happens in type
$B_n$ with $s=s_n$, and in type $G_2$ with $s=s_1$). This proves the
claim.

The theorem now follows immediately in types $CDEF$: In type DEF every
leaf node is either minuscule or forks too soon, while in type $C$ we
can only have $s=s_1$. It remains to consider type $B$ and $G_2$. 

In type $B_n$ there is one non-minuscule extreme node $s=s_{n}$. Note that
$n$ is necessarily even, since otherwise $\omega _n ^\vee$ is not in the
coroot lattice. In particular $n \geq 4$. Then

$$-\omega _n ^\vee =(\und ,-1) \da (\und , -2,1) \da (\und ,-2, 2,-1)$$ 

\noindent so that now $(\und ,-2,2,-1)$ covers the pair $\mu =(\und,
-2,2,0,-1)$ and $\eta =(\und , -2, 0, 1)$. Furthermore $\mu$ in turn
covers the pair $\mu _1 =(\und, -2, 2,0,0,-1)$ (or $(2,0,0,-1)$ if
$n=4$) and $\mu _2=(\und , -2,2,-2,1)$. Now let $\beta =\alpha _{n-1}
+ 2 \alpha _n$. Then $\beta$ is a long root with $\beta (\eta)=2$. The
\ppc\ for is immediately verified, so $\eta \da r_\beta \eta $.
Furthermore, $\eta$ differs from $\mu_{2}$ by a simple reflection.  It
follows that $\lambda$ covers a scepter, a contradiction.

In type $G_2$, we have $-\omega _2 ^\vee =-\alpha _0 ^\vee $, which is a
chain. On the other hand, $-\omega _1 ^\vee$ forks too soon (see the
Hasse diagram in \S 13).

\subsection{Dominant case}

Throughout this section we assume $\lambda \in \cQ ^\vee$ is
nontrivial and dominant. In particular, $\lam \da s_0 \lam$.

\begin{theorem} \label{paldom} \marginpar{}
Suppose $\lambda$ is dominant and nonzero. Then $X_\lambda$ is palindromic in
precisely the following cases:
\begin{enumerate}
\item [a)] $\alpha _0 (\lambda )=2$, in which case $X_\lambda$ is a \cpo.
\item [b)] $\Phi$ has type $A_n$ and $\lambda$ is a spiral class of the form
$k(n+1) \omega ^\vee _i$ for $i=1,n$ and $k\geq 1$.
\item [c)] $\Phi$ has type $G_2$ and $\lambda =\omega _1 ^\vee$. In this case
$\lambda$  is a chain. 
\end{enumerate}
\end{theorem}

\proof The theorem is trivial in type $A_1$, so from now on we exclude
that case. If $\azlam =2$ then $X_{\lambda}$ is a \cpo\ by
Proposition~\ref{domcpo}, while spiral classes are palindromic as
discussed in \S~\ref{s:spiral.varieties}. So fix a dominant nonzero
palindromic $\lambda$; we must show that one of the three conditions
holds.

Let $t$ be a positive node linked to $s_0$ by $I$.

\bigskip

Case 1: Suppose we can take $I$ of type $A$. Then we must have
$1-\alpha _0 (\lam ) + \alpha _t (\lam )\geq 0$ (otherwise there is an
affine $A$-move showing $\lambda$ covers a pair). Hence $(m_t -1)
\alpha _t (\lam ) \leq 1$, and if equality holds then there are no
other positive nodes.

1a: Suppose $m_t >1$. Then $m_t =2$ and $\lam =\omega ^\vee _t$; in
particular $\alpha _0 (\lam )=2$.

1b: Suppose $m_t =1$ (i.e., $t$ is minuscule) and there is one other
positive node $u$. Then $u$ must also be minuscule, forcing $\Phi$ of
type $A$, $D$, or $E_6$; in particular, $\Phi$ is simply-laced. Then
$u$ is also connected to $s_0$ by a type $A$ subgraph with all
interior nodes vanishing, so by Lemma~\ref{affineamove} we must have
$1-\alpha _0 (\lam ) + \alpha _u (\lam )\geq 0$. It follows that $\lam
=\omega ^\vee _t +\omega ^\vee _u$ (with $t,u$ minuscule); in
particular $\alpha _0 (\lam )=2$.

1c: Suppose $\lam =m \omega ^\vee _t$ with $t$ minuscule. Then $m \neq
1$ (otherwise (\nope ). If $m=2$ then $\alpha _0 (\lam) =2$, so
suppose $m>2$. If $\Phi$ is not of type $A$, then $t$ is not adjacent
to $s_0$, and the node $a$ adjacent to $s_0$ has $m_a =2$. Hence
firing $s_0$ yields $\mu =s_0 \lam$ with $\alpha _0 (\mu) =2(1-m) +m
=2-m<0$.  Note that $\alpha _0$ satisfies the \osc\ on $\mu$: For if
$\alpha _0 =\alpha + \alpha ^\prime$, then each of $\alpha, \alpha
^\prime$ must contain $\alpha _a$ once. Thus if $\alpha$, say,
contains $\alpha _t$, we have $\alpha (\lam)=1$ and $\alpha ^\prime
(\lam) =1-m <0$. Hence there is a linear move $\mu \da s_{\alpha _0}
\mu$ by Proposition~\ref{p:linearmove}, showing that $\lambda$ covers
a fork.

Now suppose $\Phi$ has type $A$. If $t$ is not adjacent to $s_0$, a
similar argument shows that $\lambda$ covers a trident, again
contradicting palindromy. Finally, if $t=s_1,s_n$ then $n+1$ divides $m$
(otherwise \nope ), and hence $\lambda =k(n+1) \omega ^\vee _i$ for $i=1,n$ and
$k\geq 1$.

\bigskip

This completes the proof of Theorem~\ref{paldom} in Case 1. In particular, the
theorem is now proved in the simply-laced case.

\bigskip

Case 2: $I$ has type $C$. This only happens when $\Phi$ itself has
type $C$ and $t=s_i$, $i<n$. Then $1-\alpha _0 (\lam ) + 2\alpha _t
(\lam) \neq 0$, since $\alpha _0$ takes only even values in type
$C$. By Lemma~\ref{affinebcmove} we must have $1-\alpha _0 (\lam ) +
\alpha _t (\lam )\geq 0$ (otherwise $\lam$ covers a pair). Then $\lam
=\omega ^\vee _t$ and $\alpha _0 (\lam )=2$ as in Case 1a.

\bigskip

Case 3: $I$ has type $B$. This can only happen in two ways: 

3a: $\Phi$ itself has type $B_n$ and $t=s_n$. Then $\alpha _1 (\lam
)=0$, because otherwise we can take $t=s_1$ in case 1b, a
contradiction. Hence $\alpha _0 (\lam) =2\alpha _n (\lam)$ and
$1-\alpha _0 (\lam ) + 2\alpha _n (\lam) =1 \neq 0$.  Thus, as in Case
2, we conclude $\lambda =\omega ^\vee _n$ and $\alpha _0 (\lam )=2$.

3b: $\Phi$ has type $F_4$ and $t=s_3$. Then $\lambda$ covers a pair by
Lemma~\ref{affinebcmove}, a contradiction.

\bigskip

Case 4: $I$ does not have type $ABC$. This can happen in three ways:

3a: $\Phi$ has type $C_n$ and $t=s_n$. As in Case 1c, we conclude that
$\lam =2\omega _n ^\vee$ and $\alpha _0 (\lam )=2$.

3b: $\Phi$ has type $F_4$ and $t=s_4$. Thus $\lambda =m\omega _4
^\vee$ for some $m>0$; we will show that $m=1$ and hence $\azlam
=2$. We use only the following facts: (i) $m_4 =2$; and (ii) if $m_4
(\alpha) =1$, then $\alpha $ is short.  Now suppose $m>1$, and let
$\beta$ denote the maximal positive root with $m_4(\beta) =1$. Then
$\beta $ is short, and $1-\beta (\lam )<0$. If $\alpha, \alpha
^\prime$ is a short $\beta$-positive pair, then $\beta + \alpha
=\alpha ^\prime$. Hence $\alpha$ contains $\alpha _4$ (by the
maximality of $\beta$) and $\alam >0$. Furthermore, there are no long
$\beta$-positive pairs $\alpha, \alpha ^\prime$. For in that case
$2\beta +\alpha =\alpha ^\prime$, and hence $\alpha$ does not contain
$\alpha _4$. But $\beta + \alpha$ is a root, so this contradicts the
maximality of $\beta$. Hence the \ppc\ is satisfied.  If $\alpha ,
\alpha ^\prime$ is a long \bnp , then $\alpha +\alpha ^\prime
=2\beta$. It follows from (i) and (ii) that one of $\alam , \aplam $
is zero. Hence the \npc\ is also satisfied. Then $\lam \downarrow
r_\beta \lam$ and $\lambda$ covers a pair by
Proposition~\ref{affinemove}, a contradiction.

\bigskip

{\it Remark}: Although it is not necessary to do so, one can easily
write out the roots used above explicitly, using the tables in
\cite{bourbaki}: We have $\beta =1231$. There are four $\beta$-positive
pairs, with the $\alpha$ of the pair given by the graph roots containing
$\alpha _4$: $0001$, $0011$, $0111$, $1111$. There are two long
$\beta$-negative pairs: $(1342, 1120)$ and $(1242, 1220)$.

\bigskip

3c: $\Phi$ has type $G_2$ and $t=s_1$. If $\alpha _2 (\lam )>0$, then $\lam
=\omega ^\vee _2$ by Case 1a, and $\alpha _0 (\lam )=2$. It remains to
show that if $\lam =m \omega ^\vee _1$, then $m=1$. Let $\beta = \alpha
_1 + \alpha _2 $ (the highest root containing $\alpha _1$ once). If
$m>1$ then $\beta (\lam ) \geq 2$ and the \ppc\ is satisfied. There is
one $\beta$-negative pair $\alpha _2, \alpha _0$ (a long pair). But
$\alpha _2 (\lam )=0$, so the \npc\ is satisfied and $\lambda$  covers a
pair, a contradiction.

This completes the proof of the theorem. 

\subsection{Mixed case}

We assume throughout this section that $\lambda$  is palindromic of mixed
type. Thus $\lambda$  has a unique negative node $s \in S$ and at least one
positive node $t \in S$. Most of the work is done in a series of preliminary
lemmas, culminating in Corollary~\ref{mixcor}.

Recall, we say that an element $\lambda$ of mixed type is {\it
overweight} if $\azlam \geq 2$, or equivalently, $s_0$ is a negative
node. An overweight $\lambda$ of mixed type covers a pair, and hence
is not palindromic.

\subsubsection{Preliminary lemmas} 

\begin{lemma} \label{mix1} \marginpar{}

Suppose that $\Phi$ is not of type $G_2$. Then either 

(a) there exists a positive node $t$ linked to $s$ with $\aslam + \atlam
\geq 0$, or

(b) $\Phi$ has type $C_n$ and $\lambda$  is the chain $(\und , 1,-2)$. 

\end{lemma}

\proof {\it Case 1}: There is a positive node $t$ linked to $s$ by a
subgraph $I$ of type $ABC$, where in the $BC$ case either (i) $t$ is
the minuscule node of $I$, or (ii) $s$ is the minuscule node of $I$
and $\aslam + 2\atlam \neq 0.$ In this case $\aslam +\atlam \geq 0$ by
Lemma~\ref{amove} (otherwise $\lambda$ covers a pair).

{\it Case 2}: There is a positive node $t$ linked to $s$ by a subgraph   
$I$ of type $BC$, where $s$ is the minuscule node of $I$ and             
$\aslam + 2\atlam =0.$ 

In type $B_n$ we have $t=s_n$, and we may assume there are no other
positive nodes, since such a node would be linked to $s$ by a type $A$
subgraph and we are back in Case 1. Then there is a \gsm\ based on
$I-\{s\}$; hence $\lambda$  covers a pair, a contradiction. 

In type $C_n$ we have $s=s_n$, so the assumption $\aslam + 2\atlam =0$
implies $t$ is the only positive node \owt . If $\atlam \geq 2$ there
is again a \gsm\ and $\lambda$ covers a pair, a contradiction. Hence
$\atlam =1$, with $t=s_i$ for some $i<n$. If $i<n-1$ then after firing
$s_n$ there is a linear $A$-move showing that $\lambda$ covers a
fork. Hence $\lam =(\und , 1, -2)$, which is a chain.

In type $F_4$ with $s=s_1$ and $t=s_3$, we have $\alpha _4 (\lam )=0$
\owt . Since $\aslam + \atlam <0$, firing down shows that $\lambda$  forks
too soon. If $s=s_2, s_3$ then we may assume that $\alpha _1 (\lam )=0$,
since otherwise we are back in Case 1. In both cases it follows that
$\lambda$  covers a fork, a contradiction.

{\it Case 3}: $\Phi$ has type $F_4$ and $\{s,t\}=\{s_1, s_4\}$. In
both cases $\lambda$ forks too soon.  This completes the proof of
Lemma~\ref{mix1}.

\begin{lemma} \label{mix2} \marginpar{}
Let $t$ be any positive node (not necessarily linked to $s$). 
Suppose \linebreak $m_t >m_s$ and $\alpha _s (\lam) + \alpha _t (\lam) \geq
0$. Then 

(1) $m_t =m_s +1$;

(2) $\aslam =-1$ and $\atlam =1$;

(3) there are no other positive nodes.
\end{lemma}

\proof If any one of the three conditions is not satisfied, then $\lambda$  is
overweight, a contradiction. 

\bigskip

\begin{lemma} \label{mix3} \marginpar{}

Let $t$ be a positive node linked to $s$. Suppose $m_s =m_t$ and $\aslam
+\atlam >0$. Then 

(1) $m_s =m_t =1$;

(2) $\aslam + \atlam =1$;

(3) there are no other positive nodes. 

\noindent Furthermore, either $\aslam =-1$ or $\Phi$ has type $A$ and $\lambda$  is
spiral

\end{lemma}

\proof It is immediate that conditions (1)-(3) hold \owt.  In
particular $\Phi$ has type $A$,$D$ or $E_6$, since there are two
minuscule nodes. Now suppose $\aslam \leq -2$, and let $\mu
=s\lam$. Then $\alpha _s (\mu) =-\aslam \geq 2$. In types $D$ and
$E_6$ there is a unique node $a \in S$ adjacent to $s$, and $a \neq
t$. Moreover, $a$ is minuscule in $S-\{s\}$. Hence if $\alpha , \alpha
^\prime$ is an $\alpha _s$-positive pair, so that $\alpha _s +\alpha
=\alpha ^\prime$, it follows that $\alpha _a$ occurs exactly once in
$\alpha$. Therefore, either $\alpha (\mu )=\alpha _a (\mu ) + \alpha
_t (\mu ) =1$, or $\alpha ^\prime (\mu) =\alpha _s (\mu) +\alpha _a
(\mu ) =0$. Thus $\mu$ covers a pair by Lemma~\ref{gsm} (with
$I=\{s\}$), and hence $\lambda$ covers a fork, a contradiction.

Now suppose $\Phi$ has type $A$. Then if $s$ and $t$ are not adjacent, a
similar argument shows that $\lam$ covers a trident, a
contradiction. If $s$ and $t$ are adjacent, $\lambda$ is spiral by
Proposition~\ref{spiralreps}. 

\bigskip

\begin{lemma} \label{mix4} \marginpar{}
Let $t$ be a positive node linked to $s$.  Suppose $m_s =m_t$ and
$\aslam + \atlam =0$. Then $\aslam =-1$.

\end{lemma}

\proof {\it Case 1}: Assume there are no other positive nodes.
Suppose $\aslam \leq -2$, and let $\beta$ denote the maximal root such
that $m_t \beta =m_s (\beta) +1$. Then $\beta (\lam ) \geq 2$. If
$\beta$ is a long root and $\alpha, \alpha ^\prime$ is a
$\beta$-positive pair, then since $\beta +\alpha =\alpha ^\prime$, we
must have $m_t (\alpha) \neq m_s (\alpha)$ by the maximality of
$\beta$. If $m_t (\alpha) > m_s (\alpha)$ then $\alpha (\lam )>0$. If
$m_t( \alpha) < m_s (\alpha)$ then $\alpha ^\prime (\lam )\leq
0$. Hence the \ppc\ is satisfied and we conclude that $\lambda$ covers
a pair by Proposition~\ref{affinemove}. In particular, this completes
the proof of Case 1 in the simply-laced case.

In types $B$ and $F_4$, $\beta$ is always a long root. This is clear on
inspection in type $B$. In $F_4$ we have $\{s,t\} =\{s_1, s_4\}$. If
$s_1$ is the negative node then $\beta =1342$, the second highest
root. This is clearly long since it belongs to the type $A_2$ subgraph
of \dtil\ on $s_0, s_1$. If $s_1$ is the positive node then $\beta
=1220$, the highest root of the $B_3$ subsystem. Hence $\beta$ is long,
completing the proof in $B, F_4$. 

In type $C_n$, $\beta$ will be a short root and more care is required
in the case of long pairs. Let $s=s_i, t=s_j$, where $1 \leq
i,j<n$. If $i<j$ then $\beta =e_1 + e_{i+1}$. There are no long
$\beta$-positive pairs, so the \ppc\ follows as before. However, there
is one \bnp\ $2e_1, 2e_{i+1}$. But $2e_1 =\alpha _0$ and $\alpha _0
(\lam ) =0$, so the \npc\ holds as well and $\lambda \da r_{\beta }
\lambda $ by Proposition~\ref{affinemove}.  If $i>j$ then $\beta =e_1
-e_i$. There are no long \bnp s, but there is one long \bpp\ $2e_1,
2e_i$. Thus to check the \ppc\ we have to consider $2 \beta + 2e_i
=2e_1 =\alpha _0$. But $\alpha _0 (\lam )=0$ so again $\lambda \da
r_{\beta } \lambda $.  Hence, in either case, $\lambda$ covers a pair,
a contradiction.

\bigskip

{\it Case 2}: Assume there is more than one positive node. Then the
following three conditions hold \owt :   

\bigskip
                                                                             
(1) there is only one additional positive node $u$; 
                            
(2) $u$ is minuscule;
                                                           
(3) $\alpha _u (\lam ) =1$.
                                                     
\bigskip

\noindent This rules out $E_8, F_4 $ and $G_2$ (since there is a
minuscule node) and also $C_n$ (since $\aulam $ is odd
by\S~\ref{sub:comparision}). Now assume that $\aslam \leq -2$.

Suppose first that $u,t,s$ lie on a type $A$ subgraph.  If $s$ is
linked to $u$, then $\lambda$ covers a pair by Lemma~\ref{amove}. So
suppose that $t$ lies between $u$ and $s$. If $\Phi$ has type $A$,
then there is a \gsm\ showing that $\lambda$ covers a pair where $I$
is the connected component of $S-\{u,s \}$ containing $t$.  In type
$BD$ we must have $u=s_1$ (if $u=s_{n-1}, s_n$ in type $D_n$, then
$\lam \notin \cQ ^\vee$). Thus $t=s_i$, $s=s_j$, with $1<i<j$. Now let
$\beta =e_1 +e_i$, which is the smallest root containing $\alpha _u$
such that $m_t (\beta) =1$, $m_s(\beta) =2$. Then $\beta (\lam) =1 +
\aslam <0$, and one easily checks that the \osc\ is satisfied. Hence
$\lambda$ covers a pair in this case by Proposition~\ref{affinemove}.

In type $E_6$ we have $m_s =m_t=2$. In all cases $s$ has two adjacent
zero nodes in \dtil , and therefore $\lambda$ covers a fork. In type $E_7$ we have
$s=s_1, s_2, s_3$. In all cases $\lambda$ forks too soon. 

If $u,t,s$ do not lie on a type $A$ subgraph, then we are in type $B_n$
with $u=s_1$ and $s=s_n$. Then after firing $s$ there is a linear
$A$-move showing that $\lambda$ covers a fork. (Note that this works whether
or not $t$ is adjacent to $s$, bearing in mind that back-firing along the
double bond adds $2\aslam$ to the value of its neighbor.)

\bigskip

\begin{lemma} \label{mix5} \marginpar{}

Suppose $\Phi$ is not of type $G_2$ and $m_t <m_s$ for all positive
nodes $t$. Then $\aslam =-1$. 

\end{lemma}

\proof Suppose $\aslam \leq -2$. 

{\it Case 1:} $\Phi$ has type $BCD$ and there is only one positive
node. 

Note that the positive node is minuscule.  We may assume that $2 \aslam
+\atlam \leq 1$ (otherwise $\lambda$ is overweight). If $\astlam <0$ there
is a linear $ABC$-move and $\lambda$ covers a pair. If $\astlam =0$ and
$\aslam \leq -2$ there is a \gsm , and again $\lambda$ covers a pair. So we
assume $\astlam >0$ and consider three cases:

$2 \aslam + \atlam <0$: Let $\beta$ be the smallest root containing $
2 \aslam + \atlam$. In all cases $\beta$ is a long root. It is then
clear that the \osc\ is satisfied, as desired.  Hence $\lambda \da
s_{\beta } \lambda$ by Proposition~\ref{p:linearmove} and $\lambda$
covers a pair.

$2 \aslam + \atlam =0$: Let $\beta$ be the highest root containing
$\alpha _s$ once. Then $\beta (\lam ) = \astlam \geq 2$ since
$\alpha_{s}(\lambda)\leq 2$. If $\beta$ is long then the \ppc\ is
clearly satisfied and we are done. If $\beta$ is short and there is a
long \bnp\ $\alpha , \alpha ^\prime$, then $\alpha + \alpha ^\prime
=2\beta$ and $\alpha ^\prime =\beta + (\beta -\alpha)$, where $\beta
-\alpha$ is a positive root. By the maximality of $\beta$, it follows
that $\alpha ^\prime$ must contain $\alpha _s$ twice, as well as
$\alpha _t$. Hence $\aplam =0$, and both the \ppc\ and the \npc\ are
satisfied.  By Proposition~\ref{affinemove}, $\lambda$ covers a pair.  

$2 \aslam + \atlam =1$: Note this rules out type $C$, since $\atlam$
is odd. In type $D$ there are two zero nodes adjacent to $s$, so
$\lambda$ covers a fork. Type $B_n$ is similar, except in the case
$s=s_n$. In that case, let $\lam =(a, \und , b)$ with $b<0$ and
$a+2b=1$.  Firing $s_n$ yields $\mu = (a, \und, 2b , -b)$. If $b \leq
-2$ there is an affine move $\mu \da r_{\alpha _n} \mu$ by
Proposition~\ref{gsm}, showing that $\lambda$ again covers a
fork. Here we note that if $\alpha , \alpha ^\prime$ is a \bpp , then
$2 \alpha _n + \alpha =\alpha ^\prime$, and hence $\alam =2b$ or
$\alam =1$. In either case the \ppc\ is satisfied.

This completes the proof in case 1. 

\bigskip

{\it Case 2:} $\Phi$ has type $BCD$ and there is more than one positive
node. 

Since the positive nodes must be minuscule, this can only happen in
type $D$. If $t$ is any positive node then it is linked to $s$ and
$\aslam + \atlam \geq 0$ by Lemma~\ref{amove}. Hence there are two
positive nodes $t, u$. If $\aslam + \atlam =0=\aslam + \aulam$, there
is an affine move by Lemma~\ref{amove} with $\beta$ being the highest
root containing $\alpha_{s}$ once.  Thus, $\lam$ covers a pair. If
(say) $\aslam + \atlam >0$, then $\aslam + \atlam =1$ and $\aslam +
\aulam =0$ \owt . In other words, up to symmetry the diagram of
$\lambda$ has one of the following forms, with $a\geq 2$: $(a, \und,
-a , \und , 0,a+1)$ and $(\und , -a , \und , a+1,a)$.

\bigskip
$$
{\begin{picture}(5,.8)
\mp(1,0)(1,0){2}{\ci}
\put(1,0){\num{0}}\put(2,0){\num{-a}}
\put(3,0){\makebox(0,0){\ldots}}
\put(4,0){\ci}
\put(4,0){\num{0}}
\put(1,0){\line(1,0){1.5}}
\put(4,0){\line(-1,0){.5}}
\put(4,0){\line(2,1){1}}
\put(4,0){\line(2,-1){1}}
\put(5,.5){\ci}\put(5,-.5){\ci}
\put(5,.5){\makebox(0,0)[l]{\hspace{.3\unitlength}$0$}}
\put(5,-.5){\makebox(0,0)[l]{\hspace{.3\unitlength}$a+1$}}
\put(0,.5){\ci}\put(0,-.5){\ci}
\put(0,.5){\makebox(0,0)[l]{\hspace{-.35\unitlength}$a$}}
\put(0,-.5){\makebox(0,0)[l]{\hspace{-.35\unitlength}$0$}}
\put(1,0){\line(-2,1){1}}
\put(1,0){\line(-2,-1){1}}
\end{picture}}
$$
\bigskip

\bigskip

\bigskip
$$
{\begin{picture}(5,.8)
\mp(1,0)(1,0){2}{\ci}
\put(1,0){\num{0}}\put(2,0){\num{-a}}
\put(3,0){\makebox(0,0){\ldots}}
\put(4,0){\ci}
\put(4,0){\num{0}}
\put(1,0){\line(1,0){1.5}}
\put(4,0){\line(-1,0){.5}}
\put(4,0){\line(2,1){1}}
\put(4,0){\line(2,-1){1}}
\put(5,.5){\ci}\put(5,-.5){\ci}
\put(5,.5){\makebox(0,0)[l]{\hspace{.3\unitlength}$a+1$}}
\put(5,-.5){\makebox(0,0)[l]{\hspace{.3\unitlength}$a$}}
\put(0,.5){\ci}\put(0,-.5){\ci}
\put(0,.5){\makebox(0,0)[l]{\hspace{-.35\unitlength}$0$}}
\put(0,-.5){\makebox(0,0)[l]{\hspace{-.35\unitlength}$0$}}
\put(1,0){\line(-2,1){1}}
\put(1,0){\line(-2,-1){1}}
\end{picture}}
$$
\bigskip

\noindent In the first case $\lambda$ covers a fork. The second diagram is not even in
the coroot lattice, as the reader can check. 

\bigskip

\noindent {\it Case 3:} $\Phi$ has type $E$. It is a pleasant exercise
in the palindromy game to check that $\lambda$ forks too soon or is
overweight. Details are left to the reader.

\bigskip

\noindent {\it Case 4:} $\Phi$ has type $F_4$. Then $s=s_2, s_3$. If
$s=s_2$ then $\alpha _3 (\lam )=0$ by assumption. Then $\alpha _1 (\lam)
+ \alpha _2 (\lam ) \geq 0$ by Lemma~\ref{amove}a, and $\alpha _2
(\lam ) + \alpha _4 (\lam ) \geq 0$ (otherwise $\lambda$ forks too
soon). But then $\lambda$ is overweight. 

If $s=s_3$ then for any positive node $t$ linked to $s$ we have $\aslam
+\atlam \geq 0$, by Lemma~\ref{amove}a. If $\alpha _2 (\lam )>0$
then $\alpha _2 (\lam ) + \alpha _3 (\lam) \geq 0$. By Proposition 9.2b
we have either $\alpha _2 (\lam ) + \alpha _3 (\lam) =0$ or $\alpha _2
(\lam ) + 2 \alpha _3 (\lam) \geq 0$. In the first case we must have
$\alpha _4 (\lam )>0$ (otherwise $\lambda$ forks too soon) and $\lambda$ is
overweight. In the second case $\lambda$ is again overweight. Finally if
$\alpha _2 (\lam ) =0$ then $\alpha _4 (\lam )>0$ and $\alpha _1 (\lam )
+2 \alpha _3 (\lam ) \geq 0$ (otherwise $\lambda$ forks too soon), and again
$\lambda$ is overweight. (This last argument doesn't use the assumption
$\aslam \leq -2$.)

\bigskip

The main conclusions of the five lemmas above can be summarized as
follows: 

\begin{corollary} \label{mixcor} \marginpar{}

Suppose $\Phi$ is not of type $G_2$ and $\lambda$ is palindromic of mixed
type, with unique negative node $s \in S$. Then at least one of the
following conditions holds:

(1) $\aslam =-1$;

(2) $\lambda$ is a chain; 

(3) $\Phi$ has type $A$ and $\lambda$ is spiral. 

\end{corollary}

\subsubsection{Proof of the palindromy theorem in the mixed case} 

By Corollary~\ref{mixcor} we may assume $\aslam =-1$ (except in type
$G_2$). If $s$ is a long node, then the palindromic $\lambda$ is admissible,
and hence is a \cpo\ or a chain by Lemma~\ref{smooth1}. In particular,
the proof of Theorem~\ref{pal} is now complete in the simply-laced
case. It remains to consider types $BCF$ when $s$ is a short node with
$\aslam =-1$, and type $G_2$.

\bigskip

$C_n$: Since $s$ is short, $s=s_i$ for some $i<n$. Let $s_j$ be a
positive node linked to $s_i$. If $j<n$, by Lemma~\ref{mix1} we may
assume $\astlam =0$ \owt\ and that there are no other positive nodes by
Lemma~\ref{mix3}. If $s$ and $t$ are not adjacent, then firing $s$
shows that $\lambda$ covers a fork (note this works even when $i=1$,
since then back-firing along the double bond adds -2 to the initial
value +1 on the node $s_0$). If $s$ and $t$ are adjacent then $\lambda
=\pm (\und , 1,-1,\und)$, which is a chain.

Now suppose $t=s_n$ is the unique positive node. Since $\alpha _n
(\lam)$ is necessarily even, we must have $\alpha _n (\lam )=2$ \owt
. If $i<n-1$, then firing $s$ shows that $\lambda$ covers a fork (note this
works even when $i=1$, since back-firing along the double bond puts -1
on the $s_0$ node). If $i=n-1$ we have $\lam =(\und , -1, 2)$, which is
a chain.

\bigskip

$B_n$: Here $s=s_n$. Let $s_j$ be the unique positive node linked to
$s_n$. Then $\alpha _j (\lam )=1,2,3$, where the values $2,3$ can occur
only when $j=1$ \owt .

Case 1: $\alpha _j (\lam ) =1$. If $j<n-1$ then firing $s_n$ yields $(\und,
1, \und , -2, 1)$ and there is a linear $A$-move, showing that $\lambda$
covers a fork. If $j=n-1$ then there must be another positive node
(otherwise \nope ). This forces $\lam =(1, \und, 1, -1)$, which is a singular
chain. 

\bigskip

Case 2: $\alpha _j (\lam ) =2$. This forces $\lam =(2, \und , -1)$. Thus $n$
must be even (otherwise \nope ), and in particular $n \geq 4$. Then
$\lambda$ covers a fork: Firing $s_n$ yields $(2, \und, -2, 1)$, which
admits an affine move based on the long root $\beta =\alpha _1 + \ldots +
\alpha _{n-1} + 2\alpha _n =e_1 + e_n$. This is very similar to the
anti-dominant case $(\und, -1)$; details are left to the reader.

Case 3: $\alpha _j (\lam ) =3$. This forces $\lam =(3, \und ,
-1)$. Thus $n$ must be odd (otherwise \nope ). If $n \geq 5$, we can
proceed exactly as we did in the anti-dominant case with $(\und ,
-1)$. The key point is again that $(3, \und , -2, 0, 1)$ admits an
affine move using $\beta =\alpha _{n-1} + 2\alpha _n$, and as a result
$\lambda$ covers a scepter, hence is not palindromic. This leaves the
case $n=3$, $\lambda =(3,0,-1)$ which is known to be palindromic by
Corollary~\ref{funnypal}.  

This completes the proof of
Theorem~\ref{pal} in type $B$.

\bigskip

$F_4$: Here the short nodes are $s_3,s_4$. If $s=s_4$, then $\lam
=(0,1,0,-1)$ or $\lam =(1,0,0,-1)$ \owt . Then $(0,1,0,-1)$ is a chain,
while $(1,0,0,-1)$ forks too soon (barely!).

If $s=s_3$, the argument used in the proof of Lemma~\ref{mix5} shows
that $\lam =(0,1,-1,1)$, which is a chain. 

\bigskip

$G_2$: Before starting the proof, we recall (\S 7.3) that there are
four chains of mixed type, namely $\pm (1,-1)$ and $\pm
(1,-2)$. Recall also that the initial chain coming down from $(-1,0)$
(see the Hasse diagram in \S 13)

$$(-1,0) \da (1,-3) \da (-2,3) \da (2,-3),
$$
with $(2,-3)$ covering a pair. Hence all of the displayed
elements fork too soon. 

Now suppose $\lambda$ is generic, meaning that $\alpha (\lam ) \neq 0$
for all $\alpha \in \Phi$. Let $\gamma$ range over the four non-simple
positive roots $\alpha _1 + \alpha _2$, $2\alpha _1 + \alpha _2$
$3\alpha _1 + \alpha _2$, $3\alpha _1 + 2\alpha _2$. If at least one
of the $\gamma$'s is negative on $\lambda$, then the minimal such
$\gamma$ satisfies the \osc , and $\lambda$ covers a pair by
Proposition~\ref{affinemove}.

Now suppose all of the $\gamma$'s are positive on $\lambda$. If $s_2$ is the
positive node, then $\lam$ is overweight and again $\lambda$ covers a
pair. If $s_2$ is the negative node, firing it yields a generic dominant
class $\mu$: $\lam =(a,b) \da (a +b, -b)=\mu$. We saw earlier that any
such class in type $G_2$ covers a pair, so $\lambda$ covers a fork. 

It remains to consider the case $\gamma (\lam )=0$ for some $\gamma
$. Let $a$ denote a positive integer. 

$\gamma =\alpha _1 + \alpha _2$: If $\lam =(a,-a)$ and $a\geq 2$ then
$\lam$ is overweight. If $a=1$ we have the chain $(1,-1)$. If $\lam
=(-a,a)$ and $a \geq 2$, $\lambda$ has an affine move $\lambda \da
r_{\alpha _2}\lambda$ by Proposition~\ref{affinemove}. Here $\alpha
_2$ is long and there are two \bpp s $\alpha , \alpha ^\prime$:
$(\alpha _1, \alpha _1 + \alpha _2)$ and $(3\alpha _1 + \alpha _2,
3\alpha _1 + 2\alpha _2 )$. In each case $\aplam <0$ and the \ppc\ is
satisfied. If $a=1$ we again have a chain.

$\gamma =2\alpha _1 + \alpha _2$: If $\lam =(a,-2a)$ and $a\geq 2$,
there is an affine move $r_\beta$ with $\beta =3 \alpha _1 + \alpha
_2$. Note that $\beta$ is long and there is just one \bpp\ $\alpha _2,
\alpha _0$, so the \ppc\ is satisfied. If $a=1$ we have the chain
$(1,-2)$. If $\lam =(-a, 2a)$ then $\lambda$ is overweight unless $a=1$, in
which case we again have a chain. 

$\gamma =3 \alpha _1 + \alpha _2$: If $\lam =(-a, 3a)$ then $\lambda$ is
overweight. If $\lam =(a, -3a)$ and $a  \geq 2$, there is an affine move
$r_{\alpha _1}$ and $\lambda$ covers a pair. Finally $(1,-3)$ forks too soon
as noted above. 

$\gamma =3 \alpha _1 + 2\alpha _2$: Note that in this case $|\alpha _1
(\lam )|, |\alpha _2 (\lam )| \geq 2$. If $s_2$ is the negative node
then there is an affine move $r_\beta$ with $\beta =3 \alpha _1 + \alpha
_2$. If $s_1$ is the negative node and $\alpha _1 (\lam ) + \alpha _2
(\lam ) \leq 0$ there is an affine move $r_{\alpha _2}$. If $\alpha _1
(\lam ) + \alpha _2 (\lam ) >0$ and $2\alpha _1 (\lam) + \alpha _2 (\lam
) \leq 0$, there is an affine move $r_\beta$ with $\beta = \alpha _1 +
\alpha _2$, provided that $\beta (\lam ) \geq 2$. If $\beta (\lam )=1$
then $\lam =(-2,3)$, which forks too soon as noted above. Finally if
$2\alpha _1 (\lam) + \alpha _2 (\lam ) > 0$, there is an affine move
$r_\beta$ with $\beta = 2\alpha _1 + \alpha _2$.

\section{The spiral varieties in type $A$}\label{s:spiral.varieties}
Most of the results in this section are from \cite{mfilt}, to which
the reader is referred for the missing proofs. The varieties $X_{n,k}$
are denoted $F_{n+1,k}$ in \cite{mfilt}.  We include these results for
the readers' convenience and to highlight some unsaid consequences of
the previous work.  
                      
Let $\sigma _d$ (resp. $\sigma ^\prime _d$) denote the word in \wtil\
obtained by starting at $s_0$ and---writing the word from right to
left---proceeding clockwise (resp. counterclockwise) $d$ steps around
the Coxeter diagram. For example, if $n=3$ then $\sigma
_6=s_3s_0s_1s_2s_3s_0$. Note that these words are reduced, in
$\minreps$, and rigid.    Note also that $\sigma ^\prime _d$ is
conjugate to $\sigma _d$ under the involution of the Dynkin diagram
fixing $s_0$.

Let $\sigma _{n,k} =\sigma _{kn}$, $\sigma _{n,k} ^\prime =\sigma
_{kn} ^\prime$. We call these elements and the varieties associated to
them {\it spiral}. The term is suggested by the manner in which
$\sigma _{n,k}$, $\sigma _{n,k} ^\prime$ spiral around the affine
Dynkin diagram and up in length as $k$ increases.

{\it Warning:} Note that the spiral classes have length divisible by
$n$, not $n+1$. Hence they are out of phase with the natural period of
the affine Dynkin diagram; the first (left-hand) factor $s$ of $\sigma
_{n,k}$ rotates around the diagram as $k$ increases.

The corresponding coroot lattice representatives $\lambda
_{n,k}, \lambda _{n,k}^\prime $ are described as follows:

\begin{proposition} \label{spiralreps} \marginpar{}

An element $\lam \in \cQ^\vee$ represents a spiral class if and only if 

\bigskip

(1) $\lambda$ has exactly two nonzero nodes $s,t \in \stil$;

(2) $s,t$ are adjacent; and 

(3) $\aslam + \atlam =1$. 

\bigskip

\noindent More precisely, let $k=r(n+1) +i$, where $0 \leq i <n+1$. Then
$$\lambda _{n,k} =(\und , k+1 ,-k, \und),$$                                 
\noindent where $k+1$ is in the $i$-th coordinate, or 
$\lambda _{n,k} =(-k, \und)$, $\lambda _{n,k}=(\und , k+1)$. 
Similarly 
                                                                 
$$\lambda_{n,k} ^\prime =(\und , -k , k+1 , \und)$$                      

\noindent or $\lambda _{n,k}^\prime =(k+1, \und)$, $\lambda _{n,k}
^\prime=(\und , -k)$.

\end{proposition} 

\proof The explicit formulas for $\lambda _{n,k}, \lambda_{n,k} ^\prime$
are easily obtained by induction on $k$, and show that the spiral
classes satisfy properties (1)-(3). Conversely, suppose that (1)-(3)
hold, and consider the case $s=s_i, t=s_{i+1}$ with $0<i<n$ and $\aslam
>0$. Then since $\lambda$ is in the coroot lattice, we have $i \aslam +
(i+1) \atlam =0 \, mod \, n+1$. Then $\lam =\lam _{n,k}$ with
$k=-\atlam$. The remaining cases are similar. 

\bigskip

Let $X_{n,k}=X_{\sigma _{n,k}}$, $X_{n,k}^\prime =X_{\sigma _{n,k}}
^\prime.$ Since $X_{n,k}^\prime$ is canonically isomorphic to $X_{n,k}$ as
a variety, in what follows we will state the results only for
$X_{n,k}$.          

\begin{proposition} \label{fnk}                                                 
$X_{n,k}$ is isomorphic to the variety of $k$-dimensional $\bc
[z]/z^{k}$-submodules in $\bc [z]/z^k \otimes \bc ^{n+1}$. In
particular, $X_{n,1} \cong \bc P^n$.

\end{proposition}                                                               
This description arises from the Quillen model, which identifies $\cL
_{SU(n+1)}$ with certain spaces of $\bc [z]$-lattices in $\bc
[z,z^{-1}] \otimes \bc ^{n+1}$. In fact from this point of view,
$X_{n,k}$ is precisely the intersection in $BU$ of the Ind-varieties
$\cL _{SU(n+1)}$ and $BU(k)$.

%Let $\displaystyle \left[\stackrel{n+k}{n} \right]_{q}$

Let $\qbinomial$ denote the ``q-binomial coefficient'' or equivalently
the Poincar\'e polynomial for the classical Grassmannian $G_{n}\bc^{n+k}$.  

\begin{proposition} \label{qbinom}                                              
$|X_{n,k}| =\qbinomial$. In particular, $X_{n,k}$ is palindromic.                
\end{proposition}                                               

Undoubtedly a direct combinatorial proof of this could be given. Here we
will mention two topological proofs. The first is based on:
                                                                                
\begin{proposition}  \marginpar{}                                       
$X_{n,k}-X_{n,k-1} =E(\gamma ^k \downarrow X_{n-1,k})$, the total space
  of the canonical $k$-plane bundle defined by
  Proposition~\ref{fnk}. Hence $X_{n,k}/X_{n,k-1} =T(\gamma ^k )$.
  
\end{proposition}                                                               
Thus $|X_{n,k}| =|X_{n,k-1}|+ t^k |X_{n-1,k}|$, which is exactly the
Pascalian recursion formula for the $q$-binomial coefficients.  This
yields one proof of Proposition~\ref{qbinom}. The second is based on:

\begin{theorem} \label{fnkprod}                                                 
$[X_{n,j}] \cdot [X_{n,k}]=[X_{n,j+k}]$, where the dot denotes
  Pontrjagin product in $H_* \Omega G \cong H_* \cL _G$. Moreover the
  natural map $Sym ^k (H_*X_{n,1}) \lra H_* X_{n,k}$ is an \iso .
  ($Sym^k$ denotes the $k$-th symmetric power.)

\end{theorem}                                                                   
This theorem gives another proof of Proposition~\ref{qbinom}, since one
can easily check that $|Sym ^k (H_* \bc P^n)| =\qbinomial$.

\bigskip                                                                        
\begin{theorem}  \marginpar{}                                                                                                                          
$X_{n,k}$ satisfies Poincar\'e duality integrally if and only if
$k=1$. It satisfies Poincar\'e duality rationally for all $k$.
\end{theorem} {\it Remark:} The first assertion has already been
proved in Theorem~\ref{smooth}. By Proposition~\ref{qbinom} $X_{n,k}$
is palindromic, and hence is rationally smooth by the Carrell-Peterson
theorem, proving the second assertion. Here we provide alternative
proofs of both assertions.

\bigskip

\proof First observe that Poincar\'e duality can be expressed in terms
of homology as follows: If $[X]$ is the fundamental class, then there
are bases $e_i, e_i ^\prime$ for $H_* X$ that are dual in the sense that
$\Delta _* [X] =\sum e_i \otimes e_i ^\prime$, where $\Delta$ is the
diagonal map. It follows that if $X_\lambda$ is any \svar\ and
$[X_\lambda]=y^k$ for some $y \in H_* \cL _G$ and $k>1$, then
$X_\lambda$ does not satisfy \pd : For we may assume $k=p$ is a prime,
and then $\Delta _* [X_\lambda] =(\Delta _* y)^p \, mod \, p$. Hence
$\Delta _* [X_\lambda]$ is concentrated in bidegrees divisible by $p$,
and no such dual bases exist. Since $[X_{n,k}]=[X_{n,1}]^k$ by
Theorem~\ref{fnkprod}, this proves that $X_{n,k}$ does not satisfy \pd\
for $k>1$. For $k=1$, we have $X_{n,1} =\bp ^n$. 

Now let $b_1,\ldots ,b_n$ denote the standard basis for $\tilde{H}_*
X_{n,1}=\tilde{H}_* \bc P^n$. Then $H_* X_{n,k}$ consists of polynomials
of degree at most $k$ in $b_1,\ldots ,b_n$, with $b_n ^k$ the fundamental
class. We set $b_0 =1$. Now let $r=(r_0, \ldots , r_n)$ be a multi-index
with $r_i \geq 0$ and $\sum r_i =k$. Let $b^r =b_0 ^{r_0} \ldots b_n
^{r_n}$, and let $r^* =(r_n, \ldots , r_0)$. Then the $b_r$, $b_{r^*}$ are
bases, and if we take coefficients in \bq , then up to scalar multiples
they are dual in the above sense:

$$\Delta _* (b_n ^k) =(\sum _{i+j =n} b_i \otimes b_j)^k                        
=\sum _{r} c_r (b_0 \otimes b_n)^{r_0} \ldots (b_n \otimes b_0)^{r_n}               
=\sum _r c_r b^r \otimes b^{r^*},$$                                             
\noindent where the multinomial coefficients $c_r$ are all
nonzero. Thus, up to nonzero scalar multiples, $b_r$, $b_{r^*}$ are dual
bases over \bq . This completes the proof.

\section{Hasse diagrams}

In this section we give some examples of Hasse diagrams for the Bruhat
order on \wtils\ in a range of dimensions, showing in particular the
palindromics and the \cpo s. These diagrams are easily generated by
hand in the following way: First of all, the length generating
function $|\wtils |(t)$ is given by Bott's formula
Theorem~\ref{t:Bott} where in the latter case variable $t$ is assigned
dimension 2.  So we know in advance the number of nodes at each
level. Then we begin firing up from $0 \in \cQ ^\vee$; this easily
yields the weak order on the coroot lattice representatives $\lam$. By
recording the node fired at each step, we have reduced expressions for
the minimal length representatives $\sigma \in \wtils$ as well
(incidentally, this also yields the cup product coefficients occurring
in Chevalley's formula). Then we fill in the missing covering
relations by using the standard criterion: $\sigma \da \tau$ if and
only if $\tau$ can be obtained from $\sigma$ by omitting one generator
from some reduced expression for $\sigma$. A further interesting
exercise is to find and check the affine or linear moves that produce
these descents.

The circled nodes are the non-trivial palindromic classes. A double
circle indicates a \cpo .

\pagebreak

%***Diagram A2 here***

\begin{center}
\setlength{\unitlength}{.3cm}
\begin{picture}(40,40)(0,0)
\put(0,40){\line(1,-1){5}}
\put(0,40){\circle*{.5}}
\put(5,35){\line(1,1){5}}
\put(10,40){\circle*{.5}}
\put(10,40){\line(1,-1){5}}
\put(25,35){\line(1,1){5}}
\put(30,40){\circle*{.5}}
\put(30,40){\line(1,-1){5}}
\put(40,40){\circle*{.5}}
\put(5,35){\circle*{.5}}
\put(5,35){\line(0,-1){5}}
\put(5,35){\line(2,-1){10}}
%\put(4.5,29.5){\makebox{\circbul}}
\put(5,30){\circle*{.5}}
\put(5,30){\circle{1}}
\put(1.4,29.75){\makebox{$\scriptstyle(0,-3)$}}
\put(5,30){\line(1,-1){5}}
\put(5,30){\line(2,1){10}}
\put(10,25){\circle*{.5}}
\put(10,25){\line(1,1){5}}
\put(10,25){\line(-1,1){5}}
\put(10,25){\line(2,-1){10}}
\put(10,20){\circle*{.5}}
\put(10,20){\circle{1}}
\put(6.75,19.75){\makebox{$\scriptstyle(3,0)$}}
\put(10,20){\line(0,1){5}}
\put(10,20){\line(2,1){10}}
\put(10,20){\line(1,-1){5}}
\put(15,35){\circle*{.5}}
\put(15,35){\line(1,1){5}}
\put(15,35){\line(2,-1){10}}
\put(15,35){\line(0,-1){5}}
\put(15,30){\circle*{.5}}
\put(15,30){\line(2,1){10}}
\put(15,30){\line(1,-1){5}}
\put(15,15){\circle*{.5}}
\put(11.75,14.75){\makebox{$\scriptstyle(1,-2)$}}
\put(15,15){\line(-1,1){5}}
\put(15,15){\line(1,1){5}}
\put(15,10){\circle*{.5}}
\put(15,10){\circle{1}}
\put(15,10){\circle{1.5}}
\put(11,9.75){\makebox{$\scriptstyle(-1,2)$}}
\put(15,10){\line(0,1){5}}
\put(15,10){\line(2,1){10}}
%\put(19.5,39.5){\makebox{\circbul}}
\put(20,40){\circle*{.5}}
%\put(20,40){\circle{1}}
\put(20,40){\line(1,-1){5}}
\put(20,40){\line(-1,-1){5}}
\put(20,25){\circle*{.5}}
%\put(19.5,19.5){\makebox{\circbul}}
\put(20,20){\circle*{.5}}
%\put(20,20){\circle{1}}
\put(20,20){\line(0,1){5}}
\put(21,19.5){\makebox{$\scriptstyle(-1,-1)$}}
\put(20,20){\line(2,1){10}}
\put(20,5){\circle*{.5}}
\put(20,5){\circle{1}}
\put(20,5){\circle{1.5}}
\put(21,4.75){\makebox{$\scriptstyle(1,1)$}}
\put(20,5){\line(-1,1){5}}
\put(20,5){\line(1,1){5}}
\put(20,0){\circle*{.5}}
\put(20,0){\circle*{.5}}
\put(20,0){\line(0,1){5}}
\put(20.5,0){\makebox{$\scriptstyle(0,0)$}}
\put(25,35){\circle*{.5}}
\put(25,35){\line(2,-1){10}}
\put(25,35){\line(0,-1){5}}
\put(25,30){\circle*{.5}}
\put(25,30){\line(-1,-1){5}}
\put(25,30){\line(1,-1){5}}
\put(25,30){\line(2,1){10}}
\put(25,15){\circle*{.5}}
\put(25,15){\line(1,1){5}}
\put(25,15){\line(-1,1){5}}
\put(25.75,14.75){\makebox{$\scriptstyle(-2,1)$}}
\put(25,10){\circle*{.5}}
\put(25,10){\circle{1}}
\put(25,10){\circle{1.5}}
\put(25,10){\line(0,1){5}}
\put(25,10){\line(-2,1){10}}
\put(26,9.75){\makebox{$\scriptstyle(2,-1)$}}
\put(30,25){\circle*{.5}}
\put(30,25){\line(1,1){5}}
\put(30,20){\circle*{.5}}
\put(30,20){\circle{1}}
\put(31,19.75){\makebox{$\scriptstyle(0,3)$}}
\put(30,20){\line(0,1){5}}
\put(30,20){\line(-2,1){10}}
\put(35,35){\circle*{.5}}
\put(35,35){\line(1,1){5}}
%\put(34.5,29.5){\makebox{\circbul}}
\put(35,30){\circle*{.5}}
\put(35,30){\circle{1}}
\put(36,29.75){\makebox{$\scriptstyle(-3,0)$}}
\put(35,30){\line(0,1){5}}
%\put(0,.5){\circle*{.5}}
%\put(0,.5){\circle{1}}
%\put(1,0){\makebox{= antidominant $\lambda$}}
\end{picture}
\end{center}

\begin{center}
Type $\widetilde A_2$
\end{center}
%\bigskip

The spiral classes $\lambda _{2,k}, \lambda ^\prime _{2,k}$ are the
elements of even length along the two edges of the diagram.

Note also that for $k=2$ the \ppoly\ of $X_{(3,0}$ or $X_{(0,3)}$ is
the same as that of $G_2 \bc ^4$. The ring structure in cohomology,
however, is different; \pd\ fails. When $k=3$ the \ppoly\ is the same
as that of $G_2 \bc ^5$. The order ideal below $(-3,0)$ or $(-3,0)$ is
visibly not self-dual; in particular it is not isomorphic to Bruhat
poset of $G_2 \bc ^5$.

The bilateral symmetry in the diagram reflects the automorphism of the
affine Dynkin diagram fixing the special node $s_0$.

\bigskip
\newpage

%***Diagram C2 here***

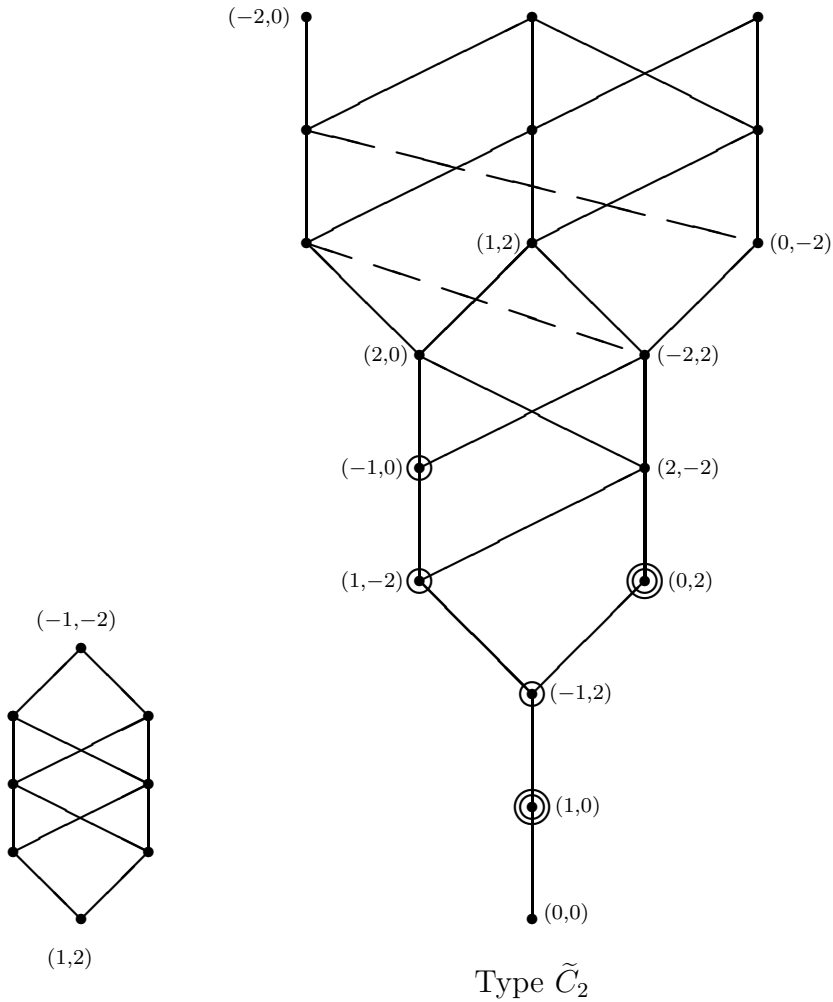
\begin{center}
\setlength{\unitlength}{.3cm}
\begin{picture}(20,40)(0,0)
\put(0,35){\line(2,1){10}}
\put(10,40){\line(2,-1){10}}
\put(0,35){\line(4,-1){3}}
\put(4,34){\line(4,-1){3}}
\put(8,33){\line(4,-1){3}}
\put(12,32){\line(4,-1){3}}
\put(16,31){\line(4,-1){3}}
\put(0,30){\line(3,-1){2}}
\put(3,29){\line(3,-1){2}}
\put(6,28){\line(3,-1){2}}
\put(9,27){\line(3,-1){2}}
\put(12,26){\line(3,-1){2}}
\put(10,0){\line(0,1){10}}
\put(10,10){\line(1,1){5}}
\put(5,15){\line(2,1){10}}
\put(5,15){\line(0,1){10}}
\put(5,20){\line(2,1){10}}
\put(0,30){\line(2,1){20}}
\put(0,30){\line(0,1){10}}
\put(10,30){\line(0,1){10}}
\put(10,30){\line(2,1){10}}
\put(7.5,29.75){\makebox{$\scriptstyle(1,2)$}}
\put(15,20){\line(-2,1){10}}
\put(15,15){\line(0,1){10}}
\put(20,30){\line(0,1){10}}
\put(15,25){\line(1,1){5}}
\put(5,25){\line(-1,1){5}}
\put(5,15){\line(1,-1){5}}
\put(5,25){\line(1,1){5}}
\put(5,15){\line(1,-1){5}}
\put(5,25){\line(1,1){5}}
\put(15,25){\line(-1,1){5}}
\put(10,0){\circle*{.5}}
\put(10,5){\circle*{.5}}
\put(10,10){\circle*{.5}}
\put(5,15){\circle*{.5}}
\put(15,15){\circle*{.5}}
\put(15,20){\circle*{.5}}
\put(15.5,19.75){\makebox{$\scriptstyle(2,-2)$}}
\put(0,30){\circle*{.5}}
\put(10,30){\circle*{.5}}
\put(0,35){\circle*{.5}}
\put(10,35){\circle*{.5}}
\put(20,35){\circle*{.5}}
\put(10,40){\circle*{.5}}
\put(20,40){\circle*{.5}}
\put(5,25){\circle*{.5}}
\put(15,25){\circle*{.5}}
\put(15.5,24.75){\makebox{$\scriptstyle(-2,2)$}}
\put(2.5,24.75){\makebox{$\scriptstyle(2,0)$}}
\put(16,14.75){\makebox{$\scriptstyle(0,2)$}}
\put(5,15){\circle{1}}
\put(1.5,14.75){\makebox{$\scriptstyle(1,-2)$}}
\put(10,10){\circle{1}}
\put(10.75,9.75){\makebox{$\scriptstyle(-1,2)$}}
\put(10,5){\circle{1}}
\put(10,5){\circle{1.5}}
\put(11,4.75){\makebox{$\scriptstyle(1,0)$}}
\put(15,15){\circle{1}}
\put(15,15){\circle{1.5}}
\put(0,40){\circle*{.5}}
%\put(0,40){\circle{1}}
\put(10.5,0){\makebox{$\scriptstyle(0,0)$}}
\put(-3.5,39.75){\makebox{$\scriptstyle(-2,0)$}}
\put(20,30){\circle*{.5}}
%\put(20,30){\circle{1}}
\put(20.5,29.75){\makebox{$\scriptstyle(0,-2)$}}
%\put(0,20){\makebox{$(-1,0)$}}
\put(1.5,19.75){\makebox{$\scriptstyle(-1,0)$}}
\put(5,20){\circle*{.5}}
\put(5,20){\circle{1}}
\put(-10,12){\circle*{.5}}
\put(-10,12){\line(1,-1){3}}
\put(-13,9){\circle*{.5}}
\put(-13,9){\circle*{.5}}
\put(-13,9){\line(1,1){3}}
\put(-13,9){\line(2,-1){6}}
\put(-13,9){\line(0,-1){6}}
\put(-13,6){\circle*{.5}}
\put(-13,6){\line(2,1){6}}
\put(-13,6){\line(2,-1){6}}
\put(-13,3){\circle*{.5}}
\put(-13,3){\line(1,-1){3}}
\put(-13,3){\line(2,1){6}}
\put(-10,0){\circle*{.5}}
\put(-10,0){\line(1,1){3}}
\put(-7,3){\line(0,1){6}}
\put(-7,9){\circle*{.5}}
\put(-7,6){\circle*{.5}}
\put(-7,3){\circle*{.5}}
\put(-12,13){\makebox{$\scriptstyle(-1,-2)$}}
\put(-11.5,-2){\makebox{$\scriptstyle(1,2)$}}
%\put(25,.5){\circle*{.5}}
%\put(25,.5){\circle{1}}
%\put(26,0){\makebox{= antidominant $\lambda$}}
\end{picture}
\end{center}

\begin{center}
Type $\widetilde C_2$
\end{center}

Note the asymmetry in the diagram. No global symmetry is expected,
because there are no automorphisms of the affine Dynkin diagram fixing
the special node $s_0$. However, the range of the diagram is too low to
show all the patterns present. In particular the smallest generic
$\lambda$ is the dominant class $(1,2)$. It lies at the bottom of the
first generic orbit, as shown in the small diagram on the left.

On the other hand, the only non-trivial palindromic Schubert varieties
are the five visible in the diagram, each of which happens to be a
chain.

\newpage

%***Diagra-1m G2 here***

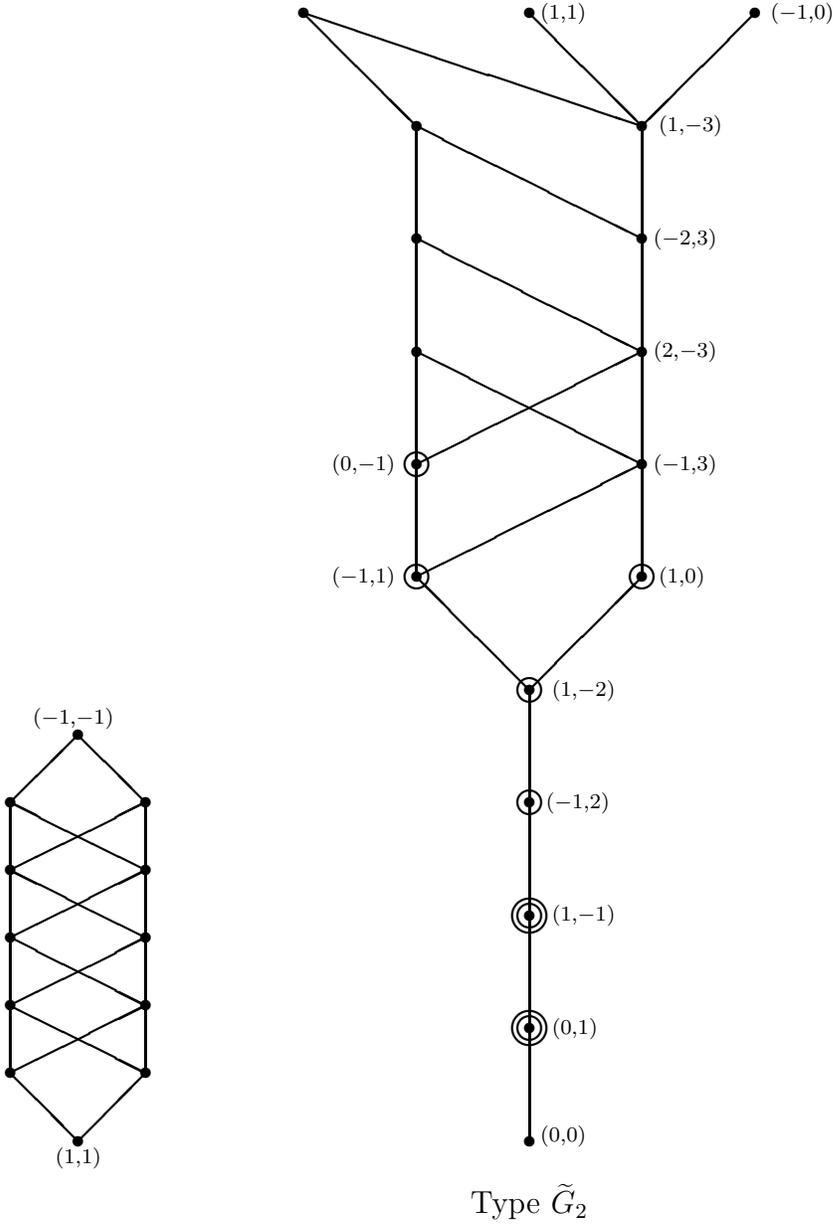
\begin{center}
\setlength{\unitlength}{.3cm}
\begin{picture}(20,50)(0,0)
\put(10,0){\line(0,1){20}}
\put(10,0){\circle*{.5}}
\put(10.5,0){\makebox{$\scriptstyle(0,0)$}}
\put(10,20){\line(1,1){5}}
\put(10,5){\circle*{.5}}
\put(10,5){\circle{1}}
\put(10,5){\circle{1.5}}
\put(11,4.75){\makebox{$\scriptstyle(0,1)$}}
\put(10,10){\circle*{.5}}
\put(10,10){\circle{1}}
\put(10,10){\circle{1.5}}
\put(11,9.75){\makebox{$\scriptstyle(1,-1)$}}
\put(10,15){\circle*{.5}}
\put(10,15){\circle{1}}
\put(10.75,14.75){\makebox{$\scriptstyle(-1,2)$}}
\put(10,20){\circle*{.5}}
\put(10,20){\circle{1}}
\put(11,19.75){\makebox{$\scriptstyle(1,-2)$}}
\put(10,20){\line(-1,1){5}}
\put(5,25){\line(2,1){10}}
\put(5,25){\circle{1}}
\put(1.25,24.75){\makebox{$\scriptstyle(-1,1)$}}
\put(5,30){\line(2,1){10}}
\put(5,30){\circle*{.5}}
\put(5,30){\circle{1}}
\put(1.25,29.75){\makebox{$\scriptstyle(0,-1)$}}
\put(5,35){\line(2,-1){10}}
\put(5,35){\circle*{.5}}
\put(5,40){\line(2,-1){10}}
\put(5,40){\circle*{.5}}
\put(5,45){\line(2,-1){10}}
\put(5,45){\circle*{.5}}
\put(0,50){\line(1,-1){5}}
\put(0,50){\circle*{.5}}
\put(0,50){\line(3,-1){15}}
\put(10,50){\line(1,-1){5}}
\put(10,50){\circle*{.5}}
\put(10.5,49.75){\makebox{$\scriptstyle(1,1)$}}
\put(20,50){\line(-1,-1){5}}
\put(20,50){\circle*{.5}}
%\put(20,50){\circle{1}}
\put(20.7,49.75){\makebox{$\scriptstyle(-1,0)$}}
\put(5,25){\line(0,1){20}}
\put(5,25){\circle*{.5}}
\put(15,25){\line(0,1){20}}
\put(15,25){\circle*{.5}}
\put(15,25){\circle{1}}
\put(15.75,24.75){\makebox{$\scriptstyle(1,0)$}}
\put(15,45){\circle*{.5}}
\put(15.75,44.75){\makebox{$\scriptstyle(1,-3)$}}
\put(15,40){\circle*{.5}}
\put(15.5,39.75){\makebox{$\scriptstyle(-2,3)$}}
\put(15,35){\circle*{.5}}
\put(15.5,34.75){\makebox{$\scriptstyle(2,-3)$}}
\put(15,30){\circle*{.5}}
\put(15.5,29.75){\makebox{$\scriptstyle(-1,3)$}}
%\put(-10,18){\circle*{.5}}
\put(-10,18){\line(1,-1){3}}
\put(-13,15){\circle*{.5}}
\put(-13,15){\line(1,1){3}}
\put(-13,12){\line(2,1){6}}
\put(-7,12){\line(-2,1){6}}
\put(-13,9){\line(2,1){6}}
\put(-13,9){\line(2,-1){6}}
\put(-13,9){\circle*{.5}}
\put(-13,12){\circle*{.5}}
\put(-7,15){\circle*{.5}}
\put(-12,18.5){\makebox{$\scriptstyle(-1,-1)$}}
\put(-10,18){\circle*{.5}}
\put(-11,-1){\makebox{$\scriptstyle(1,1)$}}
\put(-13,15){\line(0,-1){12}}
\put(-13,6){\circle*{.5}}
\put(-13,6){\line(2,1){6}}
\put(-13,6){\line(2,-1){6}}
\put(-13,3){\circle*{.5}}
\put(-13,3){\line(1,-1){3}}
\put(-13,3){\line(2,1){6}}
\put(-10,0){\circle*{.5}}
\put(-10,0){\line(1,1){3}}
\put(-7,3){\line(0,1){12}}
\put(-7,9){\circle*{.5}}
\put(-7,9){\line(-2,1){6}}
\put(-7,6){\circle*{.5}}
\put(-7,3){\circle*{.5}}
\put(-7,12){\circle*{.5}}
%\put(25,.5){\circle*{.5}}
%\put(25,.5){\circle{1}}
%\put(26,0){\makebox{= antidominant $\lambda$}}
\end{picture}
\end{center}

\begin{center}
Type $\widetilde G_2$
\end{center}

The range of the diagram is too low to show all the patterns. The lowest
generic orbit begins with the dominant class $\lambda =(1,1)$ as shown
in the diagram on the left.  However, the only non-trivial palindromic
Schubert varieties are the seven visible in the diagram, each of which
happens to be a chain.

\newpage

%\subsection{Dynkin diagrams}

%\begin{figure}\label{f:dynkin}
%\caption{Numbering of nodes on Dynkin diagrams}

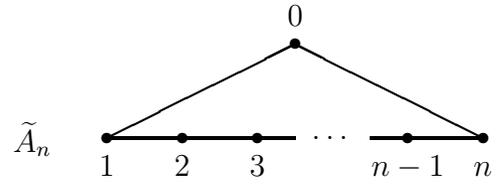
\setlength{\unitlength}{1cm}
\thicklines
$$\label{f:dynkin}
\begin{array}{lll}
\\ \\ \vspace{.3in}
\widetilde{A}_n && 
\raisebox{1ex}{\begin{picture}(5,0)
\mp(0,0)(1,0){3}{\ci}
\put(0,0){\num{1}}\put(1,0){\num{2}}\put(2,0){\num{3}}
\put(3,0){\makebox(0,0){\ldots}}
\mp(5,0)(-1,0){2}{\ci}
\put(4,0){\num{n-1}}\put(5,0){\num{n}}
\put(0,0){\line(1,0){2.5}}
\put(5,0){\line(-1,0){1.5}}
\put(2.5,2){\num{0}}
\put(0,0){\line(2,1){2.5}}
\put(5,0){\line(-2,1){2.5}}
\put(2.5,1.25){\ci}
\end{picture}}
\\ \\ \vspace{.3in}
\widetilde{B}_n && 
\raisebox{1ex}
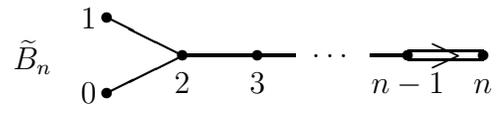
{\begin{picture}(5,0)
\mp(1,0)(1,0){2}{\ci}
\put(1,0){\num{2}}\put(2,0){\num{3}}
\put(3,0){\makebox(0,0){\ldots}}
\mp(5,0)(-1,0){2}{\ci}
\put(4,0){\num{n-1}}\put(5,0){\num{n}}
\put(1,0){\line(1,0){1.5}}
\put(4,0){\line(-1,0){.5}}
\put(4,.06){\line(1,0){1}}
\put(4,-.06){\line(1,0){1}}
\put(4.5,0){\makebox(0,0){\Large$>$}}
\put(0,.5){\ci}\put(0,-.5){\ci}
\put(0,.5){\makebox(0,0)[l]{\hspace{-.35\unitlength}$1$}}
\put(0,-.5){\makebox(0,0)[l]{\hspace{-.35\unitlength}$0$}}
\put(1,0){\line(-2,1){1}}
\put(1,0){\line(-2,-1){1}}
\end{picture}}
\\ \\ \vspace{.3in}
\widetilde{C}_n && 
\raisebox{1ex}
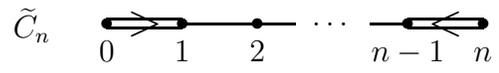
{\begin{picture}(5,0)
\mp(0,0)(1,0){3}{\ci}
\put(1,0){\num{1}}\put(2,0){\num{2}}
\put(3,0){\makebox(0,0){\ldots}}
\mp(5,0)(-1,0){2}{\ci}
\put(4,0){\num{n-1}}\put(5,0){\num{n}}
\put(1,0){\line(1,0){1.5}}
\put(4,0){\line(-1,0){.5}}
\put(4,.06){\line(1,0){1}}
\put(4,-.06){\line(1,0){1}}
\put(4.5,0){\makebox(0,0){\Large$<$}}
\put(.5,0){\makebox(0,0){\Large$>$}}
\put(0,.06){\line(1,0){1}}
\put(0,-.06){\line(1,0){1}}
\put(0,0){\num{0}}
\end{picture}}
\\ \\ \vspace{.3in}
\widetilde{D}_n && 
\raisebox{1ex}
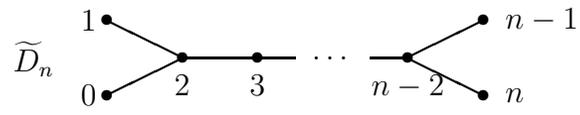
{\begin{picture}(5,.8)
\mp(1,0)(1,0){2}{\ci}
\put(1,0){\num{2}}\put(2,0){\num{3}}
\put(3,0){\makebox(0,0){\ldots}}
\put(4,0){\ci}
\put(4,0){\num{n-2}}
\put(1,0){\line(1,0){1.5}}
\put(4,0){\line(-1,0){.5}}
\put(4,0){\line(2,1){1}}
\put(4,0){\line(2,-1){1}}
\put(5,.5){\ci}\put(5,-.5){\ci}
\put(5,.5){\makebox(0,0)[l]{\hspace{.3\unitlength}$n-1$}}
\put(5,-.5){\makebox(0,0)[l]{\hspace{.3\unitlength}$n$}}
\put(0,.5){\ci}\put(0,-.5){\ci}
\put(0,.5){\makebox(0,0)[l]{\hspace{-.35\unitlength}$1$}}
\put(0,-.5){\makebox(0,0)[l]{\hspace{-.35\unitlength}$0$}}
\put(1,0){\line(-2,1){1}}
\put(1,0){\line(-2,-1){1}}
\end{picture}}
\\ \\ \vspace{.3in}
\widetilde{E}_{6} && 
\raisebox{1ex}
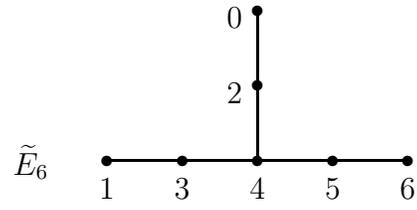
{\begin{picture}(5,2)
\mp(0,0)(1,0){5}{\ci}
\mp(2,1)(0,1){2}{\ci}
\put(0,0){\num{1}}\put(1,0){\num{3}}\put(2,0){\num{4}}\put(3,0){\num{5}}
\put(4,0){\num{6}}
\put(2,.9){\makebox(0,0)[r]{$2$\hspace{.2\unitlength}}}
\put(2,1.9){\makebox(0,0)[r]{$0$\hspace{.2\unitlength}}}
\put(2,0){\line(0,1){2}}
\put(0,0){\line(1,0){4}}
\end{picture}}
\\ \\ \vspace{.3in}
\widetilde{E}_{7} && 
\raisebox{1ex}
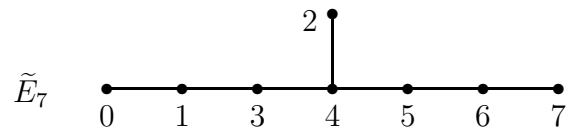
{\begin{picture}(5,1.2)
\mp(0,0)(1,0){7}{\ci}\put(3,1){\ci}
\put(0,0){\num{0}}\put(1,0){\num{1}}\put(2,0){\num{3}}\put(3,0){\num{4}}
\put(4,0){\num{5}}\put(5,0){\num{6}}\put(6,0){\num{7}}
\put(3,.9){\makebox(0,0)[r]{$2$\hspace{.2\unitlength}}}
\put(3,0){\line(0,1){1}}
\put(0,0){\line(1,0){6}}

\end{picture}}
\\ \\ \vspace{.3in}
\widetilde{E}_{8} && 
\raisebox{1ex}
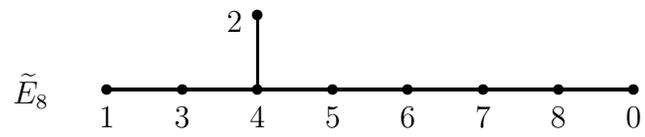
{\begin{picture}(5,1.2)
\mp(0,0)(1,0){8}{\ci}\put(2,1){\ci}
\put(0,0){\num{1}}\put(1,0){\num{3}}\put(2,0){\num{4}}\put(3,0){\num{5}}
\put(4,0){\num{6}}\put(5,0){\num{7}}\put(6,0){\num{8}} \put(7,0){\num{0}}
\put(2,.9){\makebox(0,0)[r]{$2$\hspace{.2\unitlength}}}
\put(2,0){\line(0,1){1}}
\put(0,0){\line(1,0){7}}
\end{picture}}
\\ \\ \vspace{.3in}
\widetilde{F}_4 && 
\raisebox{1ex}
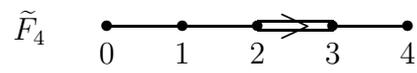
{\begin{picture}(5,0)
\mp(0,0)(1,0){5}{\ci}
\put(0,0){\num{0}}
\put(1,0){\num{1}}\put(2,0){\num{2}}\put(3,0){\num{3}}\put(4,0){\num{4}}
\put(0,0){\line(1,0){2}}
\put(3,0){\line(1,0){1}}
\put(2,.06){\line(1,0){1}}
\put(2,-.06){\line(1,0){1}}
\put(2.5,0){\makebox(0,0){\Large$>$}}
\end{picture}}
\\ \\ \vspace{.3in}
\widetilde{G}_2 && 
\raisebox{1ex}
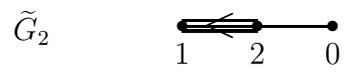
{\begin{picture}(5,0)
\mp(1,0)(1,0){3}{\ci}
\put(3,0){\num{0}}\put(1,0){\num{1}}\put(2,0){\num{2}}
\put(1,0){\line(1,0){2}}
\put(1,.08){\line(1,0){1}}
\put(1,-.06){\line(1,0){1}}
\put(1.5,0.01){\makebox(0,0){\Large$<$}}
\end{picture}}
\\
\end{array}
$$
%\end{figure}

\newpage

{\it Table of coefficients} $m_i$:

\bigskip

$A_n$: $m_i=1$ for all $i$. 

\bigskip

$B_n$: $m_1=1$; $m_i=2$ otherwise. 

\bigskip

$C_n$: $m_n=1$; $m_i=2$ otherwise. 

\bigskip

$D_n$: $m_i=1$ for $i=1,n-1,n$; $m_i=2$ otherwise. 

\bigskip

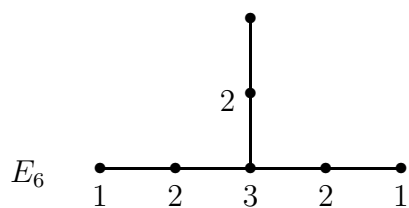
\setlength{\unitlength}{1cm}
\thicklines
$$
\begin{array}{lll}
\\ \\ \vspace{.3in}
E_{6} && 
\raisebox{1ex}{\begin{picture}(5,2)
\mp(0,0)(1,0){5}{\ci}
\mp(2,1)(0,1){2}{\ci}
\put(0,0){\num{1}}\put(1,0){\num{2}}\put(2,0){\num{3}}\put(3,0){\num{2}}
\put(4,0){\num{1}}
\put(2,.9){\makebox(0,0)[r]{$2$\hspace{.2\unitlength}}}
\put(2,1.9){\makebox(0,0)[r]{$$\hspace{.2\unitlength}}}
\put(2,0){\line(0,1){2}}
\put(0,0){\line(1,0){4}}
\end{picture}}
\\ \\ \vspace{.3in}
E_{7} && 
\raisebox{1ex}
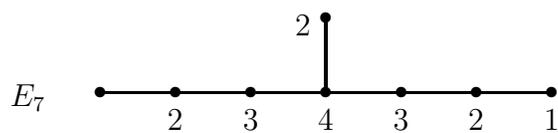
{\begin{picture}(5,1.2)
\mp(0,0)(1,0){7}{\ci}\put(3,1){\ci}
\put(0,0){\num{}}\put(1,0){\num{2}}\put(2,0){\num{3}}\put(3,0){\num{4}}
\put(4,0){\num{3}}\put(5,0){\num{2}}\put(6,0){\num{1}}
\put(3,.9){\makebox(0,0)[r]{$2$\hspace{.2\unitlength}}}
\put(3,0){\line(0,1){1}}
\put(0,0){\line(1,0){6}}

\end{picture}}
\\ \\ \vspace{.3in}
E_{8} && 
\raisebox{1ex}
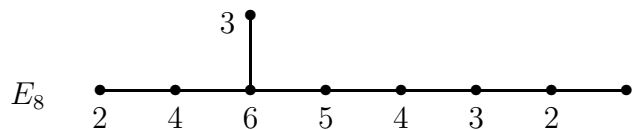
{\begin{picture}(5,1.2)
\mp(0,0)(1,0){8}{\ci}\put(2,1){\ci}
\put(0,0){\num{2}}\put(1,0){\num{4}}\put(2,0){\num{6}}\put(3,0){\num{5}}
\put(4,0){\num{4}}\put(5,0){\num{3}}\put(6,0){\num{2}} \put(7,0){\num{}}
\put(2,.9){\makebox(0,0)[r]{$3$\hspace{.2\unitlength}}}
\put(2,0){\line(0,1){1}}
\put(0,0){\line(1,0){7}}
\end{picture}}
\\ \\ \vspace{.3in}
F_4 && 
\raisebox{1ex}
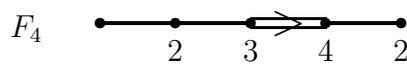
{\begin{picture}(5,0)
\mp(0,0)(1,0){5}{\ci}
\put(0,0){\num{}}
\put(1,0){\num{2}}\put(2,0){\num{3}}\put(3,0){\num{4}}\put(4,0){\num{2}}
\put(0,0){\line(1,0){2}}
\put(3,0){\line(1,0){1}}
\put(2,.06){\line(1,0){1}}
\put(2,-.06){\line(1,0){1}}
\put(2.5,0){\makebox(0,0){\Large$>$}}
\end{picture}}
\\ \\ \vspace{.3in}
G_2 && 
\raisebox{1ex}
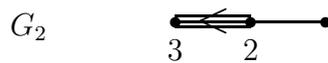
{\begin{picture}(5,0)
\mp(1,0)(1,0){3}{\ci}
\put(3,0){\num{}}\put(1,0){\num{3}}\put(2,0){\num{2}}
\put(1,0){\line(1,0){2}}
\put(1,.08){\line(1,0){1}}
\put(1,-.06){\line(1,0){1}}
\put(1.5,0.01){\makebox(0,0){\Large$<$}}
\end{picture}}
\\
\end{array}
$$
%\end{figure}

\bigskip
\newpage

\section{Acknowledgments:} The authors would like to thank Shrawan
Kumar, Thomas Lam, Mark Shimozono, Pete Littig, and Tricia Hersh for
helpful conversations. We thank Mary Sheetz for preparing the Hasse
diagrams.

\noindent Address: Department of Mathematics, University of Washington, Seattle, WA 98195
\\
Email: billey@math.washington.edu, mitchell@math.washington.edu

\end{document}